\documentclass[a4paper, 12pt]{article}

\usepackage[top = 2.5cm, bottom = 2.5cm, left = 2.4cm, right = 2.4cm]{geometry}
\usepackage{amsfonts, amsmath, amsthm, amssymb, mathtools, cases, bbm}
\usepackage{lmodern, indentfirst, setspace, changepage, fancyhdr}
\usepackage{tikz, pgfplots, float, subcaption}
\usepackage[title]{appendix}
\usepackage[english]{babel}
\usepackage{cite}

\usepackage{hyperref}
\hypersetup{
	colorlinks = true,
	linkcolor = blue,
	citecolor = blue,
}

\theoremstyle{definition}
\newtheorem{theorem}{Theorem}
\newtheorem{proposition}[theorem]{Proposition}
\newtheorem{definition}[theorem]{Definition}
\newtheorem{corollary}[theorem]{Corollary}

\newtheorem{lemma}[theorem]{Lemma}

\newtheorem{remark}{Remark}

\newcommand{\ind}[1]{\mathbbm{1}_{\left\{#1\right\}}}

\newcommand{\norm}[1]{\left|\left|#1\right|\right|}
\newcommand{\floor}[1]{\left\lfloor#1\right\rfloor}
\newcommand{\ceil}[1]{\left\lceil#1\right\rceil}
\newcommand{\map}[3]{#1 : #2 \longrightarrow #3}
\newcommand{\set}[2]{\left\{#1 : #2\right\}}

\newcommand{\indc}[1]{\mathbbm{1}_{#1}}

\newcommand{\defeq}{\vcentcolon=}
\newcommand{\eqdef}{=\vcentcolon}
\newcommand{\rhomin}{\rho_{\min}}
\newcommand{\rhomax}{\rho_{\max}}
\newcommand{\prob}{\mathbbm{P}}
\newcommand{\calA}{\mathcal{A}}

\newcommand{\calD}{\mathcal{D}}

\newcommand{\calF}{\mathcal{F}}

\newcommand{\calJ}{\mathcal{J}}
\newcommand{\calK}{\mathcal{K}}

\newcommand{\calN}{\mathcal{N}}
\newcommand{\calP}{\mathcal{P}}

\newcommand{\scdot}{{}\cdot{}}
\newcommand{\bd}{\mathrm{bd}}
\newcommand{\eq}{\mathrm{eq}}
\newcommand{\N}{\mathbbm{N}}

\newcommand{\R}{\mathbbm{R}}
\newcommand{\Z}{\mathbbm{Z}}

\newcommand{\e}{\mathrm{e}}

\usetikzlibrary{automata, arrows, positioning, calc, external, babel}
\tikzstyle{result} = [rectangle, rounded corners, minimum width = 3.6cm, text width = 3.6cm, minimum height = 1.5cm, text centered, draw = black, thick]
\tikzstyle{main} = [rectangle, rounded corners, minimum width = 3.6cm, text width = 3.6cm, minimum height = 1.5cm, text centered, draw = black, very thick]
\tikzstyle{arrow} = [thick, ->, > = stealth]
\pgfplotsset{
	compat = 1.16,
	every axis/.append style = {
		grid style = {dashed, gray, opacity = 0.2},
		label style = {font = \footnotesize},
		tick label style = {font = \footnotesize},  
		width = 1 * \columnwidth,
		height = 0.618 * 1 * \columnwidth
	}
}

\definecolor{britishracinggreen}{rgb}{0.0, 0.26, 0.15}
\definecolor{bostonuniversityred}{rgb}{0.8, 0.0, 0.0}
\definecolor{ceruleanblue}{rgb}{0.16, 0.32, 0.75}
\definecolor{airforceblue}{rgb}{0.36, 0.54, 0.66}
\definecolor{cadmiumgreen}{rgb}{0.0, 0.42, 0.24}
\definecolor{ao(english)}{rgb}{0.0, 0.5, 0.0}
\definecolor{coolblack}{rgb}{0.0, 0.18, 0.39}
\definecolor{alizarin}{rgb}{0.82, 0.1, 0.26}
\definecolor{arsenic}{rgb}{0.23, 0.27, 0.29}
\definecolor{cobalt}{rgb}{0.0, 0.28, 0.67}
\definecolor{amber}{rgb}{1.0, 0.75, 0.0}

\title{Learning and balancing unknown loads \\ in large-scale systems \vspace{\baselineskip}}

\author{
\normalsize{Diego Goldsztajn, Sem C. Borst}\\ \footnotesize{Eindhoven University of Technology, d.e.goldsztajn@tue.nl, s.c.borst@tue.nl} \\
\normalsize{Johan S.H. van Leeuwaarden}\\ \footnotesize{Tilburg University, j.s.h.vanleeuwaarden@uvt.nl} \\
}

\date{\vspace{\baselineskip} \normalsize{April 5, 2024}}

\begin{document}

%% Title	
\maketitle

\noindent\rule{\textwidth}{1pt}

\vspace{2\baselineskip}

\onehalfspacing

\begin{adjustwidth}{0.7cm}{0.7cm}
	\begin{center}
		\textbf{Abstract}
	\end{center}
	
	\vspace{0.3\baselineskip}
	
	\noindent Consider a system of identical server pools where tasks with exponentially distributed service times arrive as a time-inhomogenenous Poisson process. An admission threshold is used in an inner control loop to assign incoming tasks to server pools while, in an outer control loop, a learning scheme adjusts this threshold over time to keep it aligned with the unknown offered load of the system. In a many-server regime, we prove that the learning scheme reaches an equilibrium along intervals of time where the normalized offered load per server pool is suitably bounded, and that this results in a balanced distribution of the load. Furthermore, we establish a similar result when tasks with Coxian distributed service times arrive at a constant rate and the threshold is adjusted using only the total number of tasks in the system. The novel proof technique developed in this paper, which differs from a traditional fluid limit analysis, allows to handle rapid variations of the first learning scheme, triggered by excursions of the occupancy process that have vanishing size. Moreover, our approach allows to characterize the asymptotic behavior of the system with Coxian distributed service times without relying on a fluid limit of a detailed state descriptor.
	
	\vspace{\baselineskip}
	
	\small{\noindent \textit{Key words:} many-server asymptotics, time-varying arrival rates, phase-type service times.}
	
	\vspace{0.3\baselineskip}
	
	\small{\noindent \textit{Acknowledgment:} supported by Netherlands Organisation for Scientific Research (NWO) through Gravitation-grant NETWORKS-024.002.003 and Vici grant 202.068.} 
\end{adjustwidth}

\section{Introduction}
\label{sec: introduction}

Consider a service system where tasks are instantaneously dispatched to one out of $n$ identical server pools. All the tasks sharing a server pool are executed in parallel and task durations do not depend on the number of tasks contending for service. Nevertheless, the portion of shared resources available to individual tasks does depend on the number of tasks sharing a server pool and determines the quality of service perceived by tasks. These features are characteristic of adaptive streaming and online gaming services, where task durations are mainly determined by the application but the quality of service (e.g., in terms of video resolution) degrades as the degree of contention rises, creating an incentive to balance the load.

In order to maintain the number of tasks balanced across the server pools, the system must judiciously dispatch the incoming tasks. It was established in \cite{menich1991optimality,sparaggis1993extremal} that Join the Shortest Queue (JSQ) is the optimal non-anticipating load balancing policy under rather general conditions. More precisely, consider a system of parallel single-server queues where the sequence of arrival times is general and the service times are exponential with a rate that depends on the number of tasks in the queue in a concave and non-decreasing manner. Then the vector describing the number of tasks at each queue remains the \emph{most balanced} when JSQ is used, in a stochastic majorization sense and at all times; see \cite{marshall1979inequalities} for precise definitions. The result of \cite{menich1991optimality,sparaggis1993extremal} has been applied in the classic load balancing literature to the single-server supermarket model, where it in particular implies that JSQ stochastically minimizes the total number of tasks in the system over time under the above assumptions. Thus, JSQ also minimizes the total number of tasks in stationarity, and hence the mean delay of tasks by virtue of Little's law.

The mean delay of tasks is totally meaningless in the infinite-server setup that we consider since the sojourn times of tasks are independent of the load balancing policy; note that the total number of tasks at any given time is thus also independent of how tasks are dispatched. More pertinent performance metrics are utility functions that measure the quality of service provided to tasks. Specifically, let $x_i$ denote the number of tasks sharing server pool $i$ at some given time, and suppose that the quality of service provided to a task depends on the fraction $1 / x_i$ of resources received through a concave and non-decreasing function $\varphi$. The aggregate quality of service provided to all tasks is $\phi(x) \defeq \sum_{i = 1}^n x_i \varphi(1 / x_i)$, which is a Schur-concave function of $x \defeq (x_1, \dots, x_n)$. The form of $\phi$ and the choice of $\varphi$ subsume the family of $\alpha$-fair utility functions, which have been a central paradigm in network utility maximization problems and rate control policies for bandwidth-sharing networks \cite{mo2000fair,srikant2013communication}. They also cover a wide range of concave utility functions that have been specifically proposed as proxies for Quality-of-Experience (QoE) metrics in the context of video streaming applications \cite{hu2012qoe,joseph2015optimal}. The result of \cite{menich1991optimality,sparaggis1993extremal} applies to the infinite-server setting that we consider when task durations are exponential, and implies that JSQ stochastically maximizes any Schur-concave utility $\phi$ over time as these functions are monotone with respect to majorization.

Despite the above optimality properties, JSQ is not prevalent in practice because it requires knowledge of the exact number of tasks at each server or server pool, and maintaining such detailed state information can be problematic in large-scale systems, as noted in \cite{lu2011join}. Motivated by the latter issue, we consider a threshold-based dispatching rule that involves only limited state information, and we establish that it can maintain a balanced distribution of the load in a limiting regime, even when the arrival rate of tasks is time-varying or the task durations are non-exponential. This dispatching rule also involves knowledge of the offered load; we assume that tasks with mean duration $1 / \mu$ arrive at rate $n \lambda$ and we define the offered load of the system as $n \rho$, where $\rho \defeq \lambda / \mu$. The knowledge of the offered load has to be obtained in an online manner to set the threshold optimally, since $\rho$ is typically not exactly known or even time-varying.

More precisely, we examine the load balancing policies that result from blending the latter dispatching rule with one of two learning schemes for dynamically tracking the offered load. Namely, a \emph{basic scheme}, which keeps track of the total number of tasks in the system and uses this quantity to estimate the offered load, and a \emph{refined scheme}, which uses more detailed state information. An orthogonal dimension of the problem, which is also addressed in this paper, concerns the underlying traffic scenario; i.e., the distribution of the service times and the characteristics of the arrival process. We focus on two settings: first we assume that tasks with exponentially distributed service times arrive as a time-inhomogeneous Poisson process (i.e., $\lambda$ and $\rho$ are time-varying), and that the refined learning scheme is used to adjust the threshold over time; then we assume that tasks with Coxian distributed service times arrive at a constant rate, and that the basic learning scheme is used to estimate the offered load. The analysis of the former setting emphasizes the learning aspect of the problem, whereas the latter setting highlights the role of the service time distribution. We also indicate how the results established in this paper can be extended to other possible combinations of the underlying traffic scenario and the implemented learning scheme.

Altogether, the load balancing policies considered in this paper integrate two control loops. In the inner control loop, the dispatching rule aims at maintaining the number of tasks at each of the server pools between two given values; these values are always consecutive multiples of a positive integer $\Delta$, which is a parameter of the policy, and the lowest of these values is referred to as the threshold and is denoted by $\ell_n$. In the outer control loop, the learning scheme adjusts the threshold over time to maintain the normalized offered load $\rho$ between $\ell_n$ and $\ell_n + \Delta$. Therefore, the parameter $\Delta$ governs a tradeoff between the degree of load balance and the stability of the learning scheme. Specifically, as $\Delta$ decreases, the dispatching rule is more stringent in terms of the load balancing objective and the learning scheme is more sensitive to demand variations.

The control actions of the dispatching rule do not influence the basic learning scheme since the total number of tasks in the system does not depend on these actions. In contrast, there is a complex interdependence between the inner and outer control actions when the refined learning scheme is used. Namely, the actions of the refined learning scheme are triggered by changes in the occupancy state of the system, while at the same time, the evolution of the occupancy state is governed by the dispatching rule, which depends on the output of the refined learning scheme.

\subsection{Main contributions}

Our main result for the refined scheme, in the time-inhomogeneous exponential scenario, is that the threshold reaches an equilibrium value in intervals of time where $\rho$ is suitably bounded and that the load is balanced to an extent that depends on $\Delta$. Specifically, assume that there exist $m \in \N$ and an interval $[a, b]$ such that $m \Delta < \rho(t) < (m + 1)\Delta$ if $t \in [a, b]$. We establish that there exists $\sigma$ such that the following properties hold with probability one: $\ell_n(t) = m\Delta$ for all $t \in [a + \sigma, b]$ and all large enough $n$, the supremum over $[a + \sigma, b]$ of the fraction of server pools with less than $m \Delta$ tasks goes to zero with $n$, and the fraction of server pools with more than $(m + 1)\Delta$ tasks is $O(1)\e^{-\mu\left[t - (a + \sigma)\right]}$ for all $t \in [a + \sigma, b]$.

As noted earlier, the total number of tasks $N_n$ is independent of the dispatching rule. Thus, any of the possible distributions of the load at time $t$ majorizes that where all the server pools have either $\floor{N_n(t) / n}$ or $\ceil{N_n(t) / n}$ tasks, and in particular provides a lower aggregate quality of service for any Schur-concave utility function $\phi$. Under the above-stated assumptions, we establish that $nm\Delta < N_n(t) < n(m + 1)\Delta$ for all $t \in [a + \sigma, b]$ and all large enough $n$. If $\Delta = 1$, then this implies that the load tends to be perfectly balanced, because all the server pools have either $\floor{N_n(t)/n}$ or $\ceil{N_n(t) / n}$ tasks, except for a fraction that is at most $o(1) + O(1)\e^{-\mu\left[t - (a + \sigma)\right]}$. Further, we show that the term that does not vanish with $n$ cannot be avoided without weakening our assumptions, not even by JSQ. For example, if the arrival rate of tasks is constant and a fraction of the server pools have more than $\ceil{\rho}$ tasks at time zero, then a fraction of the server pools that is $O(1)\e^{-\mu t}$ has more than $\ceil{\rho}$ tasks at any time $t$, regardless of the load balancing policy. While the load is more coarsely balanced as $\Delta$ increases, the condition $m \Delta < \rho(t) < (m + 1)\Delta$ relaxes.

Analogous results are obtained for the basic scheme in the time-homogeneous Coxian setting. More precisely, we prove that the threshold $\ell_n$ reaches an equilibrium value for all large enough $n$ with probability one, and then the load remains balanced to an extent that depends on $\Delta$ as described above. The load tends to be perfectly balanced if $\Delta = 1$.

In order to prove our main results, we develop a novel methodology which is fundamentally different from a traditional fluid limit analysis. In particular, our approach does not decouple the analysis across the scale parameter and the time variable as in a traditional fluid limit approach, where first a sequence of scaled processes is proven to converge to a solution of a certain system of differential equations and then stability results are established for this limiting dynamical system. Instead, our methodology leverages the tractability of specific system dynamics to identify key dynamical properties of the threshold and occupancy processes, and uses strong approximations to prove that these dynamical properties hold asymptotically as $n$ grows large; these asymptotic dynamical properties are then employed to establish our main results.

An important motivation for our novel approach is the nature of the refined learning scheme, which allows for rapid changes of the threshold, triggered by small excursions of the occupancy process. Specifically, these excursions trigger changes of the threshold if their size exceeds a parameter of the learning scheme which vanishes on the fluid scale as $n$ grows large, making a traditional fluid limit analysis especially challenging, if not impossible. Furthermore, our approach allows to derive our main result for the Coxian setting in a direct way which avoids proving a fluid limit for a detailed state descriptor of the system; if tasks can undergo at most $r$ phases, then such a state descriptor could be given by a multidimensional sequence in which the element $(j_1, \dots, j_r)$ indicates the fraction of server pools with exactly $j_m$ tasks in phase $m$.

The analysis of the refined learning scheme in the time-inhomogeneous exponential scenario is more challenging from a learning perspective, whereas the analysis of the basic learning scheme in the time-homogeneous Coxian scenario involves a higher complexity in terms of the service time distribution since the residual service time distribution of the tasks in the system is heterogeneous. Nevertheless, we demonstrate that our methodology can be applied in both of these substantially distinct settings. Moreover, we believe that the principle underlying this methodology, of leveraging process-level dynamical properties through strong approximations, is of broader applicability, and can be particularly useful in situations where fluid limits are unavailable or difficult to derive.

\subsection{Related work}
\label{sub: related work}

The problem of balancing the load in parallel-server systems has received immense attention in the past few decades; a recent survey is provided in \cite{van2018scalable}. While traditionally the focus in this literature used to be on performance, more recently the implementation overhead has emerged as an equally important issue. This overhead has two sources: the communication burden of exchanging messages between the dispatcher and the servers, and the cost in large-scale deployments of storing and managing state information at the dispatcher \cite{gamarnik2018delay,gamarnik2020lower}.

As noted earlier, the present paper concerns an \emph{infinite-server} setting in the sense that task durations do not depend on the number of competing tasks. However, the level of contention does affect the share of the available resources received by each of the tasks, e.g., in terms of bandwidth for adaptive streaming applications. Infinite-server models of this kind have been commonly adopted in the literature as a natural paradigm for describing the dynamics and evaluating the performance of streaming sessions on flow-level; e.g., \cite{benameur2002quality,key2004fair}.

In contrast, the load balancing literature has been predominantly concerned with single-server models, where performance is measured in terms of queueing delays. In the latter scenario, JSQ minimizes the mean delay for exponential service times \cite{ephremides1980simple,winston1977optimality}, but a naive implementation of JSQ involves an excessive communication burden for large-scale systems. The so-called power-of-$d$ strategies \cite{mitzenmacher2001power,mukherjee2016universality,vvedenskaya1996queueing} involve substantially less communication overhead and yet provide significant improvements in delay performance over purely random routing. Another alternative are pull-based policies~\cite{badonnel2008dynamic,stolyar2015pull}, which reduce the communication burden by maintaining state information at the dispatcher. Particularly, the Join the Idle Queue (JIQ) policy \cite{lu2011join,stolyar2015pull} asymptotically matches the optimality of JSQ and involves only one message per task, by keeping a list of idle queues.

With the exception of \cite{stolyar2015pull}, the papers listed above only consider exponentially distributed service times, and their results have only been extended partially or under stringent assumptions to more general service times. In particular, \cite{stolyar2015pull,stolyar2017pull} treat a JIQ system where service times are assumed to have a decreasing hazard rate function, whereas \cite{foss2017large} considers a JIQ system with general service times but only under the assumption that the load is strictly below~$1/2$. Also, power-of-$d$ policies with general service time distributions were studied in \cite{bramson2010randomized,bramson2012asymptotic}. Relying on a \emph{propagation of chaos} ansatz, \cite{bramson2010randomized} establishes \emph{power-of-choice} benefits similar to those reported in \cite{mitzenmacher2001power,vvedenskaya1996queueing} for exponential service times; the ansatz is proved in \cite{bramson2012asymptotic} when the service time distribution has decreasing hazard rate or the arrival rate is sufficiently small. Power-of-$d$ schemes were also studied in \cite{aghajani2018pde,aghajani2019hydrodynamic} through measure-valued processes, proving that a suitably scaled sequence of processes converges to a hydrodynamic limit. More closely related to the Coxian setting in this paper is \cite{vasantam2017mean}, which examines mean-field limits for power-of-$d$ policies in loss systems with mixed-Erlang service time distributions.

As mentioned above, the infinite-server setting considered in this paper has received only limited attention in the load balancing literature, even for exponentially distributed service times. While queueing delays are hardly meaningful in this setting, load balancing still plays a crucial role in optimizing different performance measures. Besides loss probabilities in finite-capacity scenarios, other relevant performance measures are Schur-concave utility metrics associated with quality of service in streaming applications, such as the above-mentioned function $\phi$. As noted earlier, it was established in \cite{menich1991optimality,sparaggis1993extremal} that JSQ is the optimal non-anticipating load balancing policy, with respect to stochastic majorization, when service times are exponential. It was proven in \cite{mukherjee2020asymptotic} that the performance of JSQ can be asymptotically matched on the diffusion scale by certain power-of-$d$ strategies which considerably reduce the associated communication overhead.

Just like the dispatching algorithms studied in \cite{mukherjee2020asymptotic}, the policies considered in this paper aim at optimizing the overall quality of service experienced by tasks. The refined version of our policy was considered in \cite{goldsztajn2021self}, assuming exponentially distributed service times and a constant arrival rate of tasks. Resorting to a traditional fluid limit analysis, \cite{goldsztajn2021self} proves fluid-scale optimality provided that a parameter of the learning scheme satisfies a certain condition, stated in terms of the unknown offered load per server pool $\rho$. The present paper establishes how this parameter of the refined learning scheme should scale with the number of server pools, and strengthens the latter result by proving, through a novel methodology, that fluid-scale optimality is always attained with the proposed parameter scaling, even if the arrival rate of tasks is time-varying.

As alluded to above, one of the most interesting features of our policies is their capability of adapting the threshold to unknown time-varying loads. The problem of adaptation to uncertain demand patterns was addressed in the context of single-server queueing models in \cite{goldsztajn2021automatic,goldsztajn2018controlling,mukherjee2017optimal,mukherjee2019join}, which assume that the number of servers can be \emph{right-sized} on the fly to match the load. However, in all these papers the dispatching algorithm remains the same at all times since the right-sizing mechanism alone is sufficient to maintain small queueing delays when the demand for service changes. Different from these right-sizing mechanisms, the learning schemes considered in this paper modify the dispatching rule of the system over time to maintain a balanced distribution of the load while the number of server pools remains constant. 

\subsection{Outline of the paper}

In Section \ref{sec: model description} we specify the dispatching rule, the learning schemes and the traffic scenarios studied in this paper. The notation used in the paper and some technical assumptions are introduced in Section \ref{sec: notation and assumptions}, and our main results are stated in Section~\ref{sec: main results}, where we also outline our methodology and discuss some extensions. The proof of our main result for the refined learning scheme, in the time-inhomogenenous exponential scenario, is carried out in Section~\ref{sec: refined learning scheme}, and the proof of our main result for Coxian distributed service times, in the time-homogeneous scenario with the basic learning scheme, is provided in Section \ref{sec: coxian service times}. The proofs of some intermediate results are provided in Sections \ref{app: relative compactness} and \ref{app: auxiliary results} of the appendix.

\section{Model description}
\label{sec: model description}

Consider a system of $n$ identical server pools. Each server pool has infinitely many servers, tasks arrive as a possibly time-varying Poisson process of intensity $n \lambda$ and service times are independent and identically distributed with mean $1 / \mu$. In order to maximize the overall quality of service provided to tasks, it is necessary to maintain an even distribution of the load. Specifically, if the total number of tasks at a given time is $N_n$, then any Schur-concave utility function, as defined in Section \ref{sec: introduction}, is maximized if all server pools have either $\floor{N_n / n}$ or $\ceil{N_n / n}$ tasks at that time; here $\floor{\scdot}$ and $\ceil{\scdot}$ denote the floor and ceiling functions, respectively.

\subsection{Dispatching rule}
\label{sub: dispatching rule}

Recall that if service times are exponential, then JSQ achieves the highest degree of load balance among the non-anticipating policies. Specifically, given two systems with the same initial condition, one of which uses JSQ, it was proven in \cite{sparaggis1993extremal} that it is possible to couple the arrival and service completion times so that the load is \emph{more balanced} in the JSQ system, with respect to majorization, at all times. As noted in Section \ref{sec: introduction}, the result obtained in \cite{sparaggis1993extremal} covers any sequence of arrival times and implies that any Schur-concave utility is stochastically maximized by JSQ at any given time when the service times are exponential. However, JSQ requires complete state information about the total number of tasks at each server pool, making it challenging to implement in large systems, as further discussed in Section \ref{sub: implementation overheads}.

Suppose briefly that $\lambda$ is constant over time and recall that $\rho \defeq \lambda / \mu$ is the offered load per server pool; for brevity we often use the shorter term offered load. The total number of tasks in steady state is Poisson distributed with mean $n\rho$ regardless of the load balancing policy, because all server pools together form an infinite-server system. Ideally, all the server pools should have at least $\floor{\rho}$ tasks and no more than $\ceil{\rho}$ tasks. A natural way to achieve this without requiring full state information is by means of a threshold-based dispatching rule. Namely, we consider a dispatching algorithm that uses an integral parameter $\Delta \geq 1$ and a threshold $\ell_n \in \N$ to distribute the incoming tasks among the $n$ server pools. This algorithm aims at maintaining the number of tasks at each server pool between $\ell_n$ and $h_n \defeq \ell_n + \Delta$ by dispatching every incoming task as follows.
\begin{itemize}
	\item If some server pool has strictly less than $\ell_n$ tasks right before the arrival, then the new task is sent to a server pool having strictly less than $\ell_n$ tasks, chosen uniformly at random.
	
	\item If all server pools have at least $\ell_n$ tasks right before the arrival, and some server pool has strictly less than $h_n$ tasks, then the new task is sent to a server pool chosen uniformly at random among those having at least $\ell_n$ and strictly less than $h_n$ tasks.
	
	\item If all server pools have at least $h_n$ tasks right before the arrival, then the new task is sent to a server pool having at least $h_n$ tasks, chosen uniformly at random.
\end{itemize}

This threshold-based dispatching rule was analyzed in \cite{goldsztajn2021self} for $\Delta = 1$ and $\ell_n = \floor{\rho}$. While JSQ is the optimal policy, the latter paper proves that the threshold-based dispatching rule has the same fluid and diffusion limits as JSQ when the arrival rate of tasks is constant and service times are exponential; this means that the processes that describe the occupancy state of each policy satisfy the same law of large numbers and central limit theorem as $n \to \infty$, respectively.

\subsubsection{Implementation overhead}
\label{sub: implementation overheads}

There are two natural implementations of JSQ; in both implementations we assume that tasks are assigned to the server pools by a single dispatcher.
\begin{enumerate}
	\item[(a)] When a task arrives, the dispatcher polls the $n$ server pools in order to know the exact number of tasks at each server pool. Then the dispatcher chooses one of the server pools with the least number of tasks, and only afterwards sends the task to this server pool.
	
	\item[(b)] The dispatcher stores the number of tasks at each server pool using a suitable data structure, such as a sorted list or a heap; the latter being more efficient. When a task is dispatched, the dispatcher increases the corresponding entry of the data structure by one, and when a task departs, the server pool that finished the task sends a message to the dispatcher so that the dispatcher may decrease the corresponding entry of the data structure by one. In both cases the list is resorted, or the heap is rearranged, after the update.
\end{enumerate}

The first implementation does not require to store any state information at the dispatcher, but before being assigned to a server pool, every task must wait until the dispatcher has polled the $n$ server pools and has sorted all their responses. This introduces a prohibitive delay in large systems. On the other hand, the second implementation requires to store complete state information about the number of tasks at each server pool, and to maintain this information suitably arranged in a data structure. The complexity of the operations required to update the data structure after each arrival and departure scales with the size of the system. If an update takes longer than the time elapsed until the next task arrives, then this task can be delayed and may even have to be discarded, or sent to a random server pool, to avoid a queue in front of the dispatcher.

The threshold-based dispatching rule admits a token-based implementation that uses two types of tokens: the dispatcher stores a \emph{green} token for each server pool with less than $\ell_n$ tasks and a \emph{yellow} token for each server pool with less than $h_n$ tasks; both tokens are stored if a server pool has less than $\ell_n$ tasks. When a task arrives, the dispatcher selects a token following the dispatching rule and assigns the task to the corresponding server pool; the token is then discarded. On the other hand, each server pool sends an appropriate token to the dispatcher when its occupancy crosses one of the thresholds due to a departure, or when it receives a task and its number of tasks does not cross a threshold, to replace the token discarded by the dispatcher.

The amount of storage capacity needed for the tokens is much smaller than that required by JSQ with implementation (b). Further, the complexities of picking, discarding and adding tokens are low and do not scale with the size of the system. The token-based implementation of the threshold-based dispatching rule and the associated advantages are reminiscent of those offered by JIQ in the single-server setting. As noted in \cite{lu2011join}, another advantage of JIQ over JSQ is that it can be implemented in systems with multiple dispatchers. The token-based implementation of the threshold-based dispatching rule makes it also suitable for systems with multiple dispatchers, but the present paper focuses on systems with a single dispatcher.

\subsection{Learning schemes}
\label{sub: learning schemes}

The dispatching rule described above relies on knowledge of the offered load for selecting the optimal threshold value, the one that yields a balanced distribution of the load as $n$ grows large. In practice, the offered load is typically unknown and may experience changes over time due to demand fluctuations, which introduces the challenge of keeping the threshold aligned with an unknown and possibly time-dependent optimal value.

\subsubsection{Basic learning scheme}

A basic online procedure for learning the optimal threshold value relies on keeping track of the total number of tasks $N_n$ and estimating the offered load as this quantity divided by $n$. The optimal threshold value is then estimated as $\floor{N_n(t) / n}$ for all $t$. If the arrival rate of tasks is constant over time, then the total number of tasks in stationarity is Poisson distributed with mean $n\rho$, so the latter procedure is asymptotically effective by a law of large numbers. Moreover, if the arrival rate of tasks is time-varying, then it is reasonable to expect a similar effectiveness along intervals of time where the arrival rate of tasks is roughly constant.

One caveat is that the threshold determined through this scheme could exhibit rapid fluctuations if the arrival rate of tasks has large enough high-frequency oscillations. Threshold updates introduce communication overhead since all the server pools must be informed of the new threshold and the state information stored at the dispatcher must be suitably updated. Hence, it might be desirable to sacrifice load balance, to a certain extent, to gain stability of the threshold. This compromise is attained by increasing $\Delta$ and setting
\begin{equation*}
	\ell_n(t) = \floor{\frac{N_n(t)}{n}}_\Delta \defeq \max \set{k \Delta}{k \Delta \leq \frac{N_n(t)}{n}\ \text{and}\ k \in \Z} \quad \text{for all} \quad t.
\end{equation*}

\begin{remark}
	\label{rem: on communication overhead}
	The dispatcher must inform all the server pools of every threshold update, and each server pool must reply to generate appropriate tokens. However, the associated overhead is alleviated by three important properties of updates.
	\begin{itemize}
		\item The updates of the refined learning scheme are always triggered by arrivals and are processed after the task is dispatched, as described below. Hence, the dispatcher has one interarrival time to process the update. The updates of the basic learning scheme can also be processed only at arrival times and after dispatching the task; this does not modify our proofs.
		
		\item The dispatcher does not need all the tokens to follow the threshold-based dispatching rule. It is enough that one green token is present if some server pool has less than $\ell_n$ tasks, or that one yellow token is present if all the server pools have at least $\ell_n$ tasks and some server pool has less than $h_n$ tasks. Therefore, it is reasonable to expect that the dispatching rule will be correctly executed after a threshold update even if not all the tokens have been processed.
		
		\item We prove that threshold updates do not occur along intervals of time where the offered load is suitably bounded. Moreover, the updates can be made sparser by increasing $\Delta$.
	\end{itemize}
	Threshold updates could still generate issues if the time required to process the update is larger than one interarrival time; e.g., a task could be incorrectly dispatched due to inaccurate state information. However, these issues may only arise at the times of updates.
\end{remark}

\subsubsection{Refined learning scheme}

The dispatching rule described in Section \ref{sub: dispatching rule} aims at maintaining the number of tasks at each server pool between $\ell_n$ and $h_n$. An alternative learning scheme for finding the optimal threshold value consists of adjusting the threshold in the appropriate direction whenever the objective of the dispatching rule is violated. This more refined scheme is parameterized by $\alpha_n \in (0, 1)$ and uses the same state information as the dispatching rule. In this scheme, the time-dependent threshold takes values on the non-negative multiples of $\Delta$ and is adjusted in steps of $\Delta$ units at certain arrival epochs. Specifically, when an arrival occurs, the refined learning scheme acts under the following circumstances.
\begin{itemize}
	\item If the fraction of server pools having at least $\ell_n$ tasks is smaller than or equal to $\alpha_n$ right before the arrival, then the threshold is decreased $\Delta$ units once the new task has been dispatched.
	
	\item If all the server pools have at least $\ell_n$ tasks and the number of server pools with at least $h_n$ tasks is larger than or equal to $n - 1$ right before the arrival, then the threshold is increased $\Delta$ units right after dispatching the task that has just arrived.
\end{itemize}

\begin{remark}
	\label{rem: well positioned threahold}
	If the number of server pools with at least $h_n$ tasks is strictly less than $n$, then the definition of the learning scheme ensures that this property is preserved at all times. As a result, new tasks are always dispatched to server pools with at most $h_n - 1$ tasks. Depending on the initial condition, the number of server pools with at least $h_n$ tasks is strictly smaller than $n$ from the start or after a certain number of arrivals, for sufficiently large $n$. Specifically, suppose that all server pools have at most $B$ tasks at time zero. If $\ell_n(0)$ is such that all server pools have at least $h_n(0)$ tasks, then $\ell_n$ increases with each arrival until the number of server pools with at least $h_n$ tasks is strictly smaller than $n$. Note that $\ell_n(0) < \floor{B}_\Delta \Delta$ since otherwise the number of server pools with at least $h_n(0)$ tasks at time zero would be equal to zero. After $\floor{B}_\Delta$ arrivals, the number of server pools with at least $i$ tasks is strictly below $n$ for all $i \geq B + 1$ and $n > \floor{B}_\Delta$ even if all tasks are sent to server pools with $B$ tasks as a result of uniformly random routing decisions. Also, $h_n \geq B + 1$ if the threshold is increased $\floor{B}_\Delta$ times. Thus, the number of server pools with at least $h_n$ tasks is strictly smaller than $n$ after at most $\floor{B}_\Delta$ arrivals for all $n > \floor{B}_\Delta$.
\end{remark}

\subsection{Traffic scenarios}
\label{sub: traffic scenarios}

The adaptive load balancing policies that result from the combination of the dispatching rule described in Section \ref{sub: dispatching rule} and any of the two learning schemes introduced in Section \ref{sub: learning schemes} are analyzed in the context of two different traffic scenarios. Specifically, we first assume that tasks with exponentially distributed service times arrive as a time-inhomogeneous Poisson process, and we then assume that tasks with Coxian distributed service times arrive as a time-homogenenous Poisson process. In the former scenario, our analysis focuses on the refined learning scheme, although our results easily extend to the basic learning scheme. In the latter scenario, our analysis focuses on the basic learning scheme, but most of our results extend to the refined learning scheme, as discussed in Section \ref{sub: extensions} in further detail.

Note that the class of Coxian distributions is dense in the set of non-negative distributions with respect to the topology of weak convergence. In particular, any service time distribution can be approximated with an arbitrary precision by a Coxian distribution with respect to any metric that is compatible with this topology (e.g., the L\'evy-Prokhorov metric). The denseness of the class of Coxian distributions was proved in \cite{schassberger2013warteschlangen} and can also be established as in \cite[Exercise 3.3.3]{kelly2011reversibility}.

Looking beyond exponential service times and constant arrival rates of tasks is important for applications, where traffic conditions are typically different. For example, the duration of YouTube videos is not exponentially distributed; the empirical distribution of a large set of YouTube videos is provided in \cite{cheng2008statistics}. Similarly, if we consider Netflix shows (excluding movies), then the duration has a nearly discrete distribution, concentrated at certain standard show lengths, such as 20, 40 and 60 minutes. In addition, the demand for streaming services can exhibit significant variations even over short periods of time, as observed in \cite{applegate2015optimal,niu2011understanding,niu2011demand}.

\section{Notation and assumptions}
\label{sec: notation and assumptions}

We consider a finite interval of time $[0, T]$ and we assume that new tasks arrive at rate $n \lambda$ for some bounded function $\map{\lambda}{[0, T]}{[0, \infty)}$; i.e., there exists $C \geq 0$ such that $\lambda(t) \leq C$ for all $t \in [0, T]$. We denote the instantaneous offered load per server pool by $\rho(t) \defeq \lambda(t) / \mu$. All tasks undergo at most $r$ phases before they leave the system, they always enter the system in the first phase and they complete phase $m$ after an exponentially distributed time of mean $1 / \mu_m$. After completing phase $m$, a task moves to the next phase with probability $p_m$ or leaves with probability $1 - p_m$, where $p_r = 0$. Note that the case $r = 1$ corresponds to exponentially distributed service times and that the mean service time of a task is
\begin{equation*}
	\frac{1}{\mu} = \sum_{i = 1}^r (1 - p_i) \left(\prod_{j = 1}^{i - 1} p_j\right) \sum_{k = 1}^i \frac{1}{\mu_k}.
\end{equation*}
For technical reasons, we assume that $p_m \in (0, 1)$ for all $m \neq r$. The subclass of Coxian distributions that satisfy this property is dense in the class of all Coxian distributions with respect to the topology of weak convergence, hence dense in the set of all non-negative distributions.

The pair formed by the threshold and the vector-valued process which describes the number of tasks in each phase at each server pool constitutes a Markov process. Instead of considering each server pool individually, we adopt an aggregate state description, denoting by $s_n(j_1, \dots, j_r)$ the fraction of server pools with exactly $j_m$ tasks in phase $m$. In view of symmetry, the pair formed by the threshold $\ell_n$ and the occupancy process $s_n \defeq \set{s_n(j_1, \dots, j_r)}{j_1, \dots, j_r \in \N}$ also constitutes a Markov process. Sequences of these Markov processes, indexed by $n$, will be constructed on a common probability space $(\Omega, \calF, \prob)$ in Sections \ref{sub: relative compactness exponential case} and \ref{sub: relative compactness coxian case}, for exponentially distributed service times and general Coxian distributed service times, respectively. We write $s_n(\omega, t, j_1, \dots, j_r)$ for the value of $s_n(j_1, \dots, j_r)$ at $(\omega, t) \in \Omega \times [0, T]$, although $\omega$ and $t$ are often omitted.

The fraction of server pools with at least $i$ tasks, in any phase, is denoted by
\begin{equation}
	\label{eq: brief and detail descriptors}
	q_n(i) \defeq \sum_{j_1 + \dots + j_r \geq i} s_n(j_1, \dots, j_r),
\end{equation}
and we let $q_n \defeq \set{q_n(i)}{i \geq 0}$. If service times are exponentially distributed, then all the tasks in the system are identical, in the sense that they have the same residual service time distribution. In this case, $s_n$ can be recovered from $q_n$, and we use the latter process to describe the occupancy state of the system, which can be represented as in the diagram of Figure \ref{fig: state variables}. The latter diagram is also useful when service times are not exponentially distributed, although it only provides partial information about the occupancy state of the system since it ignores the phases of tasks.

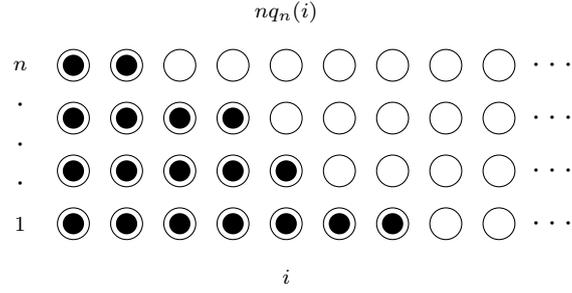
\begin{figure}
	\centering
		\begin{tikzpicture}[x = 0.7cm, y = 0.7cm]
			\foreach \x in {0, 1, 2, ..., 8}
			\foreach \y in {1, 2, ..., 4}
			\draw[black] (\x, \y) circle (0.3);
			
			\node at (4, 0)     {\scriptsize$i$};
			\node at (4, 5)     {\scriptsize$nq_n(i)$};
			\node at (-1, 1)     {\scriptsize$1$};
			\node at (-1, 1.75)  {$\cdot$};
			\node at (-1, 2.5)   {$\cdot$};
			\node at (-1, 3.25)  {$\cdot$};
			\node at (-1, 4)     {\scriptsize$n$};
			
			\foreach \y in {1, 2, ..., 4}
			\fill[black] (0, \y) circle (0.2);
			
			\foreach \y in {1, 2, ..., 4}
			\fill[black] (1, \y) circle (0.2);
			
			\foreach \y in {1, 2, 3}
			\fill[black] (2, \y) circle (0.2);
			
			\foreach \y in {1, 2, 3}
			\fill[black] (3, \y) circle (0.2);
			
			\foreach \y in {1, 2}
			\fill[black] (4, \y) circle (0.2);
			
			\fill[black] (5, 1) circle (0.2);
			
			\fill[black] (6, 1) circle (0.2);
			
			\foreach \y in {1, 2, ..., 4}
			\node at (9, \y)    {$\cdot\cdot\cdot$};
		\end{tikzpicture}
	\caption{Schematic representation of the occupancy state of the system for exponentially distributed service times. White circles represent servers and black circles represent tasks. Each row corresponds to a server pool, and these are arranged so that the number of tasks increases from top to bottom. The number of tasks in column $i$ is $nq_n(i)$.}
	\label{fig: state variables}
\end{figure}

The occupancy process $s_n$ takes values in
\begin{equation*}
	S \defeq \set{s \in [0, 1]^{\N^r}}{\sum_{j_1, \dots, j_r \in \N} s(j_1, \dots, j_r) = 1} \subset \R^{\N^r},
\end{equation*}
while the process $q_n$ takes values in
\begin{equation*}
	Q \defeq \set{q \in [0, 1]^{\N}}{q(i + 1) \leq q(i) \leq q(0) = 1\ \text{for all}\ i \geq 0} \subset \R^\N.
\end{equation*}
Consider the product topology in $\R^{\N^r}$ and assume that there exists $s(0) \in S$ such that
\begin{equation}
	\label{ass: convergence of initial conditions}
	\lim_{n \to \infty} s_n(0) = s(0) \quad \text{almost surely};
\end{equation}
here $s_n(0)$ is the occupancy process evaluated at time zero and $s(0)$ is a constant that describes the limiting occupancy state of the system. In addition, we assume that there exists $B \in \N$ such that $q_n(0, B + 1) = 0$ and $\ell_n(0) \leq B$ for all $n$ with probability one; here we also consider the system at time zero. This technical assumption precludes degenerate sequences of initial conditions, such as initial occupancy states with a (possibly vanishing) fraction of server pools having a number of tasks that approaches infinity with $n$, or initial thresholds which approach infinity with $n$.

Let $G_m$ denote the complementary cumulative distribution function of the residual service time of a task in phase $m$, and consider the function $\map{u}{[0, T]}{[0, \infty)}$ such that
\begin{equation}
	\label{eq: general total number of tasks}
	u^m(0) \defeq \sum_{j_1, \dots, j_r \in \N} j_ms(0, j_1, \dots, j_r), \quad u(t) \defeq \sum_{m = 1}^r u^m(0)G_m(t) + \int_0^t \rho(s) \mu G_1(t - s)ds.
\end{equation}
The quantity $u^m(0)$ may be interpreted as the normalized total number of tasks in phase $m$ for the limiting initial occupancy state $s(0)$. The integral expression corresponds to the functional law of large numbers of an infinite-server system where the arrival rate of tasks scales proportionally to the time-varying function $\lambda$ and the service times have the Coxian distribution specified above. In the case of exponential service times, we may also write
\begin{equation}
	\label{eq: total number of tasks}
	u(0) \defeq \sum_{i = 1}^\infty q(0, i) \quad \text{and} \quad u(t) \defeq u(0)\e^{-\mu t} + \int_0^t \rho(s) \mu \e^{-\mu(t - s)}ds,
\end{equation}
where $q(0)$ is defined in terms of $s(0)$ as in \eqref{eq: brief and detail descriptors}. If $\lambda$ is a regular enough function, then $u$ is differentiable and the integral equation can be stated as $\dot{u}(t) = \lambda(t) - \mu u(t)$, i.e., in differential form.
%We establish in Proposition \ref{prop: total number of tasks} that the total number of tasks normalized by the number of server pools converges almost surely to the function defined by \eqref{eq: total number of tasks} when the sample paths of the occupancy states and thresholds are constructed as specified in Section \ref{sub: coupled construction of sample paths}. The analogous property is established in Section \ref{sub: asymptotic dynamical properties} for more general Coxian distributed service times, in this case using \eqref{eq: general total number of tasks}.

\section{Main results}
\label{sec: main results}

In this section we present our main results, we discuss several extensions and we outline our methodology. In Section \ref{sub: refined learning scheme} we state our main result for the refined learning scheme, in the time-inhomogeneous exponential scenario, and we provide an overview of our proof technique. In Section \ref{sub: coxian service times} we formulate our main result for Coxian distributed service times, in the time-homogeneous scenario with the basic learning scheme, and we indicate the main steps of the proof. Extensions of our main results are discussed in Section \ref{sub: extensions}.

\subsection{Refined learning scheme}
\label{sub: refined learning scheme}

Assume that $\lambda$ is time-varying, that service requirements are exponentially distributed and that the refined learning scheme is used. We consider the asymptotic behavior of the threshold and the occupancy state of the system over intervals of time where the instantaneous offered load is sufficiently well-behaved, as specified in the next definition.

\begin{definition}
	\label{def: (Delta, m)-bounded}
	Consider an integer $m \geq 0$ and an interval $[a, b] \subset [0, T]$. Suppose that
	\begin{equation*}
		m\Delta < \rhomin \defeq \inf_{t \in [a, b]} \rho(t) \leq \sup_{t \in [a, b]} \rho(t) \eqdef \rhomax < (m + 1) \Delta.
	\end{equation*}
	Moreover, assume that the length of $[a, b]$ is strictly larger than
	\begin{equation*}
		\sigma(a, b, m, \Delta) \defeq \begin{cases}
			\frac{1}{\mu}\log \left(\frac{\rhomin}{\rhomin - m \Delta}\right) + \frac{1}{\mu}\left[\log\left(\frac{u(a) - \rhomax}{(m + 1)\Delta - \rhomax}\right)\right]^+ & \text{if} \quad u(a) > \rhomax, \\
			\frac{1}{\mu}\log\left(\frac{\rhomin}{\rhomin - m \Delta}\right) & \text{if} \quad u(a) \leq \rhomax.
		\end{cases}
	\end{equation*}
	We say that the offered load is $(m, \Delta)$-bounded on $[a, b]$ if the above conditions hold, and we say that the offered load is $\Delta$-bounded on $[a, b]$ if it is $(m, \Delta)$-bounded for some $m$.
\end{definition}

Our main result regarding the refined learning scheme is presented below. Loosely speaking, it states that the threshold reaches an equilibrium during $\Delta$-bounded intervals in all sufficiently large systems with probability one, provided that the control parameter $\alpha_n$ approaches one with $n$ at a suitable rate. Furthermore, the occupancy states are asymptotically nearly balanced under the latter conditions and the degree of balance increases as $\Delta$ decreases. 

\begin{theorem}
	\label{the: main theorem exponential case}
	Suppose that there exists $\gamma_0 \in (0, 1/ 2)$ such that
	\begin{equation}
		\label{ass: conditions on alpha}
		\lim_{n \to \infty} \alpha_n = 1 \quad \text{and} \quad \liminf_{n \to \infty} n^{\gamma_0}(1 - \alpha_n) > 0.
	\end{equation}
	If the offered load is $(m, \Delta)$-bounded on an interval $[a, b]$ and $\sigma(a, b, m, \Delta) < \sigma < b - a$, then there exist $c > 0$ and a set of probability one where the following statements hold:
	\begin{subequations}
		\begin{align}
			&\lim_{n \to \infty} \sup_{t \in [a + \sigma, b]} |\ell_n(t) - m \Delta| = 0, \label{seq: main result 1} \\
			&\lim_{n \to \infty} \sup_{t \in [a + \sigma, b]} n^\gamma\left|1 - q_n(t, i)\right| = 0 \quad \text{for all} \quad i \leq m \Delta \quad \text{and all} \quad \gamma \in [0, 1/2), \label{seq: main result 2} \\
			&\limsup_{n \to \infty} \sup_{t \in [a + \sigma, b]} \sum_{i > (m + 1)\Delta} q_n(t, i)\e^{\mu\left[t - (a + \sigma)\right]} \leq c. \label{seq: main result 3}
		\end{align} 
	\end{subequations}
	The constant $c$ can be taken equal to $u(a + \sigma)$.
\end{theorem}

The conditions in \eqref{ass: conditions on alpha} are crucial for stabilizing the threshold at $m\Delta$. Loosely speaking, the first condition guarantees that the threshold cannot stabilize at a larger value, while the second condition ensures that the threshold cannot decrease after settling at $m\Delta$ as a result of oscillations of the number of tasks around its mean, which are of order $\sqrt{n}$. Since the threshold takes discrete values, \eqref{seq: main result 1} yields $\ell_n(t) = m \Delta$ for all $t \in [a + \sigma, b]$ and all large enough $n$. By \eqref{seq: main result 2} and \eqref{seq: main result 3}, the supremum over $[a + \sigma, b]$ of the fraction of server pools with less than $m\Delta$ tasks is $o\left(1 / n^{\gamma}\right)$ and the fraction of server pools with more than $(m + 1) \Delta$ tasks is upper bounded by $O(1)\e^{-\mu[t - (a + \sigma)]}$ for all $t \in [a + \sigma, b]$. The constant $\sigma(a, b, m, \Delta)$ provides a closed-form upper bound for the asymptotic time until the threshold reaches an equilibrium value. When the system is initially empty and $\rho$ is constant over time, this upper bound is tight and corresponds to the asymptotic time until the scaled total number of tasks reaches $m\Delta$ for the first time.

We prove in Proposition \ref{prop: total number of tasks} that the normalized total number of tasks $N_n / n$ converges uniformly over compact sets to the function $u$ defined in \eqref{eq: total number of tasks} with probability one. The latter equation and the definition of $\Delta$-boundeness yield $m\Delta < u(t) < (m + 1) \Delta$ for all $t \in [a + \sigma, b]$. If $\Delta = 1$, then this implies that all the server pools have either $\floor{N_n(t) / n}$ or $\ceil{N_n(t) / n}$ tasks, with the exception of a fraction that is $o\left(1 / n^{\gamma}\right) +  O(1)\e^{-\mu[t - (a + \sigma)]}$. As $\Delta$ increases, the load becomes more coarsely balanced, but the threshold is more stable against small oscillations of the demand. This allows to avoid some threshold updates, and the associated communication overhead.

\begin{remark}
	\label{rem: on optimality}
	If $\lambda$ is constant over time, then the fluid limit of the normalized total number of tasks satisfies $\dot{u} = \lambda - \mu u$ and hence $u(t) = \rho + \left[u(0) - \rho\right]\e^{-\mu t}$. In particular, $u(t) \to \rho$ as $t \to \infty$ and the load is perfectly balanced in the limit if all server pools have at least $\floor{\rho}$ tasks and no server pool has more than $\ceil{\rho}$ tasks. If $\Delta = 1$, then \eqref{seq: main result 2} states that all server pools have at least $\floor{\rho}$ tasks but \eqref{seq: main result 3} only states that the number of tasks located in column $\ceil{\rho} + 1$ or further to the right in the diagram of Figure \ref{fig: state variables} decays at least exponentially fast; this implies in particular that the fraction of server pools with more than $\ceil{\rho}$ tasks decays at least exponentially fast. This is nonetheless the most balanced distribution of the load that we can expect in a transient regime for general initial conditions. Loosely speaking,
	\begin{equation*}
		\lim_{n \to \infty} \sum_{i = 1}^\infty q_n(t, i) = u(t) = \rho + \left[u(0) - \rho\right]\e^{-\mu t} \quad \text{almost surely},
	\end{equation*} 
	regardless of the load balancing policy. If $u(0) > \ceil{\rho}$, then the number of tasks located in columns $\ceil{\rho} + 1$ or further to the right in the diagram of Figure \ref{fig: state variables} is at least $u(0) - \ceil{\rho} > 0$ at time zero. The fastest rate of decay for the latter number of tasks is achieved when incoming tasks are never sent to server pools with at least $\floor{\rho}$ tasks. In this case we have the same exponential rate of decay over time as for $u$; see Proposition \ref{prop: tail number of tasks}. It follows that no load balancing policy can beat the exponential rate of decay obtained in \eqref{seq: main result 3} for general initial conditions.
\end{remark}

\subsubsection{Numerical experiments}
\label{sub: numerical experiments}

Figure \ref{fig: simulation} illustrates the tradeoff between optimality and stability governed by $\Delta$. The experiments correspond to two systems operating under the same conditions, except for the choice of $\Delta$. The time-varying offered load, the number of server pools and $\alpha_n$ are the same in both systems, but while $\Delta = 1$ in one of the systems, $\Delta = 3$ in the other.

The offered load and threshold processes of the systems are plotted in Figure \ref{fig: simulation a}. On a slower time scale, we observe that the offered load switches between two modes: the first one corresponds to a period of larger demand, during the interval $[3, 12]$, while the second one corresponds to a period of smaller demand, during $[14, 23]$. On a faster time scale, the offered load exhibits rapid oscillations, which have larger amplitude along $[3, 12]$ and become smaller during $[14, 23]$. The threshold of the system with $\Delta = 3$ is constant along both intervals, as one would expect from \eqref{seq: main result 1} since both are $3$-bounded. In contrast, the threshold of the system with $\Delta = 1$ is constant during the interval $[14, 23]$ but exhibits oscillations along $[3, 12]$, which is not $1$-bounded.

The plots on the bottom of Figure \ref{fig: simulation} depict the occupancy state of the systems along intervals of time where the threshold remains constant. Figure \ref{fig: simulation b} shows the occupancy state of the system with $\Delta = 3$ along the interval $[3, 12]$. In consonance with \eqref{seq: main result 2} and \eqref{seq: main result 3}, the fraction of server pools with fewer than $\ell_n^{(3)}$ tasks or more than $\ell_n^{(3)} + 3$ tasks is nearly zero. However, the load is not perfectly balanced since the fraction of server pools with exactly $\ell_n^{(3)} + i$ tasks remains positive for all $0 \leq i \leq 3$. In contrast, the load is perfectly balanced in the system with $\Delta = 1$ along $[14, 23]$, as Figure \ref{fig: simulation c} shows; i.e., at all times virtually all server pools have either $\ell_n^{(1)}$ or $\ell_n^{(1)} + 1$ tasks.

As explained above, the system with $\Delta = 1$ manages to distribute the load evenly among the server pools during the interval $[14, 23]$, where the offered load only exhibits small oscillations. However, the larger demand fluctuations during the interval $[3, 12]$ result in an unstable threshold, which exhibits frequent oscillations along this interval. In contrast, the system with $\Delta = 3$ holds the threshold fixed along the intervals $[3, 12]$ and $[14, 23]$, although this comes at the expense of the load being only coarsely balanced along each of these intervals.

\begin{figure}
	\centering
	\begin{subfigure}{\columnwidth}
		\centering
		\includegraphics{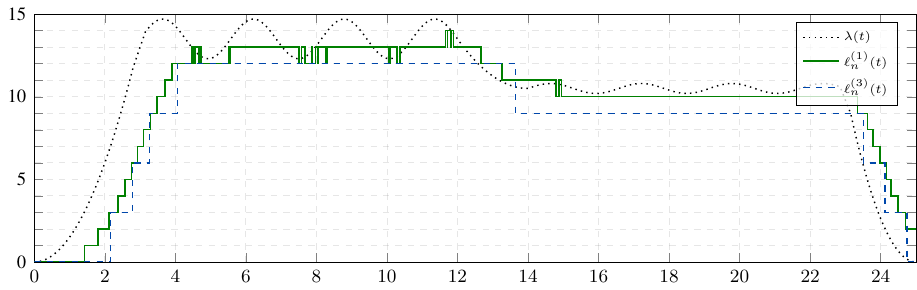}
		\caption{}
		\label{fig: simulation a}
	\end{subfigure}
	\vfill
	\begin{subfigure}{0.49\columnwidth}
		\centering
		\includegraphics{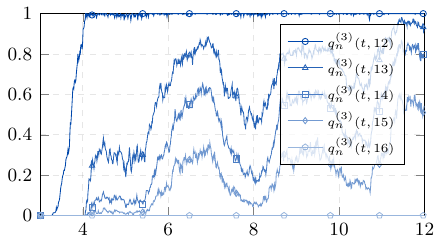}
		\caption{}
		\label{fig: simulation b}
	\end{subfigure}
	\hfill
	\begin{subfigure}{0.49\columnwidth}
		\centering
		\includegraphics{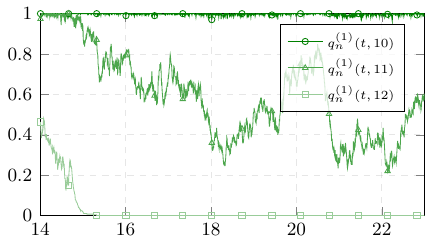}
		\caption{}
		\label{fig: simulation c}
	\end{subfigure}
	\caption{Numerical experiments with two different choices of $\Delta$. The evolution of the thresholds is plotted in (a), the occupancy state of the system with $\Delta = 3$ along the $3$-bounded interval $[3, 12]$ is depicted in (b) and the occupancy state of the system with $\Delta = 1$ along the $1$-bounded interval $[14, 23]$ is plotted in (c). Both experiments correspond to initially empty systems with $\mu = 1$, $n = 300$, $\alpha_n = 1 - 1 / n^{0.48}$ and the arrival rate function plotted in (a); the superscripts in the legends indicate the value of $\Delta$.}
	\label{fig: simulation}
\end{figure}
%alpha_n = 1 - 0.2 / n^{0.2}
%alpha_n = 1 - 0.6 / n^{0.4}

\subsubsection{Outline of the proof}
\label{sub: outline of the proof}

We now present an outline of the proof of Theorem \ref{the: main theorem exponential case}, which is illustrated in Figure \ref{fig: proof diagram}. This theorem concerns a sequence of occupancy and threshold processes constructed on a common probability space in Section \ref{sub: relative compactness exponential case}, where it is also established that the occupancy processes have relatively compact sample paths with probability one. This property is proven in Proposition \ref{prop: relative compactness of occupancy state sample paths} using a technique developed in \cite{bramson1998state}, which relies on strong approximations; specifically, on strong law of large numbers for the Poisson process.

A slight refinement of the latter law of large numbers is derived in Section \ref{sub: asymptotic dynamical properties exponential case}, where we define a set of sample paths for which this refined strong law of large numbers, and a few additional technical properties, hold. Propositions \ref{prop: total number of tasks} and \ref{prop: tail number of tasks} prove that certain dynamical properties hold asymptotically within this set of sample paths, which has probability one. These asymptotic dynamical properties concern the total and tail mass processes:
\begin{equation}
	\label{eq: total and tail mass processes}
	u_n \defeq \sum_{i = 1}^\infty q_n(i) = \frac{N_n}{n} \quad \text{and} \quad v_n(j) \defeq \sum_{i = j}^\infty q_n(i) \quad \text{for all} \quad j \geq 1,
\end{equation}
respectively. The former represents the total number of tasks in the system, normalized by the number of server pools. The latter corresponds to the total number of tasks in all but the first $j - 1$ columns of the diagram depicted in Figure \ref{fig: state variables}, also normalized by the number of server pools. Proposition \ref{prop: total number of tasks} establishes that $u_n$ converges uniformly over $[0, T]$ to the function $u$ defined by~\eqref{eq: total number of tasks} with probability one, which is the strong law of large numbers of an infinite-server system. Proposition \ref{prop: tail number of tasks} is conceived to be used in conjunction with Proposition \ref{prop: relative compactness of occupancy state sample paths} since it assumes that a convergent subsequence of occupancy process sample paths is given. This proposition provides an asymptotic upper bound for a certain tail mass process under suitable conditions on the threshold and occupancy processes. Essentially, the conditions ensure that server pools with at least a certain number $j$ of tasks are not assigned any further tasks, and the asymptotic upper bound implies that $v_n(j + 1)$ decays at least exponentially fast in a many-server regime.

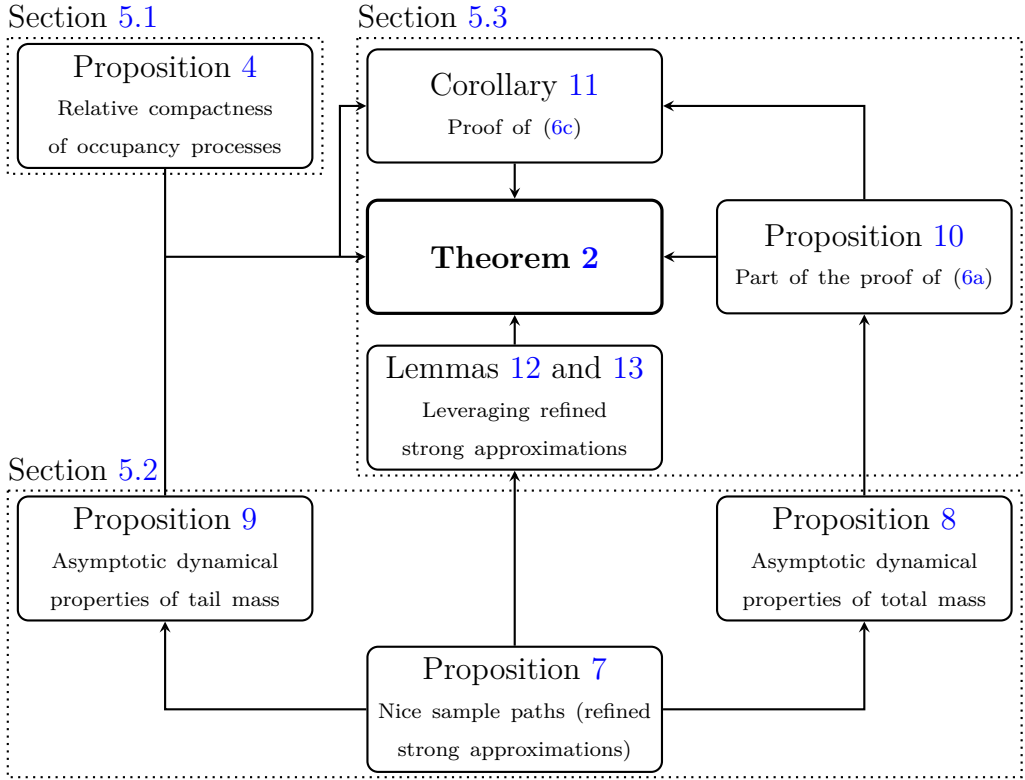
\begin{figure}
	\centering
	\begin{tikzpicture}[x = 5cm, y = 2.3cm]
		\path (0, 0)			node[result]	(theo2)		{Proposition \ref{prop: relative compactness of occupancy state sample paths}\\ \scriptsize{Relative compactness of occupancy processes}}
		+(1, 0)			node[result]	(coro1)		{Corollary \ref{cor: decay of the tail}\\ \scriptsize{Proof of \eqref{seq: main result 3}}}
		+(1,-1)			node[main]		(theo1)		{\textbf{Theorem \ref{the: main theorem exponential case}}}
		+(2,-1)			node[result]	(prop6)		{Proposition \ref{prop: bounded threshold}\\ \scriptsize{Part of the proof of \eqref{seq: main result 1}}}
		+(1,-2)			node[result]	(lemm6)		{Lemmas \ref{lem: lemma 1 for main result} and \ref{lem: lemma 2 for main result}\\ \scriptsize{Leveraging refined strong approximations}}
		+(0,-3)			node[result]	(prop5)		{Proposition \ref{prop: tail number of tasks}\\ \scriptsize{Asymptotic dynamical properties of tail mass}}
		+(2,-3)			node[result]	(prop4)		{Proposition \ref{prop: total number of tasks}\\ \scriptsize{Asymptotic dynamical properties of total mass}}
		+(1,-4)			node[result]	(prop3)		{Proposition \ref{prop: definition of Omega_T}\\ \scriptsize{Nice sample paths (refined strong approximations)}}
		+(-0.48,0.45)	node[anchor = south west]	{Section \ref{sub: relative compactness exponential case}}
		+(0.52,0.45)	node[anchor = south west]	{Section \ref{sub: evolution of the threshold}}
		+(-0.48,-2.55)	node[anchor = south west]	{Section \ref{sub: asymptotic dynamical properties exponential case}};
		
		\draw[thick]	(theo2)		--		(prop5);
		\draw[arrow]	(0,-1)		--		(theo1);
		\draw[arrow]	(0.5,-1)	|-		(coro1);
		\draw[arrow]	(coro1)		--		(theo1);
		\draw[arrow]	(prop6)		|-		(coro1);
		\draw[arrow]	(lemm6)		--		(theo1);
		\draw[arrow]	(prop3)		-|		(prop5);
		\draw[arrow]	(prop3)		--		(lemm6);
		\draw[arrow]	(prop4)		--		(prop6);
		\draw[arrow]	(prop3)		-|		(prop4);
		\draw[arrow]	(prop6)		--		(theo1);
		
		\draw[dotted, thick]	(-0.45,0.45)	rectangle	(0.45,-0.45);
		\draw[dotted, thick]	(0.55, 0.45)	rectangle	(2.45,-2.45);
		\draw[dotted, thick]	(-0.45, -2.55)	rectangle	(2.45,-4.45);
		
	\end{tikzpicture}
	\caption{Schematic representation of the relations between the various results used to prove Theorem~\ref{the: main theorem exponential case}; some intermediate results are omitted. An arrow connecting two results means that the first result is used to prove the second. The dotted boxes indicate the sections where the results are proven.}
	\label{fig: proof diagram}
\end{figure}

The proof of Theorem \ref{the: main theorem exponential case} is completed in Section \ref{sub: evolution of the threshold} through several steps. First Proposition~\ref{prop: total number of tasks} is used to prove part of \eqref{seq: main result 1}. Specifically, it is established in Proposition \ref{prop: bounded threshold} that $\ell_n$ is almost surely upper bounded by $m \Delta$ along $[a, b]$ minus a right neighborhood of $a$, for all large enough~$n$. Next Propositions \ref{prop: relative compactness of occupancy state sample paths}, \ref{prop: tail number of tasks} and \ref{prop: bounded threshold} are used to establish \eqref{seq: main result 3}. As mentioned above, Propositions \ref{prop: relative compactness of occupancy state sample paths} and \ref{prop: tail number of tasks} are used in conjunction. Loosely speaking, under the assumption that \eqref{seq: main result 3} does not hold, it is possible to construct a convergent subsequence of occupancy processes violating \eqref{seq: main result 3}, but Proposition \ref{prop: tail number of tasks} leads to a contradiction. Finally, Lemmas \ref{lem: lemma 1 for main result} and \ref{lem: lemma 2 for main result} leverage the refined strong approximations derived in Section \ref{sub: asymptotic dynamical properties exponential case} to analyze the occupancy processes on the scale of $n^{\gamma_0}$. This analysis and the asymptotic dynamical properties derived in Section \ref{sub: asymptotic dynamical properties exponential case} are then used to finish the proof of \eqref{seq: main result 1} and to also prove \eqref{seq: main result 2}, completing the proof of Theorem \ref{the: main theorem exponential case}.

\subsection{Coxian service times}
\label{sub: coxian service times}

Assume that $\lambda$ is constant over time, that service times are Coxian distributed and that the basic learning scheme is used. Our main result in this setting is that the threshold eventually reaches an equilibrium value in a many-server regime with probability one, and that the load becomes nearly balanced after the threshold settles.

\begin{theorem}
	\label{the: main theorem coxian case}
	Suppose that $\rho \notin \set{i\Delta}{i \in \N}$. Then there exist constants $\sigma, \eta, c > 0$ and a set of probability one where the following three statements hold:
	\begin{subequations}
		\begin{align}
			&\lim_{n \to \infty} \sup_{t \in [\sigma, T]} |\ell_n(t) - \floor{\rho}_\Delta| = 0, \label{seq: main result coxian 1} \\
			&\lim_{n \to \infty} \sup_{t \in [\sigma, T]} \left|1 - q_n(t, i)\right| = 0 \quad \text{for all} \quad i \leq \floor{\rho}_\Delta, \label{seq: main result coxian 2} \\
			&\limsup_{n \to \infty} \sup_{t \in [\sigma, T]} \sum_{i > \floor{\rho}_\Delta + \Delta} q_n(t, i)\e^{\eta\left(t - \sigma\right)} \leq c. \label{seq: main result coxian 3}
		\end{align} 
	\end{subequations}
	The constants $\sigma$, $\eta$ and $c$ do not depend on $T$.
\end{theorem}

This result has similar consequences as Theorem \ref{the: main theorem exponential case}. Namely, \eqref{seq: main result coxian 1} implies that $\ell_n(t) = \floor{\rho}_\Delta$ for all $t \in [\sigma, T]$ and all large enough $n$. Also, \eqref{seq: main result coxian 2} and \eqref{seq: main result coxian 3} imply that the total number of tasks in the system is asymptotically nearly balanced across the server pools during the interval $[\sigma, T]$. Furthermore, the load is evenly balanced if $\Delta = 1$; i.e., the fraction of server pools with less than $\floor{\rho}$ tasks vanishes with $n$ and the fraction of server pools with more than $\ceil{\rho}$ tasks is upper bounded by $O(1)\e^{-\eta(t - \sigma)}$. The technical condition $\rho \notin \set{i\Delta}{i \in \N}$ ensures that the threshold does not oscillate between $\rho - \Delta$ and $\rho$ when the offered load is integral and a multiple of $\Delta$. In addition, the constant $\eta$ is a lower bound for the hazard rate function of the service time distribution, and is equal to $\mu$ for exponentially distributed service requirements.

The proof of Theorem \ref{the: main theorem coxian case} is based on the methodology outlined in Section \ref{sub: outline of the proof}. Specifically, we first construct a sequence of occupancy and threshold processes on a common probability space from a set of stochastic primitives, and we establish that the sequence of occupancy processes is relatively compact with probability one; this is done in Section \ref{sub: relative compactness coxian case}. Then we identify certain dynamical properties of the system and prove that they hold asymptotically in Section \ref{sub: asymptotic dynamical properties coxian case}; these properties concern aggregate quantities analogous to the total and tail mass processes introduced in \eqref{eq: total and tail mass processes}. Finally, the latter properties are used in Section \ref{sub: proof of the main result coxian case} to prove Theorem \ref{the: main theorem coxian case}.

\subsection{Extensions}
\label{sub: extensions}

Our proof of Theorem \ref{the: main theorem exponential case} can be easily adapted to prove an analogous result for the basic learning scheme. Specifically, assume that the hypotheses of Theorem \ref{the: main theorem exponential case} hold, except that the basic learning scheme is used instead of the refined learning scheme, and let
\begin{equation*}
	\tilde{\sigma}_\eq(a, b, m, \Delta) \defeq
	\begin{cases}
		\frac{1}{\mu}\left[\log \left(\frac{\rhomin - u(a)}{\rhomin - m \Delta}\right)\right]^+ & \text{if} \quad u(a) < \rhomin, \\
		0 & \text{if} \quad \rhomin \leq u(a) \leq \rhomax, \\
		\frac{1}{\mu}\left[\log \left(\frac{u(a) - \rhomax}{(m + 1) \Delta - \rhomax}\right)\right]^+ & \text{if} \quad u(a) > \rhomax.
	\end{cases}
\end{equation*}
It follows from \eqref{eq: total number of tasks} that $m\Delta < u(t) < (m + 1)\Delta$ for all $a + \tilde{\sigma}_\eq(a, b, m, \Delta) < t \leq b$. The strong law of large numbers for the total number of tasks provided in Proposition \ref{prop: total number of tasks} implies that $u_n(t)$ converges uniformly over $[0, T]$ to $u(t)$ with probability one. Hence, \eqref{seq: main result 1} holds for all $\sigma > \tilde{\sigma}_\eq(a, b, m, \Delta)$ when the basic learning scheme is used. Corollary \ref{cor: decay of the tail} applies to the basic learning scheme with exactly the same proof; the only difference is that the quantity $\sigma_{\bd}(a, b, m, \Delta)$ in the statement must be replaced by $\tilde{\sigma}_{\eq}(a, b, m, \Delta)$. In particular, \eqref{seq: main result 3} also holds for the basic learning scheme if $\sigma > \tilde{\sigma}_\eq(a, b, m, \Delta)$. Finally, we prove that \eqref{seq: main result 2} holds for the basic learning scheme and all
\begin{equation*}
	\sigma > \tilde{\sigma}(a, b, m, \Delta) \defeq \tilde{\sigma}_\eq(a, b, m, \Delta) + \frac{1}{\mu}\log\left(\frac{\rhomin}{\rhomin - m \Delta}\right)
\end{equation*}
by arguing as in the proof of Theorem \ref{the: main theorem exponential case}. The only difference with the proof for the refined learning scheme is that all instances of $\sigma_{\bd}(a, b, m, \Delta)$ must be replaced by $\tilde{\sigma}_{\eq}(a, b, m, \Delta)$.

It is also possible to prove that \eqref{seq: main result coxian 2} and \eqref{seq: main result coxian 3} hold if the hypotheses of Theorem \ref{the: main theorem coxian case} hold, except that the basic learning scheme is replaced by the refined learning scheme. Indeed, as in the proof of Proposition \ref{prop: bounded threshold}, it can be shown that there exists $\sigma > 0$ such that $\ell_n(t) \leq \floor{\rho}_\Delta$ for all $t \geq \sigma$ and all large enough $n$ with probability one. Then \eqref{seq: main result coxian 2} and \eqref{seq: main result coxian 3} can be established as in the proof of Theorem \ref{the: main theorem coxian case}. Nevertheless, proving that the threshold eventually reaches an equilibrium in all sufficiently large systems would require non-trivial additional arguments. In particular, the arguments used in Section \ref{sub: evolution of the threshold} to establish this property in the case of exponentially distributed service times cannot be easily extended to a scenario with general Coxian service times.

As a final remark, we observe that the statements in Theorem \ref{the: main theorem exponential case} are only meaningful because the constant $\sigma(a, b, m, \Delta)$ can be evaluated and is reasonably small in many situations. The derivation of this constant heavily relies on the simplicity of \eqref{eq: total number of tasks} for exponential service times, whereas the more complex equation \eqref{eq: general total number of tasks} holds for general Coxian service times. Deriving a suitable analog of $\sigma(a, b, m, \Delta)$ would be challenging in a time-inhomogeneous scenario with Coxian distributed service times, even for the basic learning scheme. Furthermore, a closed-form expression as the one provided in Definition \ref{def: (Delta, m)-bounded} is not likely to be available in the latter setting.

\section{Refined learning scheme}
\label{sec: refined learning scheme}

In this section we prove Theorem \ref{the: main theorem exponential case}. For this purpose, we assume that the refined learning scheme is used, that $\lambda$ is time-varying and that service times are exponentially distributed. We proceed as indicated in Section \ref{sub: outline of the proof}.

\subsection{Relative compactness of occupancy processes}
\label{sub: relative compactness exponential case}

In Section \ref{subsub: coupled construction of sample paths exponential case} we construct the sample paths of the occupancy states $q_n$ and the thresholds $\ell_n$ on a common probability space for all $n$. The sample paths of the occupancy states lie in the space $D_{\R^\N}[0, T]$ of all c\`adl\`ag functions defined on $[0, T]$ with values in $\R^\N$, which we endow with the topology of uniform convergence. Specifically, consider the metric $d$ defined by
\begin{equation*}
	d(x, y) \defeq \sum_{i = 0}^\infty \frac{\min\left\{|x(i) - y(i)|, 1\right\}}{2^i} \quad \text{for all} \quad x, y \in \R^\N,
\end{equation*}
which generates the product topology on $\R^\N$. The topology of uniform convergence equipped to the space $D_{\R^\N}[0, T]$ is given by the following metric:
\begin{equation*}
	\varrho(x, y) \defeq \sup_{t \in [0, T]} d\left(x(t), y(t)\right) \quad \text{for all} \quad x, y \in D_{\R^\N}[0, T].
\end{equation*}
In Section \ref{subsub: relative compactness exponential case} we use a methodology developed in \cite{bramson1998state} to establish that $\set{q_n}{n \geq 1}$ is a relatively compact subset of $D_{\R^\N}[0, T]$ with probability one.

\subsubsection{Coupled construction of sample paths}
\label{subsub: coupled construction of sample paths exponential case}

The sample paths of the occupancy states and thresholds are defined as deterministic functions of the following stochastic primitives.

\begin{itemize}
	\item \textit{Driving Poisson processes:} a family $\set{\calN_i}{i \geq 0}$ of independent Poisson processes of unit rate, defined on a common probability space $(\Omega_D, \calF_D, \prob_D)$.
	
	\item \textit{Selection variables:} a sequence $\set{U_j}{j \geq 1}$ of independent and identically distributed uniform random variables with values on $[0, 1)$,  defined on a common probability space $(\Omega_S, \calF_S, \prob_S)$.
	
	\item \textit{Initial conditions:} families $\set{q_n(0)}{n \geq 1}$ and $\set{\ell_n(0)}{n \geq 1}$ of random variables describing the initial conditions of the systems, defined on the common probability space $(\Omega_I, \calF_I, \prob_I)$ and satisfying the assumptions introduced in Section \ref{sec: notation and assumptions}.
\end{itemize}

Consider the product probability space of $(\Omega_D, \calF_D, \prob_D)$, $(\Omega_S, \calF_S, \prob_S)$ and $(\Omega_I, \calF_I, \prob_I)$; denote its completion by $(\Omega, \calF, \prob)$. The occupancy states and thresholds are defined on the latter space from a set of equations involving the fundamental processes and random variables introduced above. In order to write these equations, it is convenient to introduce some notation.

For each occupancy state $q \in Q$, we define the intervals
\begin{equation*}
	I_i(q) \defeq \left[1 - q(i - 1), 1 - q(i)\right) \quad \text{for all} \quad i \geq 1.
\end{equation*}
These intervals form a partition of $[0, 1)$ such that the length of $I_i(q)$ is the fraction of server pools with precisely $i - 1$ tasks. If $q(j) < 1$ for some $j \in \N$, then we also define
\begin{equation*}
	J_i(q, j) \defeq \left[\frac{1 - q(i - 1)}{1 - q(j)}, \frac{1 - q(i)}{1 - q(j)}\right) \quad \text{for all} \quad 1 \leq i \leq j.
\end{equation*}
These intervals yield another partition of $[0, 1)$. In this case, the length of $J_i(q, j)$ is the fraction of server pools having exactly $i - 1$ tasks, but only among those server pools with at most $j - 1$ tasks; we define $J_i(q, j) \defeq \emptyset$ for all $1 \leq i \leq j$ when $q(j) = 1$.

If $\ell \in \N$ represents a threshold and $q(\ell) < 1$, then the length of $J_i(q, \ell)$ is equal to the probability of picking a server pool with precisely $i - 1$ tasks uniformly at random among those with strictly less than $\ell$ tasks. Similarly, if $h \defeq \ell + \Delta$, $q(\ell) = 1$ and  $q(h) < 1$, then the length of $J_i(q, h)$ is equal to the probability of picking a server pool with exactly $i - 1$ tasks uniformly at random among those with at least $\ell$ tasks and strictly less than $h$ tasks. We define
\begin{equation*}
	r_{ij}(q, \ell) \defeq \begin{cases}
		\ind{U_j \in J_i(q, \ell)} & \text{if} \quad i - 1 < \ell, \\
		\ind{q(\ell) = 1, U_j \in J_i(q, h)} & \text{if} \quad \ell \leq i - 1 < h, \\
		\ind{q(h) = 1, U_j \in I_i(q)} & \text{if} \quad i - 1 \geq h,
	\end{cases}
	\quad \text{for all} \quad i,j \geq 1.
\end{equation*}
Note that $r_{ij}(q, \ell) \in \{0, 1\}$ and that for a fixed $j$ there exists a unique $i$ such that $r_{ij}(q, \ell) = 1$. This family of random variables will be used to describe the dispatching decisions in the systems. Specifically, if $q$ and $\ell$ are the occupancy state and the threshold, respectively, when the $j^{\text{th}}$ task arrives to the system, then this task is sent to a server pool with exactly $i - 1$ tasks if and only if $r_{ij}(q, \ell) = 1$; this coincides with the dispatching rule described in Section \ref{sub: dispatching rule}.

We postulate that
\begin{equation*}
	\calN_n^\lambda(t) \defeq \calN_0\left(n\int_0^t \lambda(s)ds\right)
\end{equation*}
is the number of tasks that arrive to the system with $n$ server pools during the interval $[0, t]$, we denote the jump times of $\calN_n^\lambda$ by $\set{\tau_{n,k}}{k \geq 1}$ and we let $\tau_{n,0} \defeq 0$. Note that $\calN_n^\lambda$ is a Poisson process with time-varying intensity $n\lambda$, as required by our model.

For each $n$ and each pair of functions $\map{q}{[0, T]}{Q}$ and $\map{\ell}{[0, T]}{\N}$, we introduce two infinite-dimensional counting processes, for arrivals and departures, denoted $\calA_n(q, \ell)$ and $\calD_n(q)$, respectively. The $i = 0$ coordinates of these two counting processes are identically zero, whereas the other coordinates are defined as follows:
\begin{equation*}
	\begin{split}
		&\calA_n(q, \ell, t, i) \defeq \frac{1}{n} \sum_{j = 1}^{\calN_n^\lambda(t)} r_{ij}\left(q\left(\tau_{n, j}^-\right), \ell\left(\tau_{n, j}^-\right)\right), \\
		&\calD_n(q, t, i) \defeq \frac{1}{n} \calN_i\left(n\int_0^t \mu i \left[q(s, i) - q(s, i + 1)\right]ds\right),
	\end{split}
	\quad \text{for all} \quad i \geq 1.
\end{equation*}

Consider the following set of implicit equations:
\begin{align*}
	&q(t) = q_n(\omega, 0) + \calA_n(q, \ell, \omega, t) - \calD_n(q, \omega, t), \\
	&\ell(t) = \ell_n(\omega, 0) + \sum_{k = 1}^{\calN_n^\lambda(\omega, t)} \Delta\left[\indc{A_{n, k}}(\omega) - \indc{B_{n, k}}(\omega)\right], \\
	&A_{n, k} = \set{\omega \in \Omega}{q\left(\tau_{n, k}^-(\omega), \ell\left(\tau_{n, k}^-(\omega)\right)\right) = 1\ \text{and}\ nq\left(\tau_{n, k}^-(\omega), \ell\left(\tau_{n, k}^-(\omega)\right) + \Delta\right) \geq n - 1}, \\
	&B_{n, k} = \set{\omega \in \Omega}{q\left(\tau_{n, k}^-(\omega), \ell\left(\tau_{n, k}^-(\omega)\right)\right) \leq \alpha_n}.
\end{align*}
Recall that $q_n(0, B + 1) = 0$ for all $n$ almost surely, as indicated in Section \ref{sec: notation and assumptions}. In particular, the initial number of tasks in the system is finite with probability one. Using this property, and proceeding by forward induction on the jumps of the driving Poisson processes, it is possible to prove that there exists a set of probability one $\Gamma_0$ with the next property. For each $\omega \in \Gamma_0$ and each $n$, there exist unique c\`adl\`ag solutions $\map{q_n(\omega)}{[0, T]}{Q}$ and $\map{\ell_n(\omega)}{[0, T]}{\N}$ to the above equations.

The sample paths of the occupancy states and the thresholds are defined by extending these solutions to $\Omega$, setting $q_n(\omega) \equiv 0$ and $\ell_n(\omega) \equiv 0$ for all $\omega \notin \Gamma_0$. In addition, we define
\begin{equation*}
	\calA_n \defeq \calA_n(q_n, \ell_n), \quad \calD_n \defeq \calD_n(q_n) \quad \text{and} \quad h_n \defeq \ell_n + \Delta,
\end{equation*}
and we rewrite the implicit equations as follows:
\begin{subequations}
	\begin{align}
		&q_n(t) = q_n(0) + \calA_n(t) - \calD_n(t), \label{seq: occupancy state}\\
		&\ell_n(t) = \ell_n(0) + \sum_{k = 1}^{\calN_n^\lambda(t)} \Delta \left[\ind{q_n\left(\tau_{n, k}^-, \ell_n\left(\tau_{n, k}^-\right)\right) = 1,\ nq_n\left(\tau_{n, k}^-, h_n\left(\tau_{n, k}^-\right)\right) \geq n - 1} - \ind{q_n\left(\tau_{n, k}^-, \ell_n\left(\tau_{n, k}^-\right)\right) \leq \alpha_n}\right]. \label{seq: threshold}
	\end{align}
\end{subequations}

The above construction endows the processes $q_n$ and $\ell_n$ with the intended statistical behavior. Indeed, the time-inhomogeneous Poisson process $\calN_n^\lambda$ has instantaneous rate $n\lambda(t)$ for each $n$, and the random functions $r_{ij}$ apply the dispatching rule described in Section \ref{sub: dispatching rule} when the occupancy state and threshold are passed as arguments. In particular, $\calA_n(i)$ models the arrivals to server pools with precisely $i - 1$ tasks in a system with $n$ server pools. In addition, the instantaneous intensity of $\calD_n(i)$ is given by $n\mu i[q_n(t, i) - q_{n}(t, i + 1)]$, which is equal to the total number of tasks in server pools with exactly $i$ tasks times the rate at which tasks are executed. Therefore, $\calD_n(i)$ models the departures from server pools with precisely $i$ tasks, and thus \eqref{seq: occupancy state} corresponds to the evolution of the occupancy state in a system with $n$ server pools. Furthermore, the refined learning scheme, for adjusting the threshold, is captured by \eqref{seq: threshold}.

\subsubsection{Relative compactness of sample paths}
\label{subsub: relative compactness exponential case}

Below we state the relative compactness result mentioned at the start of Section \ref{sub: relative compactness exponential case}, deferring the proof to Section \ref{app: relative compactness} of the appendix.

\begin{proposition}
	\label{prop: relative compactness of occupancy state sample paths}
	There exists a set of probability one $\Gamma_T \subset \Gamma_0$ with the following property. The sequences $\set{\calA_n(\omega)}{n \geq 1}$, $\set{\calD_n(\omega)}{n \geq 1}$ and $\set{q_n(\omega)}{n \geq 1}$ are relatively compact for each $\omega \in \Gamma_T$ and have the property that the limit of every convergent subsequence is a function with Lipschitz coordinates, also for each $\omega \in \Gamma_T$. 
\end{proposition}

Essentially, the proof of the above proposition relies on the decomposition \eqref{seq: occupancy state} and the fact that the coordinates of $\calA_n$ and $\calD_n$ have intensities which are $O(n)$. In particular, observe that the proof provided in Section \ref{app: relative compactness} does not depend on the learning scheme used for adjusting the threshold over time, or even on the rule used to take the dispatching decisions.

\subsection{Asymptotic dynamical properties}
\label{sub: asymptotic dynamical properties exponential case}

In this section we prove asymptotic dynamical properties concerning the total and tail mass processes. These were defined in \eqref{eq: total and tail mass processes} as
\begin{equation*}
	u_n \defeq \sum_{i = 1}^\infty q_n(i) \quad \text{and} \quad v_n(j) \defeq \sum_{i = j}^\infty q_n(i) \quad \text{for all} \quad j \geq 1,
\end{equation*}
respectively. The total mass process $u_n$ corresponds to the total number of tasks in the system, normalized by the number of server pools, as the diagram of Figure \ref{fig: state variables} suggests. The tail mass process $v_n(j)$ represents the total number of tasks located in column $j$ of the latter diagram, or further to the right, also normalized by the number of server pools. Section \ref{subsub: set of nice sample paths} contains some preliminary results, and the asymptotic dynamical properties are established in Section \ref{subsub: properties of total and tail mass functions}.

\subsubsection{Set of nice sample paths}
\label{subsub: set of nice sample paths}

We begin by introducing a set of probability one consisting of sample paths which are well-behaved, in a sense that will be made clear below. For this purpose, we need two technical lemmas. The first one establishes asymptotic upper bounds, which are uniform over time, for the number of tasks sharing a server pool and the threshold.

\begin{lemma}
	\label{lem: upper bound for number of tasks and threshold}
	There exists a positive constant $B_T \in \N$ such that $q_n(t, B_T + 1) = 0$ and $\ell_n(t) \leq B_T$ for all $t \in [0, T]$ and all sufficiently large $n$ with probability one.
\end{lemma}

\begin{proof}
First let us prove that there exists $k \in \N$ such that $q_n(t, k \Delta + 1) = 0$ for all $t \in [0, T]$ and all sufficiently large $n$ with probability one. Note that the set
\begin{equation*}
	\bigcap_{m \geq 1} \bigcup_{n \geq m} \set{\omega \in \Omega}{q_n(\omega, t, k \Delta + 1) > 0\ \text{for some}\ t \in [0, T]}
\end{equation*}
is measurable for each $k \in \N$ since the occupancy processes have right-continuous sample paths. We must establish that this set has probability zero for some $k$. Consider the sets
\begin{equation*}
	E_n^k \defeq \set{\omega \in \Omega}{q_n(\omega, t, k \Delta + 1) > 0\ \text{for some}\ t \in [0, T]}.
\end{equation*}
By the Borel-Cantelli lemma, it is enough to prove that there exists $k \in \N$ such that
\begin{equation}
	\label{eq: borel-cantelli condition}
	\sum_{n = 1}^\infty \prob\left(E_n^k\right) < \infty.
\end{equation}

Recall from Section \ref{sec: notation and assumptions} that there exists $B$ such that $q_n(0, B + 1) = 0$ and $\ell_n(0) \leq B$ for all $n$ with probability one. By Remark \ref{rem: well positioned threahold}, this implies that $q_n(h_n) < 1$ holds at all times after at most $\floor{B}_\Delta$ arrivals for all $n > \floor{B}_\Delta$ with probability one. Hence, we may assume without loss of generality that $q_n(h_n) < 1$ holds from time zero, by possibly moving the origin of time to the point where this condition holds for the first time. Indeed, note that at this point $q_n(B + \floor{B}_\Delta + 1) = 0$ and $q_n(i) < 1$ for all $i \geq B + 1$ because no more than $\floor{B}_\Delta < n$ arrivals may have occurred. In addition, this in turn implies that $h_n \leq \left(\floor{B}_\Delta + 1\right) \Delta$ and therefore $\ell_n \leq \floor{B}_\Delta \Delta \leq B$. We thus assume that $q_n(0, h_n(0)) < 1$ almost surely, so tasks are always sent to pools with at most $h_n - 1$ tasks.

Suppose that $k > B / \Delta$. This implies that $q_n(0, k\Delta + 1) = 0$ and $\ell_n(0) \leq (k - 1) \Delta$. Only the server pools with at most $h_n - 1 \leq k\Delta - 1$ tasks can receive incoming tasks while $\ell_n \leq (k - 1)\Delta$, and therefore $q_n(k\Delta + 1)$ remains zero until $\ell_n$ reaches $k\Delta$. In order to reach $k\Delta$, the threshold must reach $(k - 1)\Delta$ first, and then $nq_n(k\Delta)$ must grow from $n - 1$ to $n$ for $\ell_n$ to increase. Hence,
\begin{align*}
	\prob\left(E_n^k\right) &\leq \prob\left(nq_n(t, k\Delta) \geq n - 1 \ \text{for some} \ t \in [0, T]\right) \\
	&\leq \prob\left(nu_n(t) \geq k\Delta \left(n - 1\right) \ \text{for some} \ t \in [0, T]\right),
\end{align*}
where the last inequality follows from $u_n(t) / k\Delta \geq q_n(t, k\Delta)$; this holds since $q_n(t)$ is a non-increasing sequence. The total number of tasks in the system at time $t$ is upper bounded by $\calN_n^\lambda(t) + nu_n(0)$, the number of arrivals during $[0, t]$ plus the initial number of tasks in the system. Hence,
\begin{align*}
	\prob\left(E_n^k\right) &\leq \prob \left(\calN_n^\lambda(t) + nu_n(0) \geq k\Delta (n - 1) \ \text{for some} \ t \in [0, T]\right) \\
	&\leq \prob \left(\calN_n^\lambda(t) + n B \geq k\Delta (n - 1) \ \text{for some} \ t \in [0, T]\right) \\
	&= \prob \left(\calN_n^\lambda(T) \geq (k\Delta - B) n - k\Delta \right).
\end{align*}
The second inequality follows from $q_n(0, B + 1) = 0$, which implies that $u_n(0) \leq B$.

Applying a Chernoff bound, we conclude that
\begin{align*}
	\prob\left(E_n^k\right) \leq \prob\left(\calN_n^\lambda(T) \geq (k\Delta - B) n - k\Delta \right) \leq \frac{\e^{A\left(\e - 1\right)n}}{\e^{(k\Delta - B)n - k\Delta}}, \quad \text{with} \quad A \defeq \int_0^T \lambda(s)ds < \infty.
\end{align*}
Set $k$ such that $A\left(\e - 1\right) + B - k \Delta < 0$, then \eqref{eq: borel-cantelli condition} holds, so $q_n(t, k \Delta + 1) = 0$ for all $t \in [0, T]$ and all large enough $n$ with probability one. Note that $q_n(t, (k + 1)\Delta) = 0$ for all $t \in [0, T]$ implies that the threshold can never reach $(k + 1)\Delta$ during $[0, T]$, so we may define $B_T \defeq (k + 1)\Delta$.
\end{proof}

The following lemma is a refinement of the functional strong law of large numbers for the Poisson process; a proof is provided in Section \ref{app: auxiliary results} of the appendix.

\begin{lemma}
	\label{lem: refined fslln for poisson process}
	Let $\calN$ be a Poisson process on $(\Omega, \calF, \prob)$ with unit rate. Then
	\begin{align*}
		\lim_{n \to \infty} \sup_{t \in [0, T]} n^\gamma\left|\frac{\calN(n t)}{n} - t\right| = 0 \quad \text{for all} \quad \gamma \in [0, 1 / 2) \quad \text{with probability one}.
	\end{align*}
\end{lemma}

The last two lemmas are used to prove the next proposition, which defines a set of probability one consisting of well-behaved sample paths.

\begin{proposition}
	\label{prop: definition of Omega_T}
	There exist a positive constant $B_T \in \N$ and a set of probability one $\Omega_T \subset \Gamma_T$ such that the following two properties hold.
	\begin{enumerate}
		\item[(a)] For each $\omega \in \Omega_T$, there exists $n_T(\omega)$ such that
		\begin{equation}
			\label{eq: boundedness of occupancy states and thresholds}
			q_n(\omega, t, B_T + 1) = 0 \quad \text{and} \quad \ell_n(\omega, t) \leq B_T \quad \text{for all} \quad t \in [0, T] \quad \text{and all} \quad n \geq n_T(\omega).
		\end{equation}
		
		\item[(b)] The next limits hold on $\Omega_T$ for all $\gamma \in [0, 1/2)$.
	\end{enumerate}
	\begin{subequations}
		\begin{align}
			&\lim_{n \to \infty} q_n(0) = q(0), \label{seq: convergence of initial conditions 2} \\
			&\lim_{n \to \infty} \sup_{t \in [0, T]} n^\gamma\left|\frac{1}{n} \calN_n^\lambda(t) - \int_0^t \lambda(s)ds\right| = 0, \label{seq: fslln for arrivals 2} \\
			&\lim_{n \to \infty} \sup_{t \in [0, \mu iT]} n^\gamma\left|\frac{1}{n} \calN_i(nt) - t\right| = 0 \quad \text{for all} \quad i \geq 1. \label{seq: fslln for departures 2}
		\end{align}
	\end{subequations}
\end{proposition}

\begin{proof}
By \eqref{ass: convergence of initial conditions}, we have \eqref{seq: convergence of initial conditions 2} on $\Gamma_T$. Also, it follows from Lemma \ref{lem: upper bound for number of tasks and threshold} that there exists a positive constant $B_T \in \N$ such that property (a) holds on a subset of $\Gamma_T$ which has probability one. By Lemma \ref{lem: refined fslln for poisson process}, this subset can be chosen so that \eqref{seq: fslln for arrivals 2} and \eqref{seq: fslln for departures 2} hold.
\end{proof}

\subsubsection{Properties of total and tail mass processes}
\label{subsub: properties of total and tail mass functions}

Below we prove the aforementioned asymptotic dynamical properties of the total and tail mass processes, starting with the following proposition; the proof is given in Section \ref{app: auxiliary results} of the appendix.

\begin{proposition}
	\label{prop: total number of tasks}
	For each $\omega \in \Omega_T$, the sequence $\set{u_n(\omega)}{n \geq 1}$ converges uniformly over $[0, T]$ to the function $\map{u}{[0, T]}{[0, \infty)}$ defined in \eqref{eq: total number of tasks} by
	\begin{equation*}
		u(0) \defeq \sum_{i = 1}^\infty q(0, i) \quad \text{and} \quad u(t) \defeq u(0)\e^{-\mu t} + \int_0^t \rho(s) \mu \e^{-\mu(t - s)}ds \quad \text{for all} \quad t \in [0, T].
	\end{equation*}
\end{proposition}

As mentioned in Section \ref{sec: notation and assumptions}, the last proposition is simply the functional law of large numbers for an infinite-server system. However, it is not straightforward that this law holds with probability one under the coupled construction of sample paths adopted in Section \ref{subsub: coupled construction of sample paths exponential case}; Proposition \ref{prop: total number of tasks} establishes this fact. The next proposition provides an asymptotic upper bound for certain tail mass processes, under specific conditions concerning the occupancy and threshold processes.

\begin{proposition}
	\label{prop: tail number of tasks}
	Suppose that the following two conditions hold for a given $\omega \in \Omega_T$ and a given increasing sequence $\calK$ of natural numbers.
	\begin{enumerate}
		\item[(a)] The sequence $\set{q_k(\omega)}{k \in \calK}$ converges to a function $q \in D_{\R^\N}[0, T]$ in the metric $\varrho$.
		
		\item[(b)] There exist $j \in \set{i\Delta}{i \geq 1}$ and $0 \leq t_0 < t_1 \leq T$ such that
		\begin{equation*}
			\ell_k(\omega, t) \leq j \quad \text{and} \quad q_k(\omega, t, j) < 1 \quad \text{for all} \quad t \in [t_0, t_1] \quad \text{and all} \quad k \in \calK.
		\end{equation*}
	\end{enumerate}
	Then the coordinate functions $q(i)$ are differentiable on $(t_0, t_1)$ for all $i > j$, and they satisfy
	\begin{equation*}
		\dot{q}(t, i) = -\mu i \left[q(t, i) - q(t, i + 1)\right] \quad \text{for all} \quad t \in (t_0, t_1) \quad \text{and all} \quad i > j.
	\end{equation*}
	Also, $\set{v_k(\omega, j + 1)}{k \in \calK}$ converges uniformly over $[0, T]$ to a function $v(j + 1)$ such that
	\begin{equation*}
		v(t, j + 1) < u(t_0)\e^{-\mu (t - t_0)} \quad \text{for all} \quad t \in [t_0, t_1].
	\end{equation*} 
\end{proposition}

\begin{proof}
It follows from (b) that $\calA_k(t, i) = \calA_k(t_0, i)$ for all $t \in [t_0, t_1]$ and all $i > j$, because tasks are only assigned to server pools with strictly fewer than $j$ tasks during $[t_0, t_1]$ in all the systems corresponding to the sequence $\calK$. Thus, \eqref{seq: occupancy state} implies that
\begin{equation}
	\label{eq: decomposition with no arrivals}
	q_k(t, i) = q_k(t_0, i) - \left[\calD_k(t, i) - \calD_k(t_0, i)\right] \quad \text{for all} \quad t \in [t_0, t_1] \quad \text{and all} \quad i > j.
\end{equation}

As indicated in Section \ref{app: relative compactness} of the appendix, convergence of a sequence of functions with respect to $\varrho$ implies uniform convergence over $[0, T]$ of the coordinate functions. Hence, (a) and \eqref{seq: fslln for departures 2} yield
\begin{equation}
	\label{eq: limit of departure processes}
	\lim_{k \to \infty} \sup_{t \in [0, T]} \left|\calD_k(t, i) - \int_0^t \mu i \left[q(i, s) - q(i, s + 1)\right]ds\right| = 0 \quad \text{for all} \quad i \geq 0.
\end{equation}
By Proposition \ref{prop: definition of Omega_T}, there exists $B_T > 0$ such that $q_k(i)$ is identically zero for all $i > B_T$ and all large enough $k \in \calK$; this implies, in particular, that $q(i)$ is identically zero as well for all $i > B_T$. Therefore, we conclude from (a), \eqref{eq: decomposition with no arrivals} and \eqref{eq: limit of departure processes} that
\begin{align*}
	\begin{array}{ll}
		\displaystyle q(t, i) = q(t_0, i) - \int_{t_0}^t \mu i\left[q(s, i) - q(s, i + 1)\right]ds & \quad \text{for} \quad j < i < B_T, \\ [3ex]
		\displaystyle q(t, i) = q(t_0, i) - \int_{t_0}^t \mu iq(s, i)ds & \quad \text{for} \quad i = B_T.
	\end{array}
\end{align*}
Proceeding by backward induction, starting from $i = B_T$, it is possible to conclude that $q(i)$ is differentiable on $(t_0, t_1)$ for all $j < i \leq B_T$, and such that $\dot{q}(i) = -\mu i \left[q(i) - q(i + 1)\right]$. As observed, $q(i)$ is identically zero for  all $i > B_T$, so this proves the first claim of the proposition.

The convergence of $\set{v_k(j + 1)}{k \in \calK}$ follows from (a) and the observation that $q_k(i)$ is identically zero for all $i > B_T$ and all large enough $k \in \calK$. Moreover, $v(j + 1) = \sum_{i = j + 1}^{B_T} q(i)$, and in particular, $v(j + 1)$ is differentiable on $(t_0, t_1)$. It follows that
\begin{equation*}
	\dot{v}(t, j + 1) = \sum_{i = j + 1}^{B_T} -\mu i \left[q(t, i) - q(t, i + 1)\right] = -\mu v(t, j + 1) - \mu j q(t, j + 1) \quad \text{for all} \quad t \in (t_0, t_1).
\end{equation*}
Note that $q(j + 1)$ is a non-negative function, thus
\begin{equation*}
	v(t, j + 1) \leq v(t_0, j + 1) \e^{-\mu(t - t_0)} < u(t_0) \e^{-\mu(t - t_0)} \quad \text{for all} \quad t \in [t_0, t_1].
\end{equation*}
The last inequality is strict since $j \geq 1$.
\end{proof}

\subsection{Evolution of the threshold}
\label{sub: evolution of the threshold}

In this section we prove Theorem \ref{the: main theorem exponential case}. For this purpose, we fix an integer $m \geq 0$ and an interval $[a, b] \subset [0, T]$. As in the statement of Theorem \ref{the: main theorem exponential case}, we assume that the control parameters $\alpha_n$ satisfy \eqref{ass: conditions on alpha}, that the offered load is $(m, \Delta)$-bounded on $[a, b]$ and that the length of this interval is strictly larger than $\sigma(a, b, m, \Delta)$.

The first step towards establishing Theorem \ref{the: main theorem exponential case} is to prove that all large enough constants $\sigma$ have the next property: the threshold is smaller than or equal to $m \Delta$ along the interval $[a + \sigma, b]$ in all large enough systems with probability one. This step is carried out in Section \ref{subsub: preliminary results}, which also provides an upper bound for the infimum of the constants $\sigma$ having the latter property, and in addition establishes \eqref{seq: main result 3}. The next step towards proving Theorem \ref{the: main theorem exponential case} is to demonstrate that $\sigma > \sigma(a, b, m , \Delta)$ implies that the threshold is in fact equal to $m \Delta$ along the interval $[a + \sigma, b]$ in all sufficiently large systems with probability one, which is done in Section \ref{subsub: proof of main result exponential case}; the latter statement corresponds to \eqref{seq: main result 1}, and \eqref{seq: main result 2} is also proven in Section \ref{subsub: proof of main result exponential case}.

\subsubsection{Preliminary results}
\label{subsub: preliminary results}

As mentioned above, we begin by establishing that there exists a constant $\sigma$ such that the threshold is upper bounded by $m \Delta$ along the interval $[a + \sigma, b]$ in all sufficiently large systems with probability one. This is done in the following proposition, which also provides an upper bound for the infimum of the constants $\sigma$ having the latter property.

\begin{proposition}
	\label{prop: bounded threshold}
	Suppose the hypotheses of Theorem \ref{the: main theorem exponential case} hold, and fix a constant
	\begin{align*}
		\sigma > \sigma_\bd(a, b, m, \Delta) \defeq \begin{cases}
			\frac{1}{\mu}\left[\log\left(\frac{u(a) - \rho_{\max}}{(m + 1) \Delta - \rho_{\max}}\right)\right]^+	&	\text{if} \quad u(a) > \rho_{\max}, \\
			0	&	\text{if} \quad u(a) \leq \rho_{\max},
		\end{cases}
	\end{align*}
	with $\rho_{\max}$ as in Definition \ref{def: (Delta, m)-bounded}. For each $\omega \in \Omega_T$, there exists $n_\bd(\omega)$ such that
	\begin{align*}
		\ell_n(\omega, t) \leq m \Delta \quad \text{for all} \quad t \in [a + \sigma, b] \quad \text{and all} \quad n \geq n_\bd(\omega).
	\end{align*}
\end{proposition}

\begin{proof}
It follows from \eqref{eq: total number of tasks} that the next inequality holds for all $t \in [a , b]$:
\begin{equation}
	\label{eq: upper bound for u}
	\begin{split}
		u(t) &= u(a)\e^{-\mu(t - a)} + \int_a^t \rho(s) \mu \e^{-\mu(t - s)}ds \\
		&\leq u(a)\e^{-\mu(t - a)} + \int_a^t \rho_{\max} \mu \e^{-\mu(t - s)}ds = \rho_{\max} + \left[u(a) - \rho_{\max}\right]\e^{-\mu(t - a)}.
	\end{split}
\end{equation}
We fix an arbitrary $\omega \in \Omega_T$, which is omitted from the notation for brevity. By Definition~\ref{def: (Delta, m)-bounded}, we may consider positive constants $\delta$ and $\varepsilon$ such that
\begin{align*}
	\frac{\rho_{\max}}{(m + 1) \Delta} < 1 - \delta = \frac{\rho_{\max} + \varepsilon}{(m + 1) \Delta}.
\end{align*}
Note that $\varepsilon$ can be expressed in terms of $\delta$. In other words, if $\delta$ is given, then $\varepsilon$ is determined. For each positive $\delta$, satisfying the previous inequality, and each $\theta \in (0, 1)$, we define
\begin{align*}
	\sigma(\delta, \theta) &\defeq \min \set{s \geq 0}{\left[u(a) - \rho_{\max}\right]\e^{-\mu s} \leq \theta \varepsilon} \\
	&= \begin{cases}
		\frac{1}{\mu}\left[\log\left(\frac{u(a) - \rho_{\max}}{\theta\left[(1 - \delta)(m + 1) \Delta - \rho_{\max}\right]}\right)\right]^+	&	\text{if} \quad u(a) > \rho_{\max}, \\
		0	&	\text{if} \quad u(a) \leq \rho_{\max}.
	\end{cases}
\end{align*}
Observe that $\sigma_\bd(a, b, m, \Delta)$ is the infimum of $\delta + \sigma(\delta, \theta)$. Therefore, we can choose $\delta$ and $\theta$ such that $\sigma > \delta + \sigma(\delta, \theta)$. Below we assume that $\delta$ and $\theta$ enforce the latter condition.

By Proposition \ref{prop: definition of Omega_T} and \eqref{ass: conditions on alpha}, there exists $n_0$ such that the following statements hold for all $n \geq n_0$.
\begin{enumerate}
	\item[(i)] $\ell_n(t) \leq B_T$ for all $t \in [0, T]$.
	
	\item[(ii)] $\calN_n^\lambda$ has at least $B_T - m \Delta$ jumps on each subinterval of $[a, b]$ of length at least $\delta$; recall that the load is $\Delta$-bounded in $[a, b]$ and thus $\lambda$ is lower bounded by a positive constant in $[a, b]$.
	
	\item[(iii)] $1 - \delta < \alpha_n < 1 - 1 / n$.
\end{enumerate}
Furthermore, it follows from Proposition \ref{prop: total number of tasks} that there exists $n_\bd \geq n_0$ such that
\begin{equation*}
	\sup_{t \in [0, T]} \left|u_n(t) - u(t)\right| \leq (1 - \theta)\varepsilon \quad \text{for all} \quad t \in [0, T] \quad \text{and all} \quad n \geq n_\bd.
\end{equation*}

Suppose that $j \geq (m + 1) \Delta$ and $n \geq n_\bd$. If $t \in [a, b]$, then
\begin{align*}
	q_n(t, j) \leq \frac{1}{j} \sum_{i = 1}^j q_n(t, i) \leq \frac{u_n(t)}{j} &\leq \frac{u(t) + (1 - \theta)\varepsilon}{(m + 1) \Delta} \\ 
	&\leq \frac{\rho_{\max} + \left[u(a) - \rho_{\max}\right]\e^{-\mu(t - a)} + (1 - \theta)\varepsilon}{(m + 1) \Delta}.
\end{align*}
The first inequality follows from the fact that $q_n(t) \in Q$ is a non-increasing sequence, and the last inequality follows from \eqref{eq: upper bound for u}. If $t \in [a + \sigma(\delta, \theta), b]$, then $t - a \geq \sigma(\delta, \theta)$ and thus
\begin{align*}
	&q_n(t, j) \leq \frac{\rho_{\max} + \varepsilon}{(m + 1) \Delta} = 1 - \delta < \alpha_n < 1 - \frac{1}{n} \\
	&\text{for all} \quad j \geq (m + 1) \Delta, \quad t \in [a + \sigma(\delta, \theta), b], \quad n \geq n_\bd.
\end{align*}

The third inequality implies that the number of server pools with at least $j \geq (m + 1)\Delta$ tasks is strictly smaller than $n - 1$. This implies that threshold increases cannot occur along the interval $[a + \sigma(\delta, \theta), b]$ provided that $n \geq n_{\bd}$ and $\ell_n \geq m\Delta$, because a threshold increase requires that the number of server pools with at least $h_n$ tasks is at least $n - 1$. Moreover, if $\ell_n(t) \geq (m + 1) \Delta$ for $t \in [a + \sigma(\delta, \theta), b]$ and $n \geq n_{\bd}$, then $q_n(t, \ell_n(t)) <\alpha_n$, which implies that the threshold will decrease with the next arrival. Therefore, in this case the threshold decreases with each arrival until it reaches $m \Delta$, and this occurs in at most $\delta$ units of time by (i) and (ii). Combining the above remarks, we conclude that $\ell_n(t) \leq m\Delta$ for all $t \in [a + \sigma , b]$ and $n \geq n_{\bd}$.
\end{proof}

The next result is a consequence of Propositions \ref{prop: tail number of tasks} and \ref{prop: bounded threshold}, and it implies \eqref{seq: main result 3}.

\begin{corollary}
	\label{cor: decay of the tail}
	Suppose that the assumptions of Theorem \ref{the: main theorem exponential case} hold, and fix a constant $\sigma > \sigma_\bd(a, b, m, \Delta)$. Then we have
	\begin{equation*}
		\limsup_{n \to \infty} \sup_{t \in [a + \sigma, b]} v_n(\omega, t, (m + 1)\Delta + 1)e^{\mu\left[t - (a + \sigma)\right]} \leq u(a + \sigma) \quad \text{for all} \quad \omega \in \Omega_T.
	\end{equation*}
\end{corollary}

\begin{proof}
We fix an arbitrary $\omega \in \Omega_T$. For brevity, $\omega$ is omitted from the notation and we set $r \defeq (m + 1) \Delta$. Suppose that the statement of the corollary does not hold. Then there exist $\varepsilon > 0$ and an increasing sequence of natural numbers $\calK$ such that
\begin{align*}
	\sup_{t \in [a + \sigma, b]} v_k(t, r + 1)e^{\mu\left[t - (a + \sigma)\right]} > u(a + \sigma) + \varepsilon \quad \text{for all} \quad k \in \calK.
\end{align*}

By Proposition \ref{prop: relative compactness of occupancy state sample paths}, we may assume without loss of generality that $\set{q_k}{k \in \calK}$ has a limit with respect to $\varrho$; this may require to replace $\calK$ by a subsequence. Furthermore, by Proposition \ref{prop: bounded threshold} we may also assume that $\ell_k(t) \leq m\Delta$ for all $t \in [a + \sigma, b]$ and all $k \in \calK$.

Note that $\ell_k(t) \leq m \Delta$ for all $t \in [a + \sigma, b]$ implies that $q_k(t, r) \leq 1 - 1 / k$ for all $t \in [a + \sigma, b]$. Indeed, no arrival can result in $q_k(r)$ going from $1 - 1 / k$ to $1$ because such an arrival would also result in the threshold increasing to $r = (m + 1) \Delta$. Consequently, Proposition \ref{prop: tail number of tasks} holds with $j = r$ along the interval $[a + \sigma, b]$, and in particular, there exists a function $v(r + 1)$ such that
\begin{align*}
	&v(t, r + 1) < u(a + \sigma)e^{-\mu\left[t - (a + \sigma)\right]} \quad \text{for all} \quad t \in [a + \sigma, b], \\
	&\lim_{k \to \infty} \sup_{t \in [a + \sigma, b]} \left|v_k(t, r + 1) - v(t, r + 1)\right| = 0.
\end{align*}
This leads to a contradiction, so the statement of the corollary must hold.
\end{proof}

\subsubsection{Proof of Theorem \ref{the: main theorem exponential case}}
\label{subsub: proof of main result exponential case}

We have already established that $\sigma > \sigma_\bd(a, b, m , \Delta)$ implies that the threshold is upper bounded by $m \Delta$ along the interval $[a + \sigma, b]$ in all large enough systems with probability one, and we have shown that \eqref{seq: main result 3} holds. It remains to prove that $\sigma > \sigma(a, b, m , \Delta)$ implies that the threshold is in fact equal to $m \Delta$ in all sufficiently large systems with probability one, and to establish \eqref{seq: main result 2}. To this end, we introduce the processes $\delta_n(j)$, defined as
\begin{equation}
	\label{eq: error process}
	\delta_n(t, j) \defeq \frac{1}{n}\calN_n^\lambda(t) - \int_0^t \lambda(s)ds - \sum_{i = 1}^{j} \left[\calD_n(t, i) - \int_0^t \mu i \left[q_n(s, i) - q_n(s, i + 1)\right]ds\right]  
\end{equation}
for all $t \in [0, T]$ and all $j \geq 1$. Observe that \eqref{seq: fslln for arrivals 2} and \eqref{seq: fslln for departures 2} imply that
\begin{equation}
	\label{eq: error process goes to zero}
	\lim_{n \to \infty} \sup_{t \in [0, T]} n^\gamma |\delta_n(\omega, t, j)| = 0 \quad \text{for all} \quad j \geq 1, \quad \gamma \in [0, 1 / 2) \quad \text{and} \quad \omega \in \Omega_T.
\end{equation}

The following two technical lemmas make use of the latter definition and property.

\begin{lemma}
	\label{lem: lemma 1 for main result}
	Fix $\omega \in \Omega_T$ and $j \in \set{i\Delta}{i \geq 1}$. Suppose that there exist an increasing sequence $\calK$ of natural numbers and two sequences of random times $0 \leq \tau_{k, 1} \leq \tau_{k, 2} \leq T$ such that
	\begin{equation*}
		\ell_k(\omega, t) \leq j \quad \text{and} \quad q_k(\omega, t, j) < 1 \quad \text{for all} \quad t \in \left[\tau_{k, 1}(\omega), \tau_{k, 2}(\omega)\right) \quad \text{and} \quad k \in \calK.
	\end{equation*}
	Then the next inequality holds for all $t \in \left[\tau_{k, 1}(\omega), \tau_{k, 2}(\omega)\right]$ and all $k \in \calK$:
	\begin{equation*}
		\sum_{i = 1}^j q_k(\omega, t, i) - \sum_{i = 1}^j q_k\left(\omega, \tau_{k, 1}(\omega), i\right) \geq \int_{\tau_{k, 1}(\omega)}^t \left[\lambda(s) - \mu j\right]ds - 2\sup_{s \in [0, T]} \left|\delta_k(\omega, s, j)\right|.
	\end{equation*}
\end{lemma}

\begin{proof}
We omit $\omega$ from the notation for brevity. It follows from \eqref{seq: occupancy state} that
\begin{equation*}
	\sum_{i = 1}^j q_k(t, i) - \sum_{i = 1}^j q_k\left(\tau_{k, 1}, i\right) = \sum_{i = 1}^j \left[\calA_k(t, i) - \calA_k\left(\tau_{k, 1}, i\right)\right] - \sum_{i = 1}^j \left[\calD_k(t, i) - \calD_k\left(\tau_{k, 1}, i\right)\right].
\end{equation*}
Note that all tasks are assigned to server pools with at most $j - 1$ tasks during $(\tau_{k, 1}, t]$, so the first term on the right-hand side can be expressed as follows:
\begin{equation*}
	\sum_{i = 1}^j \left[\calA_k(t, i) - \calA_k\left(\tau_{k, 1}, i\right)\right] = \frac{1}{k}\left[\calN_k^\lambda(t) - \calN_k^\lambda\left(\tau_{k, 1}\right)\right].
\end{equation*}

Using the process $\delta_k(j)$ defined in \eqref{eq: error process}, we may write
\begin{align*}
	\sum_{i = 1}^j q_k(t, i) - \sum_{i = 1}^j q_k\left(\tau_{k, 1}, i\right) &= \int_{\tau_{k, 1}}^t \lambda(s)ds \\
	&- \sum_{i = 1}^j \int_{\tau_{k, 1}}^t \mu i \left[q_k(s, i) - q_k(s, i + 1)\right]ds + \delta_k(t, j) - \delta_k(\tau_{k, 1}, j) \\
	&\geq \int_{\tau_{k, 1}}^t \left[\lambda(s) - \mu j\right]ds - 2\sup_{s \in [0, T]} \left|\delta_k(s, j)\right|.
\end{align*}
The last inequality follows from the fact that
\begin{equation*}
	\sum_{i = 1}^j i\left[q_k(s, i) - q_k(s, i + 1)\right] = \sum_{i = 1}^j q_k(s, i) - jq_k(s, j + 1) \leq j \quad \text{if} \quad s \in [0, T] \quad \text{and} \quad k \in \calK.
\end{equation*}
This completes the proof.
\end{proof}

\begin{lemma}
	\label{lem: lemma 2 for main result}
	Suppose that the hypotheses of Theorem \ref{the: main theorem exponential case} hold. Fix $\omega \in \Omega_T$ and $j \in \set{i\Delta}{0 \leq i \leq m}$. Also, assume that there exist an increasing sequence $\calK \subset \N$, a constant $\varepsilon > 0$ and two sequences of random times $a + \varepsilon \leq \zeta_{k, 1} \leq \zeta_{k, 2} \leq b$ such that
	\begin{equation*}
		q_k\left(\omega, \zeta_{k, 1}(\omega), j\right) = 1 \quad \text{and} \quad \ell_k(\omega, t) \leq j \quad \text{for all} \quad t \in \left[\zeta_{k, 1}(\omega), \zeta_{k, 2}(\omega)\right] \quad \text{and} \quad k \in \calK.
	\end{equation*}
	For all sufficiently large $k \in \calK$ and all $\gamma \in [0, 1/2)$, we have
	\begin{equation*}
		\ell_k(\omega, t) = j \quad \text{and} \quad q_k(\omega, t, j) \geq 1 - k^{-\gamma} \quad \text{for all} \quad t \in \left[\zeta_{k, 1}(\omega), \zeta_{k, 2} (\omega)\right].
	\end{equation*}
\end{lemma}

\begin{proof}
We omit $\omega$ from the notation for brevity, and we assume that $j > 0$ since otherwise the statement holds trivially. Fix $\gamma \in [0, 1/2)$ and consider the random times defined as
\begin{align*}
	&\xi_{k, 2} \defeq \inf \set{t \geq \zeta_{k, 1}}{q_k(t, j) \leq \max\{\alpha_k, 1 - k^{-\gamma}\}}, \\
	&\xi_{k, 1} \defeq \sup \set{t \leq \xi_{k, 2}}{q_k(t, j) = 1}.
\end{align*}
Note that $\lambda$ is lower bounded by a positive constant in $[a, b]$ because the load is $\Delta$-bounded in $[a, b]$. Thus, it follows from Remark \ref{rem: well positioned threahold} that $q_k(t, h_k(t)) < 1$ for all $t \geq a + \varepsilon$ and all large enough $k$. Also, $q_k(\zeta_{k, 1}, j) = 1$ and $j$ is a multiple of $\Delta$, which implies that $h_k(\zeta_{k, 1}) > j$ and thus $\ell_k(\zeta_{k, 1}) = j$. We conclude that $\ell_k(t) = j$ for all $t \in \left[\zeta_{k, 1}, \zeta_{k, 2} \wedge \xi_{k, 2}\right]$ and all $k \in \calK$. Indeed, the threshold is upper bounded by $j$ along $\left[\zeta_{k, 1}, \zeta_{k, 2}\right]$ and the threshold can only decrease if $q_k(j) \leq \alpha_k$ right before an arrival epoch. Thus, it suffices to establish that $\xi_{k, 2} \geq \zeta_{k, 2}$ for all large enough $k \in \calK$.

In order to prove this, we show that
\begin{equation}
	\label{eq: statement to prove lemma}
	j - \sum_{i = 1}^j q_k(t, i) = \sum_{i = 1}^j q_k\left(\xi_{k, 1}^-, i\right) - \sum_{i = 1}^j q_k(t, i) < \min\left\{1 - \alpha_k, k^{-\gamma}\right\} \quad \text{if} \quad t \in \left[\xi_{k, 1}, \zeta_{k, 2} \wedge \xi_{k, 2}\right],
\end{equation}
for all sufficiently large $k \in \calK$. This statement implies that $q_k(t, j) > \max\{\alpha_k, 1 - k^{-\gamma}\}$ for all $t \in \left[\xi_{k, 1}, \zeta_{k, 2} \wedge \xi_{k, 2}\right]$.  By definition of $\xi_{k, 2}$, this in turn implies that $\xi_{k, 2} > \zeta_{k, 2}$.

To prove \eqref{eq: statement to prove lemma}, let $\tau_{k, 1} \defeq \zeta_{k, 2} \wedge \xi_{k, 1}$ and $\tau_{k, 2} \defeq \zeta_{k, 2} \wedge \xi_{k, 2}$. For all $t \in \left[\tau_{k, 1}, \tau_{k, 2}\right]$, we have
\begin{align*}
	\sum_{i = 1}^j q_k(t, i) - \sum_{i = 1}^j q_k\left(\tau_{k, 1}^-, i\right) &\geq \sum_{i = 1}^j q_k(t, i) - \sum_{i = 1}^j q_k\left(\tau_{k, 1}, i\right) - \frac{j}{k} \\
	&\geq \int_{\tau_{k, 1}}^t \left[\lambda(s) - \mu j\right]ds - 2\sup_{s \in [0, T]} \left|\delta_k(s, j)\right| - \frac{j}{k} \\
	&\geq - 2\sup_{s \in [0, T]} \left|\delta_k(s, j)\right| - \frac{j}{k}.
\end{align*}
For the first inequality, note that the processes $q_k(i)$ have jumps of size $1 / k$. The second inequality follows from Lemma \ref{lem: lemma 1 for main result}, since $\ell_k(t) = j$ and $q_k(t, j) < 1$ for all $t \in \left[\tau_{k, 1}, \tau_{k, 2}\right)$ and all $k \in \calK$. For the third inequality, recall that the offered load is $(m, \Delta)$-bounded on $[a, b]$, so $\lambda(s) \geq \mu j$ for all $s \in [a, b]$.

We conclude from \eqref{ass: conditions on alpha} and \eqref{eq: error process goes to zero} that
\begin{equation*}
	\lim_{k \to \infty} \frac{1}{1 - \alpha_k}\left[2\sup_{s \in [0, T]} |\delta_k(s, j)| + \frac{j}{k}\right] = \lim_{k \to \infty} \frac{1}{k^{\gamma_0}(1 - \alpha_k)}\left[2\sup_{s \in [0, T]} k^{\gamma_0}|\delta_k(s, j)| + \frac{j}{k^{1 - {\gamma_0}}}\right] = 0.
\end{equation*}
If $1 - \alpha_k$ is replaced by $k^{-\gamma}$ on the left-hand side, then the limit is also equal to zero by \eqref{eq: error process goes to zero}. Consequently, for all sufficiently large $k \in \calK$, we have
\begin{equation*}
	\sum_{i = 1}^j q_k(t, i) - \sum_{i = 1}^j q_k\left(\tau_{k, 1}^-, i\right) \geq - 2\sup_{s \in [0, T]} \left|\delta_k(s, j)\right| - \frac{j}{k} > - \min\left\{1 - \alpha_k, k^{-\gamma}\right\}
\end{equation*}
for all $t \in \left[\tau_{k, 1}, \tau_{k, 2}\right]$. This implies \eqref{eq: statement to prove lemma} for all large enough $k \in \calK$; note that \eqref{eq: statement to prove lemma} holds trivially when $\xi_{k, 1} > \zeta_{k, 2}$.
\end{proof}

\begin{proof}[Proof of Theorem \ref{the: main theorem exponential case}.]
As in the proof of Proposition \ref{prop: bounded threshold}, it follows from \eqref{eq: total number of tasks} that the following inequality holds for all times $a \leq a + \sigma_0 \leq t \leq b$:
\begin{equation}
	\label{eq: lower bound for u}
	\begin{split}
		u(t) &= u(a + \sigma_0)\e^{-\mu[t - (a + \sigma_0)]} + \int_{a + \sigma_0}^t \rho(s) \mu \e^{-\mu(t - s)}ds \\
		&\geq u(a + \sigma_0)\e^{-\mu[t - (a + \sigma_0)]} + \int_{a + \sigma_0}^t \rho_{\min} \mu \e^{-\mu(t - s)}ds \\
		&= \rho_{\min} + \left[u(a + \sigma_0) - \rho_{\min}\right]\e^{-\mu[t - (a + \sigma_0)]}.
	\end{split}
\end{equation}
In this equation, $\rho_{\min}$ is as in Definition \ref{def: (Delta, m)-bounded}. Fix $\sigma(a, b, m, \Delta) < \sigma < b - a$ as in the statement of the theorem, and for each $0 < \varepsilon < \rho_{\min} - m\Delta$ let
\begin{equation*}
	\sigma(\varepsilon) \defeq \min \set{s \geq 0}{\rho_{\min}\left(1 - e^{-\mu s}\right) - \varepsilon \geq m\Delta} = \frac{1}{\mu}\log \left(\frac{\rho_{\min}}{\rho_{\min} - m\Delta - \varepsilon}\right).
\end{equation*}
Recall the definitions of $\sigma(a, b, m, \Delta)$ and $\sigma_{bd}(a, b, m, \Delta)$ provided in Definition \ref{def: (Delta, m)-bounded} and Proposition~\ref{prop: bounded threshold}, respectively. We fix $\sigma_0 > \sigma_{\bd}(a, b, m, \Delta)$ and $\varepsilon > 0$ such that $\sigma = \sigma_0 + \sigma(\varepsilon) + \varepsilon$. This is possible since $\sigma_{\bd}(a, b, m, \Delta) + \sigma(\varepsilon) + \varepsilon$ decreases to $\sigma(a, b, m, \Delta)$ as $\varepsilon\downarrow0$.

We fix $\omega \in \Omega_T$, which is omitted from the notation, and we define the random times
\begin{equation*}
	\xi_n \defeq \inf \set{t \geq a + \sigma_0}{q_n(t, m \Delta) = 1}.
\end{equation*}
The proofs of \eqref{seq: main result 1} and \eqref{seq: main result 2} will be completed if we show that $\xi_n \leq a + \sigma$ for all large enough $n$. Indeed, if this claim is established, then \eqref{seq: main result 1} and \eqref{seq: main result 2} become a consequence of Lemma \ref{lem: lemma 2 for main result} with $j \defeq m\Delta$, $\zeta_{n, 1} \defeq \xi_n$ and $\zeta_{n, 2} \defeq b$. In order to see why, observe that $\ell_n(t) \leq m \Delta$ for all $t \in [a + \sigma_0, b]$ and all sufficiently large $n$ by Proposition \ref{prop: bounded threshold}, so the hypothesis of Lemma \ref{lem: lemma 2 for main result} holds.

To prove that $\xi_n \leq a + \sigma$ for all large enough $n$, we will establish that
\begin{equation}
	\label{eq: limsup of xi}
	\limsup_{n \to \infty} \xi_n \leq a + \sigma_0 + \sigma(\varepsilon) < a + \sigma_0 + \sigma(\varepsilon) + \varepsilon = a + \sigma.
\end{equation}
If $m = 0$, then $\xi_n = a + \sigma_0$ for all $n$ and the above inequality holds. Thus, we assume that $m \geq 1$. In order to prove \eqref{eq: limsup of xi}, we assume that \eqref{eq: limsup of xi} does not hold and we arrive to a contradiction.

Suppose then that \eqref{eq: limsup of xi} is false. This implies that there exists an increasing sequence $\calK$ of natural numbers such that $\xi_k > a + \sigma_0 + \sigma(\varepsilon)$ for all $k \in \calK$. Furthermore, by Propositions \ref{prop: relative compactness of occupancy state sample paths} and \ref{prop: bounded threshold}, this sequence may be chosen so that the next two properties hold.
\begin{enumerate}
	\item[(i)] $\set{q_k}{k \in \calK}$ converges to some function $q \in D_{\R^\N}[0, T]$ with respect to $\varrho$.
	
	\item[(ii)] $\ell_k(t) \leq m \Delta$ for all $t \in [a + \sigma_0, b]$ and all $k \in \calK$.
\end{enumerate}
The definition of $\xi_k$ implies that $q_k(t, m\Delta) < 1$ for all $t \in [a + \sigma_0, a + \sigma_0 + \sigma(\varepsilon)] \subset [a + \sigma_0, \xi_k)$. Properties (i) and (ii), together with the last observation, imply that the hypotheses of Proposition~\ref{prop: tail number of tasks} hold with $j \defeq m \Delta$, $t_0 \defeq a + \sigma_0$ and $t_1 \defeq a + \sigma_0 + \sigma(\varepsilon)$.

Let $v(m \Delta + 1)$ be the function defined in the statement of Proposition \ref{prop: tail number of tasks}, as the uniform limit of the processes $v_k(m \Delta + 1)$ over the interval $[0, T]$. It follows from Propositions \ref{prop: total number of tasks} and \ref{prop: tail number of tasks} that
\begin{equation*}
	\sup_{t \in [0, T]} \left|u_k(t) - v_k(t, m\Delta + 1) - \left[u(t) - v(t, m \Delta + 1)\right]\right| \leq \varepsilon
\end{equation*}
for all sufficiently large $k \in \calK$. For each of these $k \in \calK$, we have
\begin{align*}
	\sum_{i = 1}^{m \Delta} q_k\left(a + \sigma_0 + \sigma(\varepsilon), i\right) &= u_k\left(a + \sigma_0 + \sigma(\varepsilon)\right) - v_k(a + \sigma_0 + \sigma(\varepsilon), m \Delta + 1) \\
	&\geq u\left(a + \sigma_0 + \sigma(\varepsilon)\right) - v(a + \sigma_0 + \sigma(\varepsilon), m \Delta + 1) - \varepsilon \\
	&> \rho_{\min} + \left[u(a + \sigma_0) - \rho_{\min}\right]\e^{-\mu\sigma(\varepsilon)} - u(a + \sigma_0)\e^{-\mu\sigma(\varepsilon)} - \varepsilon \\
	&= \rho_{\min} \left(1 - \e^{-\mu\sigma(\varepsilon)}\right) - \varepsilon = m \Delta.
\end{align*}
The third inequality follows from \eqref{eq: lower bound for u} and Proposition \ref{prop: tail number of tasks}, and the last equality follows from the definition of $\sigma(\varepsilon)$. This is a contradiction, so we conclude that \eqref{eq: limsup of xi} holds, and this establishes \eqref{seq: main result 1} and \eqref{seq: main result 2}. Recall that \eqref{seq: main result 3} holds by Corollary \ref{cor: decay of the tail}, thus the proof of Theorem \ref{the: main theorem exponential case} is complete.
\end{proof}

\section{Coxian service times}
\label{sec: coxian service times}

In this section we prove Theorem \ref{the: main theorem coxian case}. For this purpose, we assume that the basic learning scheme is used, that $\lambda$ is constant over time and that service times are Coxian distributed. Below we proceed as indicated in Section \ref{sub: coxian service times}.

\subsection{Relative compactness of occupancy processes}
\label{sub: relative compactness coxian case}

In Section \ref{subsub: coupled construction of sample paths coxian case} we construct the sample paths of the occupancy states $s_n$ and the thresholds $\ell_n$ on a common probability space for all $n$. The sample paths of the occupancy states lie in the space $D_{\R^{\N^r}}[0, T]$ of all c\`adl\`ag functions on $[0, T]$ with values in $\R^{\N^r}$, equipped with the topology of uniform convergence. In Section \ref{subsub: relative compactness of sample paths coxian case} we establish that the sequence $\set{s_n}{n \geq 1}$ is almost surely relatively compact.

Before we continue, we introduce some notation and we define a metric on $D_{\R^{\N^r}}[0, T]$. A server pool is said to have occupancy profile $\nu = \left(\nu(1), \dots, \nu(r)\right) \in \N^r$ if it has exactly $\nu(m)$ tasks in phase $m$ for each $m$. Fix an enumeration $\set{\nu_i}{i \geq 0}$ of $\N^r$ and consider the metric $d$ defined by
\begin{equation*}
	d(x, y) \defeq \sum_{i = 0}^\infty \frac{\min \left\{|x(\nu_i) - y(\nu_i)|, 1\right\}}{2^i} \quad \text{for all} \quad x, y \in \R^{\N^r}.
\end{equation*}
The latter metric generates the product topology of $\R^{\N^r}$. The topology of uniform convergence equipped to $D_{\R^{\N^r}}[0, T]$ is generated by the following metric:
\begin{equation*}
	\varrho(x, y) \defeq \sup_{t \in [0, T]} d\left(x(t), y(t)\right) \quad \text{for all} \quad x, y \in D_{\R^{\N^r}}[0, T].
\end{equation*}
Note that the previous definitions generalize the ones introduced in Section \ref{sub: relative compactness exponential case}, where $r = 1$. Also, the specific choice of the enumeration $\set{\nu_i}{i \geq 0}$ is not particularly relevant since convergence with respect to $\varrho$ is equivalent to uniform convergence of the coordinate functions.

\subsubsection{Coupled construction of sample paths}
\label{subsub: coupled construction of sample paths coxian case}

The sample paths of the occupancy states and thresholds are defined as deterministic functions of the following stochastic primitives.

\begin{itemize}
	\item \textit{Driving Poisson processes:} a family $\set{\calN_{(\nu, i, j)}}{\nu \in \N^r,\ i \in \{0, \dots, r\}\ \text{and}\ j \in \{0, 1\}}$ of unit-rate independent Poisson processes, defined on a common probability space $(\Omega_D, \calF_D, \prob_D)$.
	
	\item \textit{Selection variables:} a sequence $\set{U_j}{j \geq 1}$ of independent and identically distributed uniform random variables with values on $[0, 1)$, defined on a common probability space $(\Omega_S, \calF_S, \prob_S)$.
	
	\item \textit{Initial conditions:} a family $\set{s_n(0)}{n \geq 1}$ of random variables describing the initial conditions of the systems, defined on a common probability space $(\Omega_I, \calF_I, \prob_I)$ and satisfying the assumptions introduced in Section \ref{sec: notation and assumptions}.
\end{itemize}

Consider the product probability space of $(\Omega_D, \calF_D, \prob_D)$, $(\Omega_S, \calF_S, \prob_S)$ and $(\Omega_I, \calF_I, \prob_I)$; denote its completion by $(\Omega, \calF, \prob)$. The occupancy states and thresholds are defined on the latter space from a set of equations involving the stochastic primitives. To write these equations, we introduce some notation which is analogous to that of Section \ref{subsub: coupled construction of sample paths exponential case}.

For each occupancy state $s \in S$, we define the intervals
\begin{equation*}
	I_i(s) \defeq \left[1 - \sum_{j = i}^\infty s\left(\nu_j\right), 1 - \sum_{j = i + 1}^\infty s\left(\nu_j\right)\right) \quad \text{for all} \quad i \geq 0.
\end{equation*}
These intervals form a partition of $[0, 1)$ such that the length of $I_i(s)$ is the fraction of server pools with occupancy profile $\nu_i$. Let $\Sigma_i \defeq \sum_{m = 1}^r \nu_i(m)$ denote the total number of tasks in a server pool with profile $\nu_i$, and define $q \in Q$ as in \eqref{eq: brief and detail descriptors}, but in terms of $s$. If $q(j) < 1$, then we let
\begin{equation*}
	J_i(s, j) \defeq \left[\frac{1 - \sum_{k = i}^\infty s\left(\nu_k\right) \ind{\Sigma_k < j}}{1 - q(j)}, \frac{1 - \sum_{k = i + 1}^\infty s\left(\nu_k\right)\ind{\Sigma_k < j}}{1 - q(j)}\right) \quad \text{for all} \quad i \geq 0.
\end{equation*}
These intervals yield another partition of $[0, 1)$. In this case, the length of $J_i(s, j)$ is the fraction of server pools with occupancy profile $\nu_i$ but only among those server pools with at most $j - 1$ tasks in total; we adopt the convention that $J_i(s, j) \defeq \emptyset$ for all $i \geq 0$ when $q(j) = 1$.

If $\ell \in \N$ represents the threshold and $q(\ell) < 1$, then the length of $J_i(s, \ell)$ is equal to the probability of picking a server pool with occupancy profile $\nu_i$ uniformly at random among those with strictly less than $\ell$ tasks in total. Similarly, if $h \defeq \ell + \Delta$, $q(\ell) = 1$ and  $q(h) < 1$, then the length of $J_i(s, h)$ is equal to the probability of picking a server pool with occupancy profile $\nu_i$ uniformly at random among those with at least $\ell$ and strictly fewer than $h$ tasks in total. We define
\begin{equation*}
	r_{ij}(s, \ell) \defeq \begin{cases}
		\ind{U_j \in J_i(s, \ell)} & \text{if} \quad \Sigma_i < \ell, \\
		\ind{q(\ell) = 1, U_j \in J_i(s, h)} & \text{if} \quad \ell \leq \Sigma_i < h, \\
		\ind{q(h) = 1, U_j \in I_i(s)} & \text{if} \quad \Sigma_i \geq h,
	\end{cases}
	\quad \text{for all} \quad i,j \geq 1.
\end{equation*}
Note that $r_{ij}(s, \ell) \in \{0, 1\}$ and that for a fixed $j$ there exists a unique $i$ such that $r_{ij}(s, \ell) = 1$. This family of random variables will be used to describe the dispatching decisions, as in Section~\ref{subsub: coupled construction of sample paths exponential case}. Specifically, if $s$ and $\ell$ are the occupancy state and the threshold, respectively, when the $j^{\text{th}}$ incoming task arrives, then this task is sent to a server pool with occupancy profile $\nu_i$ if and only if $r_{ij}(s, \ell) = 1$; this coincides with the dispatching rule described in Section \ref{sub: dispatching rule}.

We postulate that $\calN_n^\lambda(t) \defeq \calN_{(0, 0, 0)}(n\lambda t)$ is the number of tasks that arrive to the system with $n$ server pools during the interval $[0, t]$, we denote the jump times of $\calN_n^\lambda$ by $\set{\tau_{n,k}}{k \geq 1}$ and we define $\tau_{n,0} \defeq 0$. Note that $\calN_n^\lambda$ is a Poisson process with intensity $n\lambda$, as required by our model.

As in Section \ref{subsub: coupled construction of sample paths exponential case}, it is possible to construct c\`adl\`ag processes $s_n$ and $\ell_n$ such that the following equations hold on a set of probability one $\Gamma_0$ for all $n$ and all $t \in [0, T]$.
\begin{subequations}
	\begin{align}
		&s_n(t) = s_n(0) + \bar{\calA}_n(t) + \bar{\calP}_n(t) + \bar{\calD}_n(t), \label{seq: occupancy state coxian case}\\
		&\ell_n(t) = \floor{u_n(t)}_\Delta \label{seq: threshold coxian case}
	\end{align}
\end{subequations}
In the above equations, $u_n$ is defined in terms of $q_n$ as in \eqref{eq: total and tail mass processes}, whereas $q_n$ is defined in terms of $s_n$ using \eqref{eq: brief and detail descriptors}. In addition, the processes $\bar{\calA}_n$, $\bar{\calP}_n$ and $\bar{\calD}_n$ satisfy:
\begin{align*}
	&\bar{\calA}_n(t, \nu_i) = \calA_n(t, \nu_i - e_1) - \calA_n(t, \nu_i), \\
	&\bar{\calP}_n(t, \nu_i) = \sum_{m = 1}^{r - 1} \left[\calP_n(t, \nu_i + e_m - e_{m + 1}, m) - \calP_n(t, \nu_i, m)\right], \\
	&\bar{\calD}_n(t, \nu_i) = \sum_{m = 1}^r \left[\calD_n(t, \nu_i + e_m, m) -\calD_n(t, \nu_i, m)\right], \\
	&\calA_n(t, \nu_i) = \frac{1}{n} \sum_{j = 1}^{\calN_n^\lambda(t)} r_{ij}\left(s_n\left(\tau_{n, j}^-\right), \ell_n\left(\tau_{n, j}^-\right)\right), \\
	&\calP_n(t, \nu_i, m) = \frac{1}{n} \calN_{(\nu_i, m, 0)} \left(n\int_0^t p_m\mu_m\nu_i(m)s_n(\tau, \nu_i)d\tau\right), \\
	&\calD_n(t, \nu_i, m) = \frac{1}{n} \calN_{(\nu_i, m, 1)} \left(n\int_0^t (1 - p_m)\mu_m\nu_i(m)s_n(\tau, \nu_i)d\tau\right).
\end{align*}
Here $\{e_1, \dots, e_r\}$ denotes the canonical basis of $\R^r$. Also, $\calA_n(t, \nu)$ and $\calP_n(t, \nu, k)$ are defined as zero if the vector $\nu$ has a negative component. A coupled construction of all the above processes can be performed by forward induction on the jumps of the driving Poisson processes since the initial number of tasks in the system is almost surely finite for all $n$; this precludes an infinite number of events in finite time with probability one.

The above construction endows the processes $s_n$ and $\ell_n$ with the intended statistical behavior. In particular, the process $\calA_n(\nu_i)$ counts the arrivals to server pools with occupancy profile $\nu_i$, the process $\calP_n(\nu_i, m)$ counts transitions from phase $m$ to phase $m + 1$ of tasks in server pools with profile~$\nu_i$, and the process $\calD_n(\nu_i, m)$ counts departures of tasks in phase $m$ from server pools with occupancy profile $\nu_i$. Observe that \eqref{seq: occupancy state coxian case} is equivalent to \eqref{seq: occupancy state} if $r = 1$, whereas \eqref{seq: threshold coxian case} now corresponds to the basic learning scheme.

\subsubsection{Relative compactness of sample paths}
\label{subsub: relative compactness of sample paths coxian case}

Below we state the relative compactness result mentioned at the start of Section \ref{sub: relative compactness coxian case}. The proof essentially relies on the decomposition \eqref{seq: occupancy state coxian case} and the fact that the coordinate processes $\calA_n(\nu)$, $\calP_n(\nu, m)$ and $\calD_n(\nu, m)$ have $O(n)$ intensities for all $\nu$ and $m$. In particular, the proof is analogous to that of Proposition \ref{prop: relative compactness of occupancy state sample paths} and is thus omitted.

\begin{proposition}
	\label{prop: relative compactness of occupancy state sample paths coxian case}
	There exists a set of probability one $\Gamma_T \subset \Gamma_0$ with the next property. The sequences $\set{\calA_n(\omega)}{n \geq 1}$, $\set{\calP_n(\omega, m)}{n \geq 1}$, $\set{\calD_n(\omega, m)}{n \geq 1}$ and $\set{s_n(\omega)}{n \geq 1}$ are relatively compact with respect to $\varrho$ for each $m$ and each $\omega \in \Gamma_T$. Also, they have the property that the limit of every convergent subsequence is a function with Lipschitz coordinates. 
\end{proposition}

\subsection{Asymptotic dynamical properties}
\label{sub: asymptotic dynamical properties coxian case}

In this section we establish a few asymptotic dynamical properties concerning the following four processes:
\begin{equation}
	\label{eq: total and tail mass processes coxian case}
	u_n \defeq \sum_{i = 1}^\infty q_n(i), \quad x_n \defeq \sum_{i = 1}^{\floor{\rho}_\Delta} q_n(i), \quad y_n \defeq \sum_{i = \floor{\rho}_\Delta + 1}^{\floor{\rho}_\Delta + \Delta} q_n(i) \quad \text{and} \quad z_n \defeq \sum_{i = \floor{\rho}_\Delta + \Delta + 1}^\infty q_n(i).
\end{equation}
At any given time, the server pools can be arranged as in the diagram of Figure \ref{fig: state variables}, ignoring the phases of tasks. The mass processes $x_n$, $y_n$ and $z_n$ correspond to the number of tasks in the columns of the latter diagram covered by the respective summations in \eqref{eq: total and tail mass processes coxian case}. The total mass process $u_n$, considered in Section \ref{sub: asymptotic dynamical properties exponential case} as well, represents the total number of tasks in the system, normalized by~$n$, or equivalently, the total mass in the diagram of Figure \ref{fig: state variables}.

\subsubsection{Set of nice sample paths}
\label{subsub: set of nice sample paths coxian case}

As in Section \ref{subsub: set of nice sample paths}, we begin by introducing a set of probability one consisting of  well-behaved sample paths. For this purpose, we need the following result, counterpart of Lemma \ref{lem: upper bound for number of tasks and threshold}; the proof is provided in Section \ref{app: auxiliary results} of the appendix.

\begin{lemma}
	\label{lem: upper bound for number of tasks and threshold coxian case}
	There exists a positive constant $B_T \in \N$ such that $q_n(t, B_T + 1) = 0$ and $\ell_n(t) \leq B_T$ for all $t \in [0, T]$ and all sufficiently large $n$ with probability one.
\end{lemma}

The next proposition defines the aforementioned set of well-behaved sample paths.

\begin{proposition}
	\label{prop: definition of Omega_T coxian case}
	There exist a positive constant $B_T \in \N$ and a set of probability one $\Omega_T \subset \Gamma_T$ such that the following two properties hold.
	\begin{enumerate}
		\item[(a)] For each $\omega \in \Omega_T$, there exists $n_T(\omega)$ such that
		\begin{equation}
			\label{eq: boundedness of occupancy states and thresholds coxian case}
			q_n(\omega, t, B_T + 1) = 0 \quad \text{and} \quad \ell_n(\omega, t) \leq B_T \quad \text{for all} \quad t \in [0, T] \quad \text{and all} \quad n \geq n_T(\omega).
		\end{equation}
		
		\item[(b)] The following limits hold on $\Omega_T$:
	\end{enumerate}
	\begin{subequations}
		\begin{align}
			&\lim_{n \to \infty} s_n(0) = s(0), \label{seq: convergence of initial conditions 2 coxian case} \\
			&\lim_{n \to \infty} \sup_{t \in [0, T]} \left|\frac{1}{n} \calN_n^\lambda(t) - \lambda t\right| = 0, \label{seq: fslln for arrivals 2 coxian case} \\
			&\lim_{n \to \infty} \sup_{t \in [0, \mu_i \nu(i) T]} \left|\frac{1}{n} \calN_{(\nu, i, j)}(nt) - t\right| = 0 \quad \text{for all} \quad (\nu, i, j) \in \N^r \times \{1, \dots, r\} \times \{0, 1\}. \label{seq: fslln for departures 2 coxian case}
		\end{align}
	\end{subequations}
\end{proposition}

\begin{proof}
It follows from Lemma \ref{lem: upper bound for number of tasks and threshold coxian case} that there exists a positive $B_T \in \N$ such that property (a) holds on a set $\Omega_T \subset \Gamma_T$ of probability one. By \eqref{ass: convergence of initial conditions} and the strong law of large numbers for the Poisson process, the latter set can be chosen so that property (b) holds as well.
\end{proof}

\subsubsection{Properties of mass processes}
\label{subsub: properties of mass processes}

Next we prove the asymptotic dynamical properties of the mass processes. The proof of the next proposition provided in Section \ref{app: auxiliary results} of the appendix.

\begin{proposition}
	\label{prop: total number of tasks coxian case}
	For each $\omega \in \Omega_T$, the sequence $\set{u_n(\omega)}{n \geq 1}$ converges uniformly over $[0, T]$ to the function $\map{u}{[0, T]}{[0, \infty)}$ defined in \eqref{eq: general total number of tasks} by
	\begin{equation*}
		u^m(0) \defeq \sum_{i = 0}^\infty \nu_i(m)s(0, \nu_i) \quad \text{and} \quad u(t) \defeq \sum_{m = 1}^r u^m(0)G_m(t) + \int_0^t \lambda G_1(t - s)ds.
	\end{equation*}
\end{proposition}

While the above law of large numbers is known to hold weakly, and even for more general service times, it is not immediate that it holds with probability one under the coupled construction of sample paths adopted in Section \ref{subsub: coupled construction of sample paths coxian case}; Proposition \ref{prop: total number of tasks coxian case} establishes this fact.

The following proposition pertains to the dynamics of the mass processes $x_n$, $y_n$ and $z_n$ along intervals of time where $\ell_n = \ell \defeq \floor{\rho}_\Delta$. Before stating it formally, we provide some intuition using the diagram of Figure \ref{fig: state variables}. Namely, at any given time, the server pools, represented by the rows in said diagram, can be arranged monotonically with respect to the total number of tasks that they currently have. The tasks on each row may be arranged in an arbitrary manner, ignoring the phase in which they are, but regardless of how tasks are arranged within each row, every task may leave the system at a rate that is at least
\begin{equation*}
	\eta \defeq \min \set{(1 - p_m)\mu_m}{1 \leq m \leq r} > 0.
\end{equation*}
If $\ell_n = \ell$, then at least one of the first $\ell + \Delta$ columns of the diagram is not full, which implies that new tasks are always dispatched to one of these columns. Similarly, tasks are sent to one of the first $\ell$ columns if $x_n < \ell = \ell_n$. Recall that $z_n$ represents the number of tasks outside of the first $\ell + \Delta$ columns, hence $z_n$ decreases at a rate larger than $\eta z_n$ while $\ell_n = \ell$. Likewise, $y_n$ decreases at a rate that is lower bounded by $\eta \left[y_n - \Delta q_n(\ell + \Delta + 1)\right] \geq \eta \left(y_n - \Delta z_n\right)$ while $x_n < \ell = \ell_n$.

\begin{proposition}
	\label{prop: mass processes coxian case}
	Let us fix $\omega \in \Omega_T$ and an increasing sequence $\calK$ of natural numbers, in the sequel we omit $\omega$ for brevity. Suppose that the following conditions hold.
	\begin{enumerate}
		\item[(a)] The sequence $\set{s_k}{k \in \calK}$ converges to a function $s \in D_{\R^{\N^r}}[0, T]$ with respect to $\varrho$.
		
		\item[(b)] There exist $0 \leq t_0 < t_1 \leq T$ such that
		\begin{equation*}
			\ell_k(t) = \floor{\rho}_\Delta \quad \text{for all} \quad t \in [t_0, t_1] \quad \text{and all} \quad k \in \calK.
		\end{equation*}
	\end{enumerate}
	Then the sequences $\set{x_k}{k \in \calK}$, $\set{y_k}{k \in \calK}$ and $\set{z_k}{k \in \calK}$ converge to absolutely continuous functions $x$, $y$ and $z$, respectively, uniformly over $[0, T]$. Also, the next properties hold for $t \in (t_0, t_1)$.
	\begin{enumerate}
		\item[(i)] If $x(t) < \floor{\rho}_\Delta$, then $\dot{y} \leq -\eta (y - \Delta z)$ almost everywhere on a neighborhood of $t$.
		
		\item[(ii)] Moreover, $\dot{z} \leq -\eta z$ almost everywhere on $[t_0, t_1]$.
	\end{enumerate}
\end{proposition}

\begin{proof}
The uniform convergence of $\set{x_k}{k \in \calK}$, $\set{y_k}{k \in \calK}$ and $\set{z_k}{k \in \calK}$ follows from (a) and \eqref{eq: boundedness of occupancy states and thresholds coxian case}; the latter equation implies that $s_k(\nu_i)$ is identically zero if $\Sigma_i > B_T$ for all large enough $k \in \calK$. Furthermore, it follows from Proposition \ref{prop: relative compactness of occupancy state sample paths coxian case} that $s(\nu_i)$ is Lipschitz continuous for all $i$, therefore $x$, $y$ and $z$ are absolutely continuous, in fact Lipschitz.

Assume that $x(t) < \floor{\rho}_\Delta$ for some $t \in (t_0, t_1)$ and write $\ell \defeq \floor{\rho}_\Delta$ for brevity. Since $x$ is continuous and $x_k$ converges uniformly to $x$, there exist $a < t < b$ such that
\begin{equation}
	\label{eq: hypothesis of (i)}
	x_k(\tau) < \ell = \ell_k(\tau) \quad \text{for all} \quad \tau \in [a, b] \subset (t_0, t_1) \quad \text{and all large enough} \quad k \in \calK.
\end{equation}
Below we establish (i) by proving that
\begin{equation*}
	y(\tau) - y(a) \leq -\int_a^\tau \eta \left[y(\zeta) - \Delta z(\zeta)\right]d\zeta \quad \text{for all} \quad \tau \in [a, b].
\end{equation*}

Suppose that $\tau$ and $k$ are such that \eqref{eq: boundedness of occupancy states and thresholds coxian case} and \eqref{eq: hypothesis of (i)} hold. By \eqref{eq: brief and detail descriptors} and \eqref{seq: occupancy state coxian case},
\begin{align*}
	y_k(\tau) - y_k(a) &= \sum_{j = \ell + 1}^{\ell + \Delta} \left[q_k(\tau, j) - q_k(a, j)\right] \\
	&= \sum_{j = \ell + 1}^{\ell + \Delta} \sum_{\Sigma_i \geq j} \left[s_k(\tau, \nu_i) - s_k(a, \nu_i)\right] = \sum_{j = \ell + 1}^{\ell + \Delta} \sum_{\sigma = j}^\infty \sum_{\Sigma_i = \sigma} \left[\bar{\calD}_k(\tau, \nu_i) - \bar{\calD}_k(a, \nu_i)\right].
\end{align*}
For the last identity, first observe that $\calA_k(\tau, \nu_i) - \calA_k(a, \nu_i) = 0$ if $\Sigma_i \geq \ell$ since $q_k(\ell) < 1$ along the interval $[a, b]$ by \eqref{eq: hypothesis of (i)}. In addition, note that
\begin{equation*}
	\sum_{\Sigma_i = \sigma} \bar{\calP}_k(\nu_i) = \sum_{\Sigma_i = \sigma} \sum_{m = 1}^{r - 1} \left[\calP_k(\nu_i + e_m - e_{m + 1}, m) - \calP_k(\nu_i, m)\right] = 0 \quad \text{for all} \quad \sigma \geq 0.
\end{equation*}

Now observe that
\begin{align*}
	\sum_{\Sigma_i = \sigma} \bar{\calD}_k(\nu_i) &= \sum_{\Sigma_i = \sigma} \sum_{m = 1}^r \left[\calD_k(\nu_i + e_m, m) - \calD_m(\nu_i, m)\right] \\
	&= \sum_{\Sigma_i = \sigma + 1} \sum_{m = 1}^r \calD_k(\nu_i, m) - \sum_{\Sigma_i = \sigma} \sum_{m = 1}^r \calD_k(\nu_i, m).
\end{align*}
As a result, we may write
\begin{subequations}
	\begin{align}
		&y_k(\tau) - y_k(a) = \sum_{j = \ell + 1}^{\ell + \Delta} \sum_{\sigma = j}^\infty \left[D_k(\sigma + 1) - D_k(\sigma)\right] = - \sum_{j = \ell + 1}^{\ell + \Delta} D_k(j), \label{seq: expression for yk}\\
		&D_k(\sigma) \defeq \sum_{\Sigma_i = \sigma} \sum_{m = 1}^r \left[\calD_k(\tau, \nu_i, m) - \calD_k(a, \nu_i, m)\right].
	\end{align}
\end{subequations}
Here the last identity in \eqref{seq: expression for yk} uses \eqref{eq: boundedness of occupancy states and thresholds coxian case}, which implies that $D_k(\sigma) = 0$ if $\sigma > B_T$.

Using \eqref{seq: fslln for departures 2 coxian case}, we conclude that
\begin{equation*}
	\lim_{k \to \infty} D_k(j) = \sum_{\Sigma_i = j} \sum_{m = 1}^r \int_a^\tau (1 - p_m)\mu_m\nu_i(m)s(\zeta, \nu_i)d\zeta \geq \int_a^\tau \sum_{\Sigma_i = j} \eta j s(\zeta, \nu_i)d\zeta.
\end{equation*}
Therefore, taking limits on both sides of \eqref{seq: expression for yk}, we obtain
\begin{equation*}
	y(\tau) - y(a) \leq - \int_a^\tau \sum_{j = \ell + 1}^{\ell + \Delta} \sum_{\Sigma_i = j} \eta j s(\zeta, \nu_i)d\zeta \leq - \int_a^\tau \eta \left[y(\zeta) - \Delta z(\zeta)\right] d\zeta.
\end{equation*}
For the last inequality, note that
\begin{align*}
	\sum_{j = \ell + 1}^{\ell + \Delta} \sum_{\Sigma_i = j} js_k(\nu_i) &= \sum_{j = \ell + 1}^{\ell + \Delta} j\left[q_k(j) - q_k(j + 1)\right] \\
	&= y_k + \ell q_k(\ell + 1) - (\ell + \Delta) q_k(\ell + \Delta + 1) \\
	&\geq y_k - \Delta q_k(\ell + \Delta + 1) \geq y_k - \Delta z_k.
\end{align*}
Clearly, the inequality between the extremes also holds if $s_k$, $y_k$ and $z_k$ are replaced by their limits $s$, $y$ and $z$. This completes the proof (i).

The proof of (ii) follows from similar arguments. Namely, $\ell_k = \ell$ implies that $u_k < \ell + \Delta$ and therefore $q_k(\ell + \Delta) < 1$. Consequently, $\calA_k(b, \nu_i) - \calA_k(a, \nu_i) = 0$ for all $t_0 \leq a < b \leq t_1$ and all $k \in \calK$ if $\Sigma_i \geq \ell + \Delta$. Then one may proceed as in the proof of (i).
\end{proof}

\subsection{Proof of Theorem \ref{the: main theorem coxian case}}
\label{sub: proof of the main result coxian case}

In this section we use the asymptotic dynamical properties of the mass processes to complete the proof of Theorem \ref{the: main theorem coxian case}.

\begin{proof}[Proof of Theorem \ref{the: main theorem coxian case}.]
Let $\ell \defeq \floor{\rho}_\Delta$ for brevity, and choose $\varepsilon > 0$ such that $\ell + 3\varepsilon < \rho < \ell + \Delta$. It follows from \eqref{eq: general total number of tasks} that $u$ is differentiable. Furthermore, $u(t) \to \rho$ and $\dot{u}(t) \to 0$ as $t \to \infty$. Therefore, there exists a time $t_0 \geq 0$ such that
\begin{equation}
	\label{eq: conditions on u}
	\ell + 3\varepsilon < u(t) < \ell + \Delta \quad \text{and} \quad \dot{u}(t) > -\eta\varepsilon \quad \text{for all} \quad t \in [t_0, T].
\end{equation}
The previous statement is trivial if $t_0 > T$, but if $T$ is large enough, then the statement holds for some $t_0 < T$. Let us fix some $\omega \in \Omega_T$, which we omit from the notation for brevity. Proposition \ref{prop: total number of tasks coxian case} states that $u_n$ converges uniformly over $[0, T]$ to $u$. Therefore, $\ell < u_n(t) < \ell + \Delta$, and thus $\ell_n(t) = \ell$, for all $t \in [t_0, T]$ and all large enough $n$. In particular, this proves \eqref{seq: main result coxian 1}.

Fix an increasing sequence of natural numbers. By Proposition \ref{prop: relative compactness of occupancy state sample paths coxian case}, there exists a subsequence $\calK$ such that $\set{s_k}{k \in \calK}$ converges to a function $s \in D_{\R^{\N^r}}[0, T]$ with respect to $\varrho$. Also, it follows from Proposition \ref{prop: mass processes coxian case} that $\set{x_k}{k \in \calK}$, $\set{y_k}{k \in \calK}$ and $\set{z_k}{k \in \calK}$ converge uniformly over $[0, T]$ to absolutely continuous functions $x$, $y$ and $z$, respectively. Moreover,
\begin{equation}
	\label{eq: upper bound for z}
	z(t) \leq z(t_0)\e^{-\eta(t - t_0)} \leq u(t_0)\e^{-\eta(t - t_0)} \quad \text{for all} \quad t \in [t_0, T]
\end{equation}
by property (ii) of Proposition \ref{prop: mass processes coxian case}. In particular,
\begin{equation*}
	z(t) < \frac{\varepsilon}{\Delta + 1} \quad \text{for all} \quad t \in [t_1, T], \quad \text{where} \quad t_1 \defeq t_0 + \frac{1}{\eta}\left[\log\left(\frac{(\Delta + 1)u(t_0)}{\varepsilon}\right)\right]^+.
\end{equation*}

Suppose that $x(t) < \ell$ for some $t \in (t_1, T)$, then the next statements hold almost everywhere on a neighborhood of $t$. Since $y = u - x - z > u - \ell - z$, Proposition \ref{prop: mass processes coxian case} and \eqref{eq: conditions on u} imply that
\begin{equation*}
	\dot{y} \leq -\eta \left(y - \Delta z\right) < -\eta \left[u - \ell - (\Delta + 1)z\right] < -2\eta\varepsilon.
\end{equation*}
As a result, we have $\dot{x} = \dot{u} - \dot{y} - \dot{z} \geq \dot{u} - \dot{y} > \eta\varepsilon$, where the last inequality uses \eqref{eq: conditions on u}. Thus,
\begin{equation}
	\label{eq: limit of x}
	x(t) = \ell \quad \text{for all} \quad t \in [\sigma, T] \quad \text{with} \quad \sigma \defeq t_1 + \frac{\ell}{\eta\varepsilon}.
\end{equation}
Here $\sigma$ is an upper bound for the time until $x$ reaches $\ell$ if $x(t_1) < \ell$, computed from $\dot{x} > \eta\varepsilon$. Also, $x$ cannot decrease once it reaches $\ell$ since $\dot{x} > \eta\varepsilon$ almost everywhere on $\set{t \in (t_1, T)}{x(t) < \ell}$.

We conclude from \eqref{eq: upper bound for z} and \eqref{eq: limit of x} that \eqref{seq: main result coxian 2} and \eqref{seq: main result coxian 3} hold with $c = u(t_0)$ along the sequence $\calK$. In fact, \eqref{seq: main result coxian 2} and \eqref{seq: main result coxian 3} hold without restricting the corresponding limits to the sequence $\calK$ since $t_0$, $t_1$, $\sigma$ and $c$ do not depend on $\calK$, they only depend on $\rho$, $\Delta$, $\eta$, $\varepsilon$ and $u$. Specifically, if \eqref{seq: main result coxian 2} and \eqref{seq: main result coxian 3} do not hold, then one may construct an increasing sequence of natural numbers which violates the latter equations. Nevertheless, reasoning as above, one obtains a subsequence $\calK$ along which \eqref{seq: main result coxian 2} and \eqref{seq: main result coxian 3} hold, which leads to a contradiction.
\end{proof}

% Appendix here
% Options are (1) APPENDIX (with or without general title) or 
%             (2) APPENDICES (if it has more than one unrelated sections)
% Outcomment the appropriate case if necessary
%
% \begin{APPENDIX}{<Title of the Appendix>}
	% \end{APPENDIX}
%
%   or 
%
% \begin{APPENDICES}
	% \section{<Title of Section A>}
	% \section{<Title of Section B>}
	% etc
	% \end{APPENDICES}

\begin{appendices}
	
	Section \ref{app: relative compactness} of this appendix contains the proof of Proposition \ref{prop: relative compactness of occupancy state sample paths}, whereas Section \ref{app: auxiliary results} contains the proofs of several auxiliary results used throughout the paper.
	
	\section{Relative compactness of sample paths}
	\label{app: relative compactness}
	
	In this section we use a technique developed in \cite{bramson1998state} to demonstrate that there exists a set of probablity one $\Gamma_T$ where $\set{q_n}{n \geq 1}$ is relatively compact with respect to $\varrho$. First we define the set $\Gamma_T$ through the following proposition.
	
	\begin{proposition}
		\label{prop: definition of Gamma_T}
		There exists a set of probability one $\Gamma_T \subset \Gamma_0$ where:
		\begin{subequations}
			\begin{align}
				&\lim_{n \to \infty} q_n(0) = q(0), \label{seq: convergence of initial conditions} \\
				&\lim_{n \to \infty} \sup_{t \in [0, T]} \left|\frac{1}{n} \calN_n^\lambda(t) - \int_0^t \lambda(s)ds\right| = 0, \label{seq: fslln for arrivals} \\
				&\lim_{n \to \infty} \sup_{t \in [0, \mu iT]} \left|\frac{1}{n} \calN_i(nt) - t\right| = 0 \quad \text{for all} \quad i \geq 1. \label{seq: fslln for departures}
			\end{align}
		\end{subequations}
	\end{proposition}
	
	\begin{proof}
	This result is a straightforward consequence of \eqref{ass: convergence of initial conditions} and the functional strong law of large numbers for the Poisson process. The proof of \eqref{seq: fslln for arrivals} also relies on the fact that $\lambda$ is bounded.
	\end{proof}
	
	Let us fix an arbitrary $\omega \in \Gamma_T$, which is omitted for brevity. We will show that the sequences of sample paths $\set{\calA_n}{n \geq 1}$ and $\set{\calD_n}{n \geq 1}$ are relatively compact subsets of $D_{\R^\N}[0, T]$ and that the limit of every convergent subsequence is a function with Lipschitz coordinates. It will then follow from \eqref{seq: occupancy state} and \eqref{seq: convergence of initial conditions} that $\set{q_n}{n \geq 1}$ has the same properties.
	
	The next characterization of relative compactness with respect to $\varrho$ will be useful. Let $D[0, T]$ denote the space of real c\`adl\`ag functions on $[0, T]$, endowed with the uniform norm:
	\begin{equation*}
		\norm{x}_T = \sup_{t \in [0, T]} |x(t)| \quad \text{for all} \quad x \in D[0, T].
	\end{equation*}
	Note that a sequence of functions $x_n \in D_{\R^\N}[0, T]$ converges to a function $x \in D_{\R^\N}[0, T]$ with respect to $\varrho$ if and only if $x_n(i)$ converges to $x(i)$ with respect to $\norm{\scdot}_T$ for all $i \geq 0$.
	
	\begin{proposition}
		\label{prop: relative compactness from coordinates}
		A sequence $\set{x_n}{n \geq 1} \subset D_{\R^\N}[0, T]$ is relatively compact if and only if the sequences $\set{x_n(i)}{n \geq 1} \subset D[0, T]$ are relatively compact for all $i \geq 0$.
	\end{proposition}
	
	\begin{proof}
	We only need to prove the converse, so assume that the sequences $\set{x_n(i)}{n \geq 1}$ are relatively compact for all $i \geq 0$. Given an increasing sequence $\calK \subset \N$, we must show that there exists a subsequence of $\set{x_k}{k \in \calK}$ which converges with respect to $\varrho$. For this purpose, we may construct a family of increasing sequences $\set{\calJ_i}{i \geq 0}$ with the following properties.
	\begin{enumerate}
		\item[(i)] $\calJ_{i + 1} \subset \calJ_i \subset \calK$ for all $i \geq 0$.
		
		\item[(ii)] $\set{x_j(i)}{j \in \calJ_i}$ has a limit $x(i) \in D[0, T]$ for each $i \geq 0$.
	\end{enumerate}
	
	Consider the sequence $\set{k_j}{j \geq 1} \subset \calK$ such that $k_j$ is the $j^{\text{th}}$ element of $\calJ_j$. Then
	\begin{equation*}
		\lim_{j \to \infty} \norm{x_{k_j}(i) - x(i)}_T = 0 \quad \text{for all} \quad i \geq 0.
	\end{equation*} 
	Define $x \in D_{\R^\N}[0, T]$ so that its coordinates are the functions $x(i)$ introduced in~(ii). Then $x_{k_j}$ converges to $x$ with respect to $\varrho$, which completes the proof.
	\end{proof}
	
	As a result, it suffices to prove that the sequences $\set{\calA_n(i)}{n \geq 1}$ and $\set{\calD_n(i)}{n \geq 1}$ are relatively compact in $D[0, T]$ for all $i \geq 0$. Consider the sets
	\begin{equation*}
		L_M \defeq \set{x \in D[0, T]}{x(0) = 0\ \text{and}\ |x(t) - x(s)| \leq M|t - s|\ \text{for all}\ s, t \in [0, T]},
	\end{equation*}
	which are compact by the Arzel\'a-Ascoli theorem. For each $i \geq 0$, we prove that there exists $M_i > 0$ such that $\calA_n(i)$ and $\calD_n(i)$ approach $L_{M_i}$ as $n$ grows to infinity. Then we leverage the compactness of $L_{M_i}$ to show that the sequences $\set{\calA_n(i)}{n \geq 1}$ and $\set{\calD_n(i)}{n \geq 1}$ are relatively compact with respect to $\norm{\scdot}_T$. For this purpose, we introduce the spaces
	\begin{equation*}
		L_M^\varepsilon \defeq \set{x \in D[0, T]}{x(0) = 0\ \text{and}\ |x(t) - x(s)| \leq M|t - s| + \varepsilon\ \text{for all}\ s, t \in [0, T]}.
	\end{equation*}
	
	\begin{lemma}
		\label{lem: lemma from bramson}
		If $x \in L_M^\varepsilon$, then there exists $y \in L_M$ such that $\norm{x - y}_T \leq 4\varepsilon$.
	\end{lemma}
	
	The previous lemma is a restatement of \cite[Lemma 4.2]{bramson1998state}, and together with the next lemma, it implies that for each $i \geq 0$ there exists $M_i > 0$ such that $\calA_n(i)$ and $\calD_n(i)$ approach $L_{M_i}$.
	
	\begin{lemma}
		\label{lem: approximate lipschitz property}
		For all $i \geq 0$, there exist $M_i > 0$, and a vanishing sequence $\set{\varepsilon_{i, n} > 0}{n \geq 1}$, such that the functions $\calA_n(i)$ and $\calD_n(i)$ lie in $L_{M_i}^{\varepsilon_{i, n}}$ for all $n$.
	\end{lemma}
	
	\begin{proof}
	Let us fix $i \geq 1$, the statement holds trivially if $i = 0$. Recall that $\lambda$ is a bounded function, thus there exists $K > 0$ such that $|\lambda(t)| \leq K$ for all $t \in [0, T]$. We have
	\begin{align*}
		\left|\calA_n(t, i) - \calA_n(s, i)\right| &\leq \frac{1}{n} \left|\calN_n^\lambda(t) - \calN_n^\lambda(s)\right| \\
		&\leq \left|\int_s^t \lambda(u)du\right| + 2\sup_{r \in [0, T]} \left|\frac{1}{n}\calN_n^\lambda(r) - \int_0^r \lambda(v)dv\right| \\
		&\leq K|t - s| + 2\sup_{r \in [0, T]} \left|\frac{1}{n}\calN_n^\lambda(r) - \int_0^r \lambda(v)dv\right| \quad \text{for all} \quad s, t \in [0, T].
	\end{align*}
	By Proposition \ref{prop: definition of Gamma_T}, there exists a vanishing sequence $\set{\delta_{i, n}^1 > 0}{n \geq 1}$ such that
	\begin{align*}
		\left|\calA_n(t, i) - \calA_n(s, i)\right| \leq K|t - s| + \delta_{i, n}^1 \quad \text{for all} \quad s, t \in [0, T].
	\end{align*}
	
	For each $t \in [0, T]$, let us define
	\begin{align*}
		f_n(t, i) = \int_0^t \mu i\left[q_n(s, i) - q_n(s, i + 1)\right]ds,
	\end{align*}
	which is a non-decreasing function. Also, $f_n(0, i) = 0$ and $|f_n(t, i) - f_n(s, i)| \leq \mu i|t - s|$ for all $s, t \in [0, T]$. In particular, $f_n(T, i) \leq \mu i T$. Using these properties, we conclude that
	\begin{align*}
		\left|\calD_n(t, i) - \calD_n(s, i)\right| &\leq \frac{1}{n} \left|\calN_i\left(n f_n(t, i)\right) - \calN_i\left(n f_n(s, i)\right)\right| \\
		&\leq |f_n(t, i) - f_n(s, i)| + 2\sup_{r \in [0, T]} \left|\frac{1}{n}\calN_i\left(n f_n(r, i)\right) - f_n(r, i)\right| \\
		&\leq \mu i|t - s| + 2\sup_{r \in [0, \mu iT]} \left|\frac{1}{n}\calN_i(n r) - r\right| \quad \text{for all} \quad s, t \in [0, T].
	\end{align*}
	It follows from \eqref{seq: fslln for departures} that there exists a vanishing sequence $\set{\delta_{i, n}^2 > 0}{n \geq 1}$ such that
	\begin{align*}
		\left|\calD_n(t, i) - \calD_n(s, i)\right| \leq \mu i|t - s| + \delta_{i, n}^2 \quad \text{for all} \quad s, t \in [0, T].
	\end{align*}
	The proof is completed setting $M_i \defeq \max\{K, \mu i\}$ and $\varepsilon_{i, n} \defeq \max\{\delta_{i, n}^1, \delta_{i, n}^2\}$.
	\end{proof}
	
	We now establish the relative compactness of the occupancy state sample paths.
	
	\begin{proof}[Proof of Proposition \ref{prop: relative compactness of occupancy state sample paths}.]
	As above, we fix an arbitrary $\omega \in \Gamma_T$ and we omit it from the notation for brevity. It follows from \eqref{seq: occupancy state} and \eqref{seq: convergence of initial conditions} that it is enough to establish the result for the sequences $\set{\calA_n}{n \geq 1}$ and $\set{\calD_n}{n \geq 1}$. For this it suffices to show that $\set{\calA_n(i)}{n \geq 1}$ and $\set{\calD_n(i)}{n \geq 1}$ are relatively compact subsets of $D[0, T]$ for all $i \geq 0$, such that the limits of their convergent subsequences are Lipschitz; the sufficiency was explained above. We fix $i \geq 0$ and we do this just for $\set{\calA_n(i)}{n \geq 1}$ since the same arguments can be used if $\calA_n$ is replaced by $\calD_n$.
	
	Let $M_i$ and $\set{\varepsilon_{i, n}}{n \geq 1}$ be as in the statement of Lemma \ref{lem: approximate lipschitz property}. It follows from Lemma \ref{lem: lemma from bramson} that for each $n$ there exists $x_n(i) \in L_{M_i}$ such that $\norm{\calA_n(i) - x_n(i)}_T \leq 4\varepsilon_{i, n}$. Recall that $L_{M_i}$ is compact, thus every increasing sequence of natural numbers has a subsequence $\calK$ such that $\set{x_k(i)}{k \in \calK}$ converges to some function $x \in L_{M_i}$. Also, note that
	\begin{equation*}
		\limsup_{k \to \infty} \norm{\calA_k(i) - x(i)}_T \leq \lim_{k \to \infty} \norm{\calA_k(i) - x_k(i)}_T + \lim_{k \to \infty} \norm{x_k(i) - x(i)}_T = 0;
	\end{equation*}
	the limits are taken along $\calK$. This proves that every subsequence of $\set{\calA_n(i)}{n \geq 1}$ has a further subsequence which converges to a Lipschitz function.
	\end{proof}
	
	\section{Auxiliary results}
	\label{app: auxiliary results}
	
	Below are the proofs of some auxiliary results.
	
	\begin{proof}[Proof of Lemma \ref{lem: refined fslln for poisson process}.]
	It is enough to establish that
	\begin{equation*}
		\bigcap_{m \geq 1} \bigcup_{n \geq m} \set{\omega \in \Omega}{\sup_{t \in [0, T]} n^\gamma\left|\frac{\calN(\omega, n t)}{n} - t\right| \geq \varepsilon}
	\end{equation*}
	has probability zero for each $\gamma \in [0, 1/2)$ and each $\varepsilon > 0$. Let us fix $\gamma$ and $\varepsilon$, by the Borel-Cantelli lemma, it suffices to prove that
	\begin{equation*}
		\sum_{n = 1}^\infty \prob\left(\sup_{t \in [0, T]} \left|\calN(n t) - nt\right| \geq \varepsilon n^{1 - \gamma}\right) < \infty \quad \text{for all} \quad \varepsilon > 0.
	\end{equation*}
	
	For each $k \geq 1$, the process $\left[\calN(t) - t\right]^{2k}$ is a submartingale. It follows from Doob's inequality that
	\begin{equation*}
		\prob\left(\sup_{t \in [0, T]} \left|\calN(n t) - nt\right| \geq \varepsilon n^{1 - \gamma}\right) \leq \frac{\mu_{2k}(n T)}{\varepsilon^{2k}n^{2k(1 - \gamma)}},
	\end{equation*}
	where $\mu_i(x)$ denotes the $i^{\text{th}}$ central moment of the Poisson distribution with mean $x$. Below we prove that there exists $k$ such that the sum over $n$ of the expression on the right-hand side is finite.
	
	The central moments $\mu_i(x)$ are polynomials of $x$ such that
	\begin{equation*}
		\mu_1(x) = 0, \quad \mu_2(x) = x \quad \text{and} \quad \mu_{i + 1}(x) = x\left[\mu_i'(x) + i\mu_{i - 1}(x)\right] \quad \text{for all} \quad i \geq 2;
	\end{equation*}
	for the last identity see \cite[Equation (5.22)]{kendall1945advanced}. This recursion implies that $\mu_{2i}$ and $\mu_{2i + 1}$ have degree~$i$. Therefore, there exist coefficients $a_{k, j}$ such that
	\begin{equation*}
		\mu_{2k}(n T) = \sum_{j = 0}^k a_{k, j} n^j.
	\end{equation*}
	
	Fix $k$ such that $k - 2k \gamma > 1$ and observe that
	\begin{equation*}
		\frac{\mu_{2k}(n T)}{\varepsilon^{2k}n^{2k(1 - \gamma)}} = \sum_{j = 0}^k \frac{a_{k, j}}{\varepsilon^{2k}} \left(\frac{1}{n}\right)^{2k - j - 2k\gamma}.
	\end{equation*}
	To conclude, note that $2k - j - 2k\gamma \geq k - 2k\gamma > 1$ for all $0 \leq j \leq k$. Hence,
	\begin{align*}
		\sum_{n = 1}^\infty \frac{\mu_{2k}(n T)}{\varepsilon^{2k}n^{2k(1 - \gamma)}} \leq \sum_{n = 1}^\infty \left(\sum_{j = 0}^k \frac{a_{k, j}}{\varepsilon^{2k}}\right) \left(\frac{1}{n}\right)^{k - 2k\gamma} < \infty.
	\end{align*}
	This completes the proof.
	\end{proof}
	
	\begin{proof}[Proof of Proposition \ref{prop: total number of tasks}.]
	We fix an arbitrary $\omega \in \Omega_T$, which is omitted from the notation for brevity. It follows from Proposition \ref{prop: definition of Omega_T} that $q_n(i)$, $\calA_n(i)$ and $\calD_n(i)$ are identically zero for all large enough $n$ and all $i > B_T' \defeq B_T + \Delta$\footnote{If $B_T' = B_T$, a task can be sent to a server pool with $B_T$ tasks and at the same instant a task can leave from the same server pool, although this has probability zero; here $q_n(B_T + 1)$ remains zero but $\calA_n(B_T + 1)$ and $\calD_n(B_T + 1)$ have a jump. Setting $B_T' = B_T + \Delta$ precludes arrivals to server pools with $B_T'$ tasks since $\ell_n \leq B_T$.}. Consequently, we have
	\begin{align*}
		u_n(t) = \sum_{i = 1}^{B_T'} q_n(t, i) &= u_n(0) + \sum_{i = 1}^{B_T'} \left[\calA_n(t, i) - \calD_n(t, i)\right] \\
		&= u_n(0) + \frac{1}{n}\calN_n^\lambda(t) - \sum_{i = 1}^{B_T'} \calD_n(t, i)
	\end{align*}
	for all $t \in [0, T]$ and all large enough $n$.
	
	To prove the proposition, it suffices to establish that every increasing sequence of natural numbers has a subsequence $\calK$ such that $\set{u_k}{k \in \calK}$ converges uniformly over $[0, T]$ to $u$. For this purpose, we first observe that \eqref{seq: convergence of initial conditions 2} and \eqref{seq: fslln for arrivals 2} imply that
	\begin{equation*}
		\lim_{n \to \infty} u_n(0) = u(0) \quad \text{and} \quad \lim_{n \to \infty} \sup_{t \in [0, T]} \left|\frac{1}{n}\calN_n^\lambda(t) - \int_0^t \lambda(s)ds\right| = 0.
	\end{equation*}
	Given an increasing sequence of natural numbers, Proposition \ref{prop: relative compactness of occupancy state sample paths} implies that there exist some subsequence $\calK$ and a function $q \in D_{\R^\N}[0, T]$, such that $\set{q_k}{k \in \calK}$ converges to $q$ with respect to $\varrho$; note that $q(i)$ is identically zero for all $i > B_T'$. If we define $\tilde{u} \defeq \sum_{i = 1}^{B_T'} q(i)$, then $\set{u_k}{k \in \calK}$ converges to $\tilde{u}$ uniformly over $[0, T]$, and \eqref{seq: fslln for departures 2} yields
	\begin{align*}
		&\lim_{k \to \infty} \sup_{t \in [0, T]} \left|\sum_{i = 1}^{B_T'} \calD_k(t, i) - \int_0^t \mu \tilde{u}(s)ds\right| = \\
		&\lim_{k \to \infty} \sup_{t \in [0, T]} \left|\sum_{i = 1}^{B_T'} \left[\calD_k(t, i) - \int_0^t \mu i \left[q(s, i) - q(s, i + 1)\right]ds\right]\right| = 0.
	\end{align*}
	We conclude that $\map{\tilde{u}}{[0, T]}{[0, \infty)}$ satisfies the integral equation
	\begin{equation*}
		\tilde{u}(t) = u(0) + \int_0^t \left[\lambda(s) - \mu\tilde{u}(s)\right]ds \quad \text{for all} \quad t \in [0, T].
	\end{equation*}
	The unique solution to this integral equation is $u$. Therefore, the sequence $\set{u_n}{n \geq 1}$ converges uniformly over $[0, T]$ to $u$.
	\end{proof}
	
	\begin{proof}[Proof of Lemma \ref{lem: upper bound for number of tasks and threshold coxian case}.]
	It suffices to establish that there exists $k \in \N$ such that $u_n(t) < (k + 1)\Delta$ for all $t \in [0, T]$ and all sufficiently large $n$. Specifically, if $u_n(t) < (k + 1) \Delta$, then it follows from \eqref{seq: threshold coxian case} that $\ell_n(t) \leq k\Delta$, and clearly $q_n(t, (k + 1)\Delta) < 1$. In particular, incoming tasks are never sent to server pools with $(k + 1) \Delta$ tasks or more during $[0, T]$. Also, recall from Section \ref{sec: notation and assumptions} that $q_n(0, B + 1) = 0$ for all $n$ with probability one, so we can take $B_T = \max \{B, (k + 1)\Delta\}$.
	
	Therefore, it is enough to show that
	\begin{equation*}
		\bigcap_{m \geq 1} \bigcup_{n \geq m} \set{\omega \in \Omega}{u_n(\omega, t) \geq M\ \text{for some}\ t \in [0, T]}
	\end{equation*}
	has probability zero for some $M \in \N$. Consider the sets
	\begin{equation*}
		E_n^M \defeq \set{\omega \in \Omega}{u_n(\omega, t) \geq M\ \text{for some}\ t \in [0, T]}.
	\end{equation*}
	By the Borel-Cantelli lemma, it suffices to prove that there exists $M \in \N$ such that
	\begin{equation}
		\label{eq: borel-cantelli condition coxian case}
		\sum_{n = 1}^\infty \prob\left(E_n^M\right) < \infty.
	\end{equation}
	
	The total number of tasks in the system at time $t$ is upper bounded by $\calN_n^\lambda(t) + nu_n(0)$, the number of arrivals during $[0, t]$ plus the initial number of tasks in the system. Hence,
	\begin{align*}
		\prob\left(E_n^M\right) &\leq \prob \left(\calN_n^\lambda(t) + nu_n(0) \geq nM \ \text{for some} \ t \in [0, T]\right) \\
		&\leq \prob \left(\calN_n^\lambda(t) + n B \geq nM \ \text{for some} \ t \in [0, T]\right) = \prob \left(\calN_n^\lambda(T) \geq n\left(M - B\right) \right).
	\end{align*}
	The second inequality follows from $q_n(0, B + 1) = 0$, which implies that $u_n(0) \leq B$.
	
	Applying a Chernoff bound, we conclude that
	\begin{align*}
		\prob\left(E_n^M\right) \leq \prob\left(\calN_n^\lambda(T) \geq n(M - B) \right) \leq \frac{\e^{\lambda T \left(\e - 1\right)n}}{\e^{(M - B)n}} = \e^{\left[\lambda T(\e - 1) + B - M\right]n}.
	\end{align*}
	Therefore, \eqref{eq: borel-cantelli condition coxian case} holds for any $M > \lambda T (\e - 1) + B$.
	\end{proof}
	
	\begin{proof}[Proof of Proposition \ref{prop: total number of tasks coxian case}.]
	We fix an arbitrary $\omega \in \Omega_T$, which is omitted from the notation for brevity. It follows from Proposition \ref{prop: definition of Omega_T coxian case} that $s_n(\nu_i)$, $\calA_n(\nu_i)$, $\calP_n(\nu_i, m)$ and $\calD_n(\nu_i, m)$ are identically zero for all large enough $n$ whenever $\Sigma_i > B_T' \defeq B_T + \Delta$\footnote{If $B_T' = B_T$, a task can be sent to a server pool with $B_T$ tasks and at the same instant a task can leave from the same server pool, although this has probability zero; here $s_n(\nu_i)$ remains zero but $\calA_n(\nu_i - e_1)$ and $\calD_n(\nu_i)$ have a jump, for some $i$ with $\Sigma_i = B_T + 1$. Setting $B_T' = B_T + \Delta$ precludes arrivals to server pools with $B_T'$ tasks since $\ell_n \leq B_T$.}.
	
	Let $u_n^m \defeq \sum_{i = 0}^\infty \nu_i(m)s_n(\nu_i)$ denote the total number of tasks in the system in phase $m$, and observe that the following equations hold for all $t \in [0, T]$.
	\begin{subequations}
		\begin{align}
			&u_n^1(t) = u_n^1(0) + \frac{1}{n}\calN_n^\lambda(t) - \sum_{i = 0}^\infty \calP_n(t, \nu_i, 1) - \sum_{i = 0}^\infty \calD_n(t, \nu_i, 1), \label{seq: number of tasks in phase 1} \\
			&u_n^m(t) = u_n^m(0) + \sum_{i = 0}^\infty \left[\calP_n(t, \nu_i, m - 1) - \calP_n(t, \nu_i, m)\right] - \sum_{i = 0}^\infty \calD_n(t, \nu_i, m) \quad \text{if} \quad m \neq 1. \label{seq: number of tasks in phase k}
		\end{align}
	\end{subequations}
	These equations follow from \eqref{seq: occupancy state coxian case}. For example, the second term on the right-hand side of \eqref{seq: number of tasks in phase k} follows from the following arithmetic computations:
	\begin{align*}
		\sum_{i = 0}^\infty \nu_i(m)\bar{\calP}_n(\nu_i) &= \sum_{j = 1}^{r - 1}  \sum_{i = 0}^\infty \nu_i(m) \left[\calP_n(\nu_i +e_j - e_{j + 1}, j) - \calP_n(\nu_i, j)\right] \\
		&= \sum_{i = 0}^\infty \nu_i(m)\left[\calP_n(\nu_i +e_{m - 1} - e_m, m - 1) - \calP_n(\nu_i, m - 1)\right] \\
		&+ \sum_{i = 0}^\infty \nu_i(m)\left[\calP_n(\nu_i +e_m - e_{m + 1}, m) - \calP_n(\nu_i, m)\right] \\
		&= \sum_{\nu(m - 1) \geq 1} [\nu(m) + 1]\calP_n(\nu, m - 1) - \sum_{\nu \in \N^r} \nu(m)\calP_n(\nu, m - 1) \\
		&+ \sum_{\nu(m) \geq 1} [\nu(m) - 1]\calP_n(\nu, m) - \sum_{\nu \in \N^r} \nu(m)\calP_n(\nu, m) \\
		&= \sum_{i = 0}^\infty \left[\calP_n(t, \nu_i, m - 1) - \calP_n(t, \nu_i, m)\right].
	\end{align*}
	For the second and last identities, note that $\calP_n(\nu, j)$ is identically zero if $\nu(j) = 0$.
	
	For all sufficiently large $n$, the $i^{\text{th}}$ term in each of the summations appearing in \eqref{seq: number of tasks in phase 1} and \eqref{seq: number of tasks in phase k} is zero if $\Sigma_i > B_T'$, a condition that is met by all but finitely many $i$. Therefore, as in the proof of Proposition \ref{prop: total number of tasks}, it follows from \eqref{seq: convergence of initial conditions 2 coxian case}, \eqref{seq: fslln for arrivals 2 coxian case} and \eqref{seq: fslln for departures 2 coxian case} that every increasing sequence of natural numbers has a subsequence $\calK$ such that $\set{(u_k^1, \dots, u_k^r)}{k \in \calK}$ converges uniformly over $[0, T]$ to a function $(u^1, \dots, u^r)$ that solves
	\begin{align*}
		&\dot{u}^1 = \lambda - \mu_1u^1, \\
		&\dot{u}^m = p_{m - 1}\mu_{m - 1}u^{m - 1} - \mu_m u^m \quad \text{if} \quad m \neq 1.
	\end{align*}
	This system of differential equations has a unique solution with initial conditions $u^m(0)$ defined as in the statement of the proposition. Since $u_n^m(0)$ converges to $u^m(0)$ for all $m$ by \eqref{seq: convergence of initial conditions 2 coxian case}, it follows that $\set{(u_n^1, \dots, u_n^r)}{n \geq 1}$ converges uniformly over $[0, T]$ to this unique solution.
	
	The above arguments imply that $\set{u_n}{n \geq 1}$ converges uniformly over $[0, T]$ on $\Omega_T$. By the weak law of large numbers of an infinite-server system, $\set{u_n}{n \geq 1}$ converges weakly to the deterministic process defined by \eqref{eq: general total number of tasks}. Since $\Omega_T$ has probability one, it follows that $\set{u_n(\omega)}{n \geq 1}$ converges uniformly over $[0, T]$ to the function in \eqref{eq: general total number of tasks} for all $\omega \in \Omega_T$.
	\end{proof}
	
\end{appendices}
	
%% References
\newcommand{\noop}[1]{}
\bibliographystyle{IEEEtranS}
\bibliography{IEEEabrv,bibliography}

% Generated by IEEEtranS.bst, version: 1.14 (2015/08/26)
\begin{thebibliography}{10}
\providecommand{\url}[1]{#1}
\csname url@samestyle\endcsname
\providecommand{\newblock}{\relax}
\providecommand{\bibinfo}[2]{#2}
\providecommand{\BIBentrySTDinterwordspacing}{\spaceskip=0pt\relax}
\providecommand{\BIBentryALTinterwordstretchfactor}{4}
\providecommand{\BIBentryALTinterwordspacing}{\spaceskip=\fontdimen2\font plus
\BIBentryALTinterwordstretchfactor\fontdimen3\font minus
  \fontdimen4\font\relax}
\providecommand{\BIBforeignlanguage}[2]{{%
\expandafter\ifx\csname l@#1\endcsname\relax
\typeout{** WARNING: IEEEtranS.bst: No hyphenation pattern has been}%
\typeout{** loaded for the language `#1'. Using the pattern for}%
\typeout{** the default language instead.}%
\else
\language=\csname l@#1\endcsname
\fi
#2}}
\providecommand{\BIBdecl}{\relax}
\BIBdecl

\bibitem{aghajani2018pde}
R.~Aghajani, X.~Li, and K.~Ramanan, ``The {PDE} method for the analysis of
  randomized load balancing networks,'' \emph{Proceedings of the ACM on
  Measurement and Analysis of Computing Systems}, vol.~1, no.~2, pp. 1--28,
  2017.

\bibitem{aghajani2019hydrodynamic}
R.~Aghajani and K.~Ramanan, ``The hydrodynamic limit of a randomized load
  balancing network,'' \emph{The Annals of Applied Probability}, vol.~29,
  no.~4, pp. 2114--2174, 2019.

\bibitem{applegate2015optimal}
D.~Applegate, A.~Archer, V.~Gopalakrishnan, S.~Lee, and K.~K. Ramakrishnan,
  ``Optimal content placement for a large-scale {V}o{D} system,'' in
  \emph{Proceedings of the 6th International COnference, Co-NEXT '10}.\hskip
  1em plus 0.5em minus 0.4em\relax ACM, 2010, pp. 1--12.

\bibitem{badonnel2008dynamic}
R.~Badonnel and M.~Burgess, ``Dynamic pull-based load balancing for autonomic
  servers,'' in \emph{NOMS 2008-2008 IEEE Network Operations and Management
  Symposium}.\hskip 1em plus 0.5em minus 0.4em\relax IEEE, 2008, pp. 751--754.

\bibitem{benameur2002quality}
N.~Benameur, S.~B. Fredj, S.~Oueslati-Boulahia, and J.~W. Roberts, ``Quality of
  service and flow level admission control in the internet,'' \emph{Computer
  Networks}, vol.~40, no.~1, pp. 57--71, 2002.

\bibitem{van2018scalable}
M.~{\noop{Boor}}van~der Boor, S.~C. Borst, J.~S.~H. van Leeuwaarden, and
  D.~Mukherjee, ``Scalable load balancing in networked systems: A survey of
  recent advances,'' \emph{SIAM Review}, vol.~64, no.~3, pp. 554--622, 2022.

\bibitem{bramson1998state}
M.~Bramson, ``State space collapse with application to heavy traffic limits for
  multiclass queueing networks,'' \emph{Queueing Systems}, vol.~30, pp.
  89--140, 1998.

\bibitem{bramson2010randomized}
M.~Bramson, Y.~Lu, and B.~Prabhakar, ``Randomized load balancing with general
  service time distributions,'' \emph{ACM SIGMETRICS Performance Evaluation
  Review}, vol.~38, no.~1, pp. 275--286, 2010.

\bibitem{bramson2012asymptotic}
------, ``Asymptotic independence of queues under randomized load balancing,''
  \emph{Queueing Systems}, vol.~71, pp. 247--292, 2012.

\bibitem{cheng2008statistics}
X.~Cheng, C.~Dale, and J.~Liu, ``Statistics and social network of {Y}ou{T}ube
  videos,'' in \emph{2008 16th International Workshop on Quality of
  Service}.\hskip 1em plus 0.5em minus 0.4em\relax IEEE, 2008, pp. 229--238.

\bibitem{ephremides1980simple}
A.~Ephremides, P.~Varaiya, and J.~Walrand, ``A simple dynamic routing
  problem,'' \emph{IEEE transactions on Automatic Control}, vol.~25, no.~4, pp.
  690--693, 1980.

\bibitem{foss2017large}
S.~G. Foss and A.~L. Stolyar, ``Large-scale join-idle-queue system with general
  service times,'' \emph{Journal of Applied Probability}, vol.~54, no.~4, pp.
  995--1007, 2017.

\bibitem{gamarnik2018delay}
D.~Gamarnik, J.~N. Tsitsiklis, and M.~Zubeldia, ``Delay, memory, and messaging
  tradeoffs in distributed service systems,'' \emph{Stochastic Systems},
  vol.~8, no.~1, pp. 45--74, 2018.

\bibitem{gamarnik2020lower}
------, ``A lower bound on the queueing delay in resource constrained load
  balancing,'' \emph{The Annals of Applied Probability}, vol.~30, no.~2, pp.
  870--901, 2020.

\bibitem{goldsztajn2021self}
D.~Goldsztajn, S.~C. Borst, J.~S.~H. van Leeuwaarden, D.~Mukherjee, and P.~A.
  Whiting, ``Self-learning threshold-based load balancing,'' \emph{INFORMS
  Journal on Computing}, vol.~34, no.~1, pp. 39--54, 2022.

\bibitem{goldsztajn2021automatic}
D.~Goldsztajn, A.~Ferragut, and F.~Paganini, ``Automatic cloud instance
  provisioning with quality and efficiency,'' \emph{Performance Evaluation},
  vol. 149-150, p. 102209, 2021.

\bibitem{goldsztajn2018controlling}
D.~Goldsztajn, A.~Ferragut, F.~Paganini, and M.~Jonckheere, ``Controlling the
  number of active instances in a cloud environment,'' \emph{ACM SIGMETRICS
  Performance Evaluation Review}, vol.~45, no.~3, pp. 15--20, 2017.

\bibitem{hu2012qoe}
H.~Hu, X.~Zhu, Y.~Wang, R.~Pan, J.~Zhu, and F.~Bonomi, ``{QoE}-based
  multi-stream scalable video adaptation over wireless networks with proxy,''
  in \emph{2012 IEEE International Conference on Communications (ICC)}.\hskip
  1em plus 0.5em minus 0.4em\relax IEEE, 2012, pp. 7088--7092.

\bibitem{joseph2015optimal}
V.~Joseph, S.~C. Borst, and M.~I. Reiman, ``Optimal rate allocation for video
  streaming in wireless networks with user dynamics,'' \emph{IEEE/ACM
  Transactions on Networking}, vol.~24, no.~2, pp. 820--835, 2016.

\bibitem{kelly2011reversibility}
F.~P. Kelly, \emph{Reversibility and stochastic networks}.\hskip 1em plus 0.5em
  minus 0.4em\relax Cambridge University Press, 2011.

\bibitem{kendall1945advanced}
M.~G. Kendall, \emph{The advanced theory of statistics. Volume I}.\hskip 1em
  plus 0.5em minus 0.4em\relax Charles Griffin and Co., Ltd., 42 Drury Lane,
  London, 1945.

\bibitem{key2004fair}
P.~Key, L.~Massouli{\'e}, A.~Bain, and F.~Kelly, ``Fair internet traffic
  integration: network flow models and analysis,'' \emph{Annales des
  Telecommunications}, vol.~59, no.~11, pp. 1338--1352, 2004.

\bibitem{lu2011join}
Y.~Lu, Q.~Xie, G.~Kliot, A.~Geller, J.~R. Larus, and A.~Greenberg,
  ``{Join-Idle-Queue}: A novel load balancing algorithm for dynamically
  scalable web services,'' \emph{Performance Evaluation}, vol.~68, no.~11, pp.
  1056--1071, 2011.

\bibitem{marshall1979inequalities}
A.~W. Marshall, I.~Olkin, and B.~C. Arnold, \emph{Inequalities: theory of
  majorization and its applications}.\hskip 1em plus 0.5em minus 0.4em\relax
  Academic Press, 1979.

\bibitem{menich1991optimality}
R.~Menich and R.~F. Serfozo, ``Optimality of routing and servicing in dependent
  parallel processing systems,'' \emph{Queueing Systems}, vol.~9, pp. 403--418,
  1991.

\bibitem{mitzenmacher2001power}
M.~Mitzenmacher, ``The power of two choices in randomized load balancing,''
  \emph{IEEE Transactions on Parallel and Distributed Systems}, vol.~12,
  no.~10, pp. 1094--1104, 2001.

\bibitem{mo2000fair}
J.~Mo and J.~Walrand, ``Fair end-to-end window-based congestion control,''
  \emph{IEEE/ACM Transactions on Networking}, vol.~8, no.~5, pp. 556--567,
  2000.

\bibitem{mukherjee2016universality}
D.~Mukherjee, S.~C. Borst, J.~S.~H. van Leeuwaarden, and P.~A. Whiting,
  ``Universality of load balancing schemes on the diffusion scale,''
  \emph{Journal of Applied Probability}, vol.~53, no.~4, pp. 1111--1124, 2016.

\bibitem{mukherjee2020asymptotic}
------, ``Asymptotic optimality of power-of-$d$ load balancing in large-scale
  systems,'' \emph{Mathematics of Operations Research}, vol.~45, no.~4, pp.
  1535--1571, 2020.

\bibitem{mukherjee2017optimal}
D.~Mukherjee, S.~Dhara, S.~C. Borst, and J.~S.~H. van Leeuwaarden, ``Optimal
  service elasticity in large-scale distributed systems,'' \emph{Proceedings of
  the ACM on Measurement and Analysis of Computing Systems}, vol.~1, no.~1, pp.
  1--28, 2017.

\bibitem{mukherjee2019join}
D.~Mukherjee and A.~L. Stolyar, ``Join idle queue with service elasticity:
  Large-scale asymptotics of a nonmonotone system,'' \emph{Stochastic Systems},
  vol.~9, no.~4, pp. 338--358, 2019.

\bibitem{niu2011understanding}
D.~Niu, B.~Li, and S.~Zhao, ``Understanding demand volatility in large {V}o{D}
  systems,'' in \emph{Proceedings of the 21st International Workshop on Network
  and Operating Systems Support for Digital Audio and Video}, 2011, pp. 39--44.

\bibitem{niu2011demand}
D.~Niu, Z.~Liu, B.~Li, and S.~Zhao, ``Demand forecast and performance
  prediction in peer-assisted on-demand streaming systems,'' in \emph{2011
  Proceedings IEEE INFOCOM}.\hskip 1em plus 0.5em minus 0.4em\relax IEEE, 2011,
  pp. 421--425.

\bibitem{schassberger2013warteschlangen}
R.~Schassberger, \emph{Warteschlangen}.\hskip 1em plus 0.5em minus 0.4em\relax
  Springer Vienna, 1973.

\bibitem{sparaggis1993extremal}
P.~D. Sparaggis, D.~Towsley, and C.~Cassandras, ``Extremal properties of the
  shortest/longest non-full queue policies in finite-capacity systems with
  state-dependent service rates,'' \emph{Journal of Applied Probability},
  vol.~30, pp. 223--236, 1993.

\bibitem{srikant2013communication}
R.~Srikant and L.~Ying, \emph{Communication networks: an optimization, control,
  and stochastic networks perspective}.\hskip 1em plus 0.5em minus 0.4em\relax
  Cambridge University Press, 2013.

\bibitem{stolyar2015pull}
A.~L. Stolyar, ``Pull-based load distribution in large-scale heterogeneous
  service systems,'' \emph{Queueing Systems}, vol.~80, pp. 341--361, 2015.

\bibitem{stolyar2017pull}
------, ``Pull-based load distribution among heterogeneous parallel servers:
  the case of multiple routers,'' \emph{Queueing Systems}, vol.~85, pp. 31--65,
  2017.

\bibitem{vasantam2017mean}
T.~Vasantam, A.~Mukhopadhyay, and R.~R. Mazumdar, ``Mean-field analysis of loss
  models with mixed-{E}rlang distributions under power-of-$d$ routing,'' in
  \emph{2017 29th International Teletraffic Congress}.\hskip 1em plus 0.5em
  minus 0.4em\relax IEEE, 2017, pp. 250--258.

\bibitem{vvedenskaya1996queueing}
N.~D. Vvedenskaya, R.~L. Dobrushin, and F.~I. Karpelevich, ``Queueing system
  with selection of the shortest of two queues: An asymptotic approach,''
  \emph{Problemy Peredachi Informatsii}, vol.~32, no.~1, pp. 20--34, 1996.

\bibitem{winston1977optimality}
W.~Winston, ``Optimality of the shortest line discipline,'' \emph{Journal of
  Applied Probability}, vol.~14, pp. 181--189, 1977.

\end{thebibliography}
	
\end{document}